\documentclass[a4paper,oneside,11pt]{article}
\usepackage{enumitem}
\usepackage{subfigure}
\usepackage{verbatim}
\usepackage{amssymb}
\usepackage{amsmath} 
\usepackage{amsthm}
\usepackage{color}
\usepackage[dvips]{graphicx}
\usepackage{fullpage}
\usepackage{color}
\usepackage{wasysym} 
\usepackage{endnotes}
\numberwithin{equation}{section}
\newtheorem{theorem}{Theorem}
\numberwithin{theorem}{section}
\newtheorem{lemma}[theorem]{Lemma}
\newtheorem{fact}[theorem]{Fact}
\theoremstyle{remark}
\newtheorem{AuxiliaryCl}{Claim}[theorem]
\theoremstyle{plain}

\newcounter{lemmafirststep}
\newcounter{lemmasecondstep}
\newcounter{lematko66}
\newcounter{lematko67}
\newcounter{lematko68}
\newcounter{lematko69}
\newcounter{lematko610}
\newcounter{lematko511}

\newtheorem{conjecture}[theorem]{Conjecture}
\newtheorem{definition}[theorem]{Definition}

\newtheorem{setting}[theorem]{Setting}
\newcounter{claim71counter}
\newtheorem{Claim71}[claim71counter]{Claim}
\def\wdeg{\mathrm{d\overline{eg}}}
\def\density{\mathrm{d}}
\def\neighbor{\mathrm{N}}
\def\gap{\mathrm{gap}}
\def\disc{\mathrm{disc}}

\def\ci{\mathrm{ci}}
\def\children{\mathrm{Ch}}
\def\parent{\mathrm{Par}}

\def\subset{\subseteq}
\def\dcup{\dot\cup} 
\def\XXX{\mathcal O}
\renewcommand{\leftrightarrow}{-} 
\renewcommand{\epsilon}{\varepsilon}
\newcommand{\By}[2]{\overset{\mbox{\tiny{#1}}}{#2}}
\newcommand{\geBy}[1]{    \By{#1}{\ge} }
\newcommand{\leBy}[1]{    \By{#1}{\le} }
\newcommand{\lBy}[1]{    \By{#1}{<} }
\newcommand{\Referee}[1]{}
\newcommand{\RefereeX}[2]{}
\usepackage{fancyhdr} 
\def\alphaX{\zeta} 
\def\betaX{\alpha} 
\def\gammaX{\gamma} 
\def\epsilonX{\beta} 
\def\etaX{\vartheta} 
\def\omegaX{\kappa} 
\def\sigmaX{\lambda} 
\def\overV{V_*}
\def\overL{L_*}
\def\overN{N_*}
\def\overcalV{\mathcal{V}_*}

\begin{document}

\author{Jan Hladk\'y\thanks{Institute of Mathematics, Czech Academy of Science. \v Zitn\'a 25, 110 00, Praha, Czech Republic. The Institute of Mathematics of the Academy of Sciences of the Czech Republic is supported by RVO:67985840. E-mail: {\tt
honzahladky@gmail.com}.}
\and 
Diana Piguet\thanks{Institute of Computer Science, Czech Academy of Sciences, Pod Vod\'arenskou v\v e\v z\'i 2, 182~07 Prague, Czech Republic. With institutional support RVO:67985807.}}
\date{}
\title{Loebl--Koml\'os--S\'os Conjecture: dense case}
\maketitle
\begin{abstract}
We prove a version of the Loebl--Koml\'os--S\'os Conjecture for dense graphs. For each
$q>0$ there exists a number $n_0\in
\mathbb{N}$ such that for each $n>n_0$ and $k>qn$ the following
holds: if $G$ is a graph of order $n$ with at least $\tfrac n2$ vertices
of degree at least $k$, then each tree of order $k+1$ is a subgraph of~$G$.
\end{abstract}

\section{Introduction}
Embedding problems play a central role in Graph Theory. A variety
of graph embeddings (subgraphs, minors, subdivisions, immersions, etc) have been studied extensively. A graph (finite, undirected, loopless, simple; here as well as in
the rest of the paper) $H$ {\em embeds} in a graph $G$ if there exists
an injective mapping $\phi: V(H)\rightarrow V(G)$ which preserves
the edges of $H$, i.\,e.,
$\phi(x)\phi(y)\in E(G)$ for every edge $xy\in E(H)$. As a synonym we say that $G$ {\em
contains} $H$ ({\em as a subgraph}) and write $H\subset G$. Let
$\mathcal{H}$ be a family of graphs. The graph $G$ is {\em $\mathcal{H}$-universal} if it contains
 every graph from $\mathcal{H}$. This fact is denoted by $\mathcal{H}\subset G$.

In this paper we investigate embeddings of trees. This topic has received considerable attention during the last 40 years.
The class $\mathcal{T}_\ell$ consists of all trees of order $\ell$. One can ask which properties force a graph $G$ to be $\mathcal{T}_\ell$-universal. One sufficient condition for $\mathcal{T}_\ell$-universality can be given in terms of minimum degree. 
\begin{fact}\label{fact:greedyBUDA}
If a graph $G$ has the minimum degree $\delta(G)\ge k$ then $\mathcal{T}_{k+1}\subset G$.\Referee{(1)}
\end{fact}
To prove Fact~\ref{fact:greedyBUDA} it suffices to embed a given tree $T\in\mathcal{T}_{k+1}$ greedily in the host graph $G$. Loebl, Koml\'os and S\'os conjectured (see~\cite{EFLS95}) that the minimum degree condition can be relaxed to a median degree one.
\begin{conjecture}[LKS Conjecture]
Let $G$ be a graph of order $n$. If at least $\tfrac{n}2$ of the vertices
of~$G$ have degree at least $k$, then $\mathcal{T}_{k+1}\subset G$.
\end{conjecture}

The bound on $k$ of the minimal degree of large degree vertices cannot be
decreased. Indeed, if $G$ is a graph with maximum degree $k-1$, then it does not contain a star
$K_{1,k}$. The graph shown in Figure~\ref{fig:ExtremalGraph} shows that the requirement on the number of large degree vertices cannot be relaxed substantially below $\frac n2$. 
\begin{figure}[t]
\centering 
\includegraphics[scale=0.7]{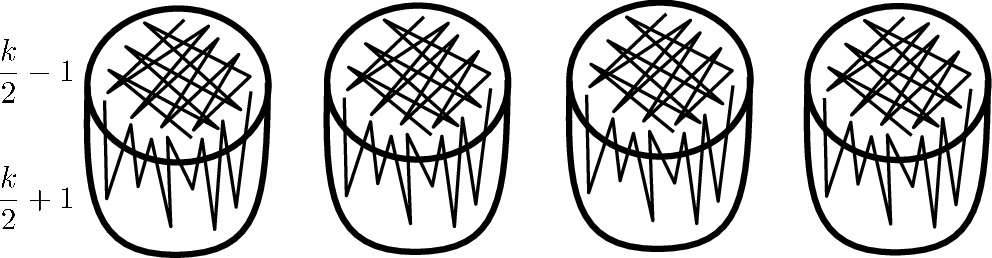}
\caption{A graph with almost half of its vertices of degree $k$ which does not contain a path of length~$k$.}
\label{fig:ExtremalGraph}
\end{figure}
See~\cite{Z07+} and~\cite{HladkyMSC} for further discussions.

There have been several partial results concerning the
LKS~Conjecture. In~\cite{BLW00}, Bazgan, Li and Wo{\'z}niak proved the
conjecture for paths. Piguet and Stein~\cite{PS2} proved that the
LKS~Conjecture is true when restricted to the class of trees of
diameter at most 5, improving upon results of Barr and Johansson~\cite{Barr} and of Sun~\cite{Sun07}.
There are several results proving the LKS~Conjecture under
additional assumptions on the host graph.
Soffer~\cite{Sof00} showed that the conjecture is true if the
host graph has girth at least 7. Dobson~\cite{Dob02} proved the
conjecture when the complement of the host graph does not contain a~$K_{2,3}$.

A special case of the LKS
Conjecture is when $k=\tfrac{n}2$.
This is often referred to as the
($\tfrac{n}2$--$\tfrac{n}2$--$\tfrac{n}2$)~Conjecture,
or Loebl's Conjecture.
 Zhao~\cite{Z07+} proved the conjecture
 for large
graphs.
\begin{theorem}\label{thm_Zhao}
There exists a number $n_0\in\mathbb N$ such that if a graph
$G$ of order $n>n_0$ has at least $\tfrac{n}2$ of the vertices of
degrees at least $\tfrac{n}{2}$, then $\mathcal{T}_{\lfloor \tfrac{n}2
\rfloor+1}\subset G$.
\end{theorem}
An approximate version of the LKS~Conjecture for dense graphs was proven by Piguet
and Stein~\cite{PS07+}.
\begin{theorem}\label{thm_PiguetStein}
For each $q,\epsilon>0$  there exists a number $n_0$ such that for each $n>n_0$ and $k>qn$ the following holds. If  $G$  is a graph of order $n$ with at
least $\tfrac n2$ vertices of degree at least $(1+\epsilon)k$, then
$\mathcal{T}_{k+1}\subset G$.
\end{theorem}
In this paper we strengthen Theorem~\ref{thm_PiguetStein} by
removing the $\epsilon$ term.
\begin{theorem}[Main Theorem]\label{thm_main}
For each $q>0$ there exists  a number $n_0=n_0(q)\in\mathbb N$ such that for each
$n>n_0$ and $k>qn$ the following holds. If $G$ is a graph of order $n$ with at least $\tfrac{n}2$ vertices of degree at least $k$, then $\mathcal{T}_{k+1}\subset G$.
\end{theorem}

\smallskip
We can see from our proof of Theorem~\ref{thm_main} that the requirement on
the number of vertices of large
degree can be relaxed in the case when $\tfrac{n}k$ is far from being an integer.
\begin{theorem}\label{thm_mainstronger}
For each $q_2>q_1>0$ such that the interval $[\tfrac{1}{q_2},\tfrac{1}{q_1}]$ does not
contain any integer, there exist $\varepsilon=\varepsilon(q_1,q_2)>0$ and $n_0$ such that for each $n>n_0$ and $k\in(q_1n,q_2n)$ the following
holds: if $G$ is a graph of order $n$ with at least $(\tfrac12-\varepsilon)n$
vertices
of degree at least $k$, then $\mathcal{T}_{k+1}\subset G$.
\end{theorem}
In the paper, we explicitly prove only Theorem~\ref{thm_main}. In
Section~\ref{sec_outline} we sketch how the proof method can be revised to give
Theorem~\ref{thm_mainstronger}. However, determining the
optimal value of $\varepsilon(q_1,q_2)$ remains open.
Note also that Theorem~\ref{thm_main} has
slightly weaker assumptions
on $G$ than Theorem~\ref{thm_Zhao} when reduced to the case $k=\lfloor \tfrac{n}{2}\rfloor$ --- when $n$ is odd, the requirement on degrees of  large vertices in Theorem~\ref{thm_main} is smaller by one compared to Theorem~\ref{thm_Zhao}.
\bigskip

The property which is considered in the LKS conjecture is given in terms of the median degree. If
we consider the average degree instead we obtain a famous
conjecture of Erd\H os and S\'os which dates back to~1963. 
\begin{conjecture}[{ES Conjecture,~\cite[p.30]{Erdos:ExtremalProblems}}] Let $G$ be a graph of order $n$ with more than $\tfrac12(k-2)n$ edges. Then $\mathcal{T}_k \subset G$.\Referee{(2)}
\end{conjecture}
If true, the ES~Conjecture is sharp. After several partial results on
the problem, a breakthrough was achieved by Ajtai, Koml\'os,
Simonovits and Szemer\'edi, who announced a proof of the Erd\H os--S\'os Conjecture for large $k$.
\begin{theorem}\label{thm_ES}
There exists a number $k_0$ such that for each $k>k_0$ the following holds:
if a graph $G$ of order $n$ has more than $\tfrac12(k-2)n$ edges, then 
$\mathcal{T}_k \subset G$.
\end{theorem}
A version of Theorem~\ref{thm_ES} for $k$ linear in $n$ could be obtained by
an application of the Regularity Lemma; such a theorem would be a counterpart to Theorem~\ref{thm_main}. The proof of Theorem~\ref{thm_ES} by Ajtai et al.\ uses a decomposition technique which substantially generalizes the Regularity Lemma, and which is applicable even
to sparse graphs. Hladk\'y, Koml\'os, Piguet, Simonovits, Stein, and Szemer\'edi~\cite{LKSsparse1,LKSsparse2,LKSsparse3,LKSsparse4} used this decomposition technique to prove an approximate version of the LKS~Conjecture (see also~\cite{LKSsparseOverview} for a high-level overview of the proof). 
\begin{theorem}\label{thm:LKSapproxsparse}
For each $\epsilon>0$ there exists a number $k_0$ such that for each $k>k_0$ the following holds. If  $G$ is a graph of order $n$ with at least $(\tfrac12+\epsilon)n$ vertices of degrees at least $(1+\epsilon)k$, then $\mathcal{T}_{k+1}\subset G$.
\end{theorem}
We believe that the techniques developed for Theorem~\ref{thm_main} and for Theorem~\ref{thm:LKSapproxsparse} can be utilized to proving the LKS Conjecture for $k$ sufficiently large.

\bigskip

The current work builds on techniques of Zhao~\cite{Z07+} and of
Piguet and Stein~\cite{PS07+}. We postpone a detailed discussion of
similarities between our approach and theirs and of our own contribution until Section~\ref{sec_outline}.
After the first version of this manuscript was posted on the arXiv, Oliver Cooley~\cite{Cooley} published an independent proof of Theorem~\ref{thm_main}.\Referee{(3)}

\subsection{Ramsey number of trees}
In this section we show the connection between the LKS Conjecture and the Ramsey number of trees. For two graphs~$F$ and~$H$ we write $R(F,H)$ for the {\em Ramsey
number} of the graphs~$F$ and~$H$. This is the smallest number $m$ such
that in each red/blue edge-coloring of $K_m$ there is a red copy
of $F$ or a blue copy of $H$. For two families of
graphs $\mathcal{F}$ and $\mathcal{H}$ the Ramsey number
$R(\mathcal{F},\mathcal{H})$ is the smallest number $m$ such that in
each red/blue edge-coloring of $K_m$ the graph induced by the red edges is $\mathcal{F}$-universal, or the graph induced by the blue edges is $\mathcal{H}$-universal. Theorem~\ref{thm_main} implies an almost tight upper bound (up to an additive error of one) on the Ramsey number of pairs of families of trees of similar orders. This partially answers a question of Erd\H os, F\"uredi, Loebl and S\'os~\cite{EFLS95}.
For a fixed real $p\in (0,\tfrac{1}{2})$ consider two natural numbers $\ell_1$
and $\ell_2$ such that \Referee{(4)}
\begin{equation}\label{Amalka}
n_0<\ell_1\le \ell_2<\tfrac{\ell_1}{p}\;,
\end{equation}
where $n_0=n_0(\tfrac{p}{2})$ comes from Theorem~\ref{thm_main}. Consider any red/blue edge-coloring of the
graph $K_{\ell_1+\ell_2}$. We color a vertex $v\in
V(K_{\ell_1+\ell_2})$ red if it incident with at least $\ell_1$ red
edges, and blue otherwise (in which case it is incident with at least $\ell_2$ blue edges).\RefereeX{(5)}{Slightly different wording} Thus at least
half of the vertices of $K_{\ell_1+\ell_2}$ have the same color. Applying Theorem~\ref{thm_main} to the graph whose edges are induced
by this color, we
conclude that $R(\mathcal{T}_{\ell_1+1},\mathcal{T}_{\ell_2+1})\le
\ell_1+\ell_2$.

For the lower bound, first consider the case when at least one of
$\ell_1$ and $\ell_2$ is odd.\RefereeX{(6)}{No, it is not needed to know which one. The fact that the red degree of vertices is $\ell_1-1$ is not connected to which of $\ell_1$ and $\ell_2$ is odd.} It is a well-known fact that there
exists a red/blue edge-coloring of $K_{\ell_1+\ell_2-1}$ such that
the red degree of every vertex is $\ell_1-1$. Neither a red copy of
$K_{1,\ell_1}$ nor a blue copy of $K_{1,\ell_2}$ is contained in
$K_{\ell_1+\ell_2-1}$ with this coloring. Thus
$R(\mathcal{T}_{\ell_1+1},\mathcal{T}_{\ell_2+1})> \ell_1+\ell_2-1$. A
construction in a similar spirit shows that
$R(\mathcal{T}_{\ell_1+1},\mathcal{T}_{\ell_2+1})> \ell_1+\ell_2-2$, if
both $\ell_1$ and $\ell_2$ are even.\Referee{(7)} Under the assumptions given by~\eqref{Amalka} we thus have
\begin{align}
R(\mathcal{T}_{\ell_1+1},\mathcal{T}_{\ell_2+1})&=\ell_1+\ell_2\; ,\quad\mbox{if
$\ell_1$ is odd or $\ell_2$ is odd, and}\\
\ell_1+\ell_2-1\le
R(\mathcal{T}_{\ell_1+1},\mathcal{T}_{\ell_2+1})&\le\ell_1+\ell_2\; ,\quad\mbox{
otherwise.} \label{eq_RamseyNotPrecise}
\end{align}
The ES~Conjecture, if true, shows \RefereeX{(8)}{Reformulated differently than suggested by the referee.} that the lower
bound in~\eqref{eq_RamseyNotPrecise} is attained.

Ramsey numbers of several other classes of trees have been investigated; the
reader is referred to a survey of Burr~\cite{Burr74} and to
newer results in~\cite{EFRS82,GHK79,HLT02}.

\section{Outline of the proof}\label{sec_outline}
We iterate the following procedure in steps $i=1,2,3,\ldots$. At the beginning of step $i$ we are given sets $V_1,\ldots, V_{i-1}$ that were obtained in previous steps. We then find a set $Q\subseteq V(G)\setminus \bigcup_{j< i}V_j$ such that at least about a half of the vertices in $Q$ are \emph{large} (i.\,e., of degree at least $k$). Furthermore, the set $Q$ is almost isolated from the rest of the graph. Using the Regularity Lemma, we try to embed $T$ in $Q$. If we do not succeed, then we can extract from $Q$ a subset $V_{i}\subseteq  Q$ of size approximately $k$ which is nearly isolated from the rest of the graph, and for which at least half of the vertices are large.
If we cannot embed $T$ in any of the iterating steps (i.\,e., $V(G)\setminus \bigcup_iV_i\cong\emptyset$), we obtain a particular configuration of the graph $G$, called the {\em Extremal Configuration}. The structure of $G$ is then very similar to that depicted in Figure~\ref{fig:ExtremalGraph}. In this case, we prove that $T\subseteq G$, without the use of the Regularity Lemma.

In the remainder of the overview, we explain in more detail the proof of the part using the Regularity Lemma, as well as the part when $G$ is in the Extremal configuration.

\paragraph{The Regularity Lemma Part.}
Before applying the Regularity Lemma,\Referee{(9)} we first resolve two simple cases. The first one is when $Q$ is close to a bipartite graph with one of its color classes being the large vertices (see Lemma~\ref{prop_SCHolds}). The second case (see Lemma~\ref{prop:UnbalancedCase}) is when the tree $T$ is locally unbalanced (see the definition on page~\pageref{p.:def-unbalanced}). In both cases easy arguments show that $T\subset G$.

In other cases we\Referee{(10)} use the Regularity Lemma on the graph $G$ and obtain a cluster graph $\mathbf{G}$.  We apply a matching lemma (Lemma~\ref{prop_TutteType}) to the subgraph induced by the clusters in $Q$. This lemma guarantees the existence of one of two certain matching structures in $\mathbf{G}$. Each of these structures exposes a matching $M$ in the cluster graph, and two clusters $A$ and $B$ that are adjacent in $\mathbf{G}$ and that have high average degree to the matching $M$. These structures are called Case~I and Case~II. The principle of the embedding is to use the edges of $M$ to embed parts of the tree~$T$ in them, and use the clusters $A$ and $B$ to connect these parts.

\paragraph{The Extremal Case Configuration.} In the Extremal case we are given disjoint sets $V_1,\ldots,V_i\subset V(G)$ such that each of them has size approximately $k$, contains at least nearly $\tfrac{k}{2}$ large vertices, and each set $V_j$ is almost isolated from the rest of the graph.

If the sets $V_1,\ldots,V_i$ exhaust the whole graph $G$, we are able to show
$T\subset G$ as follows. We find a set $V_{i_0}$ so that most of $T$ can be
embedded in $V_{i_0}$. We may need to use a\Referee{(11)} few edges that connect distinct sets $V_j$
and embed some part of $T$\Referee{(12)} outside $V_{i_0}$. The way of finding these
``bridges'' depends on the structure of the tree $T$.

If $V_1,\ldots, V_i$ do not exhaust $G$, the method remains the same. However, it has two possible outcomes. Either we show that $T\subset G$ or we are able to exhibit a set $Q\subseteq V\setminus \bigcup_{j< i}V_j$ with the properties as above allowing the next step of the iteration.

\paragraph{Strengthening of
Theorem~\ref{thm_main} --- Theorem~\ref{thm_mainstronger}.}  The only place where
we use the exact bound on the number of large vertices is the last step of the
Extremal case. That is, the whole vertex set $V(G)$ is decomposed into sets
$V_1,\ldots,V_s$, each of size approximately~$k$.\Referee{(13)} Assume now that $k\in(q_1n,q_2n)$. We have $n=|V_1|+|V_2|+\ldots+|V_s|\approx ks\in (q_1sn,q_2sn)$, yielding that the the interval $(q_1s,q_2 s)$ must contain~1 (or at least to be ``close to~1''). Thus the Extremal case cannot occur when $[\tfrac{1}{q_2},\tfrac{1}{q_1}]\cap
\mathbb{N}=\emptyset$. This suffices to prove Theorem~\ref{thm_mainstronger}.

\paragraph{Relation to previous work.}
The proof of Theorem~\ref{thm_main} is inspired by techniques used to
prove Theorem~\ref{thm_PiguetStein} (\cite{PS07+}) and Theorem~\ref{thm_Zhao}
(\cite{Z07+}). Both these papers build on a seminal paper of Ajtai, Koml\'os and
Szemer\'edi~\cite{AKS95} where an approximate version of the
$(\tfrac{n}{2}$--$\tfrac{n}{2}$--$\tfrac{n}{2})$~Conjecture was\Referee{(14)} proven. In~\cite{AKS95} the basic strategy is outlined. It is worth noting that even though~\cite{AKS95} addresses explicitly only the   $(\tfrac{n}{2}$--$\tfrac{n}{2}$--$\tfrac{n}{2})$~Conjecture the proof actually yields Theorem~\ref{thm_PiguetStein} in the regime $\tfrac{k}{n}\ge\tfrac{1}{2}$. As in the proof overview above, the key step is a certain matching lemma applied to the cluster graph of the host graph.

The key ingredient in~\cite{Z07+} was to identify --- using the approach of Ajtai, Koml\'os and Szemer\'edi combined with the Stability method of Simonovits~\cite{S68} --- one extremal case. This extremal case was analysed and resolved by ad-hoc methods.\Referee{(15)} 
The main contribution of~\cite{PS07+} is a more general
matching lemma, which is applicable even when $\tfrac{k}{n}<\tfrac{1}{2}$.
In this paper we further strengthen the matching lemma
from~\cite{PS07+}. The Extremal case is an extensive generalization of the
Extremal case from~\cite{Z07+}.

\paragraph{Algorithmic questions.}
Let us remark that our proof of Theorem~\ref{thm_main} yields a polynomial time algorithm for finding an embedding of each tree $T\in \mathcal{T}_{k+1}$ in $G$, given that $k$ and $G$ satisfy the conditions of Theorem~\ref{thm_main}. Indeed, all the existential results we use (Regularity Lemma, and various matching theorems) are known to have polynomial-time constructive algorithmic counterparts. We omit details. 
\section{Notation and preliminaries}\label{sec:preliminaries}
\renewcommand{\div}{\triangle}
For $n\in\mathbb{N}$ we write $[n]=\{1,2,\ldots,n\}$. 
The symbol
$\div$ means the symmetric difference of two sets.
The function $\ci: \mathbb{R}\rightarrow \mathbb{Z}$ is the {\em closest integer function} defined by $\ci(x)=\lfloor x\rfloor$ if $x-\lfloor x\rfloor<0.5$, and $\ci(x)=\lceil x\rceil$ otherwise.

We use standard graph theory terminology and notation, following Diestel's book~\cite{Die05}. We define here only symbols that are not used there. The order of a graph $H$ and the number of its edges are denoted by $v(H)$ and $e(H)$, respectively. For two vertex sets $X$ and $Y$ we write $E(X,Y)$ for the set of edges with one end-vertex in $X$ and the other in $Y$. We write $e(X,Y)=|E(X,Y)|$ (note that edges inside $X\cap Y$ get counted only once). When $X$ and $Y$ are disjoint, we write $H[X,Y]$ for the bipartite graph they induced. For a vertex $x$ and a vertex set $X$ we define $\deg(x,X)=\deg_X(x)=e(\{x\},X)$. For two sets $X,Y\subset V(H)$ we define the {\em average degree} from $X$ to $Y$ by $\wdeg_H(X,Y)=\tfrac{e(X,Y\setminus X)}{|X|}$. We write $\wdeg_H(X)$ as a short for $\wdeg_H(X,V(H))$. Let $X$ and~$Y$ are arbitrary (not necessarily disjoint vertex sets). We define two variants of the minimum degree: $\delta(X)=\min_{v\in X}\deg(v)$, and $\delta(X,Y)=\min_{v\in X}\deg(v,Y)$. In this case, we may write $H$ in the subscript (e.g.~$\delta_H(X)$)\Referee{(20)} to emphasize which graph we are dealing with. We denote by $\neighbor(x)$ the set of neighbors of the vertex $x$, by $\neighbor_X(x)$ the neighborhood of $x$ restricted to a set $X$, i.\,e., $\neighbor_X(x)=\neighbor(x)\cap X$, and by $\neighbor(X)$ the set of all vertices in $H$ which are adjacent to at least one vertex from $X$, i.\,e., $\neighbor(X)=\bigcup_{v\in X}\neighbor(v)$.

Let $P=v_1v_2\ldots v_\ell$ be a path. For arbitrary sets of
vertices $X_1, X_2, \ldots, X_\ell$ we say that $P$ is an {\em
$X_1\leftrightarrow X_2\leftrightarrow \ldots \leftrightarrow
X_\ell$-path} if $v_i\in X_i$ for every $i\in [\ell]$. An edge $xy$ is an $X\leftrightarrow Y$ edge if $x\in X$ and $y\in Y$ and a matching $M$ is an $X\leftrightarrow Y$ matching if its every edge is an $X\leftrightarrow Y$ edge.

A pair $(H,\omega)$ is a {\em weighted graph} if $H$ is a graph
and $\omega:E(H)\rightarrow (0,+\infty)$ is a weight function. For two sets $X,Y\subset V(H)$ the {\em weight of the edges crossing from $X$ to $Y$} is defined by $\omega(X,Y)=\sum_{xy\in E(X,Y)}\omega(xy)$.
Denote also by $\omega$ the weighted degree, $\omega(v)=\sum_{u\in V(H),
vu\in E(H)}\omega(vu)$. For a vertex $v$ and a vertex set $X$ we define
$\omega(v,X)$ analogously to $\deg(v,X)$.

We omit rounding symbols when this does not effect the correctness
of calculations.

\subsection{Trees} Let $T$ be a rooted tree with a root $r\in V(T)$. We define a partial
order
$\preceq$ on $V(T)$ by saying that $a\preceq b$ if and only if the
vertex $b$ lies on the (unique) path connecting $a$ with $r$.
If $a\preceq b$ and $a\neq b$ we say that~$a$ {\em is below} $b$. A vertex~$a$ is a
{\em child of} $b$ if $a\preceq b$ and $ab\in E(T)$. The vertex~$b$ is then the
{\em parent of} $a$. $\children(b)$ denotes the set of children of $b$. The parent of a vertex~$a$ is denoted $\parent(a)$ (note that $\parent(a)$ is undefined if
$a=r$). We extend the definitions of $\children(\cdot)$ and
$\parent(\cdot)$ to an arbitrary set $U\subset V(T)$ by
$\parent(U)=\bigcup_{u\in U}\parent(u)$ and
$\children(U)=\bigcup_{u\in U}\children(u)$.
We say that a tree
$T_1\subset T$ is {\em induced} by a vertex $x\in V(T)$ if
$V(T_1)=\{v\in V(T)\: :\: v \preceq x\}$ and we write
$T_1=T(r,\downarrow x)$, or if the root is obvious from the
context $T_1=T(\downarrow x)$.
Subtrees induced by a vertex are called {\em end subtrees}. Other subtrees are called {\em internal subtrees}.
A subtree~$T_0$ of $T$ is
a {\em full-subtree}, if there exists  a vertex $y\in V(T)$ and a set
$C\subset \children(y)$, $C\not= \emptyset$ such that
$T_0=T[\{y\}\cup\bigcup_{b\in C} \{v\: :\: v\preceq b\}]$. \emph{Internal} vertices are simply non-leaf vertices.

We will want to find a full-subtree in such a way that we have some control over its order or over its number of leaves. To this end we will use the following fact.
\begin{fact}[{\cite[Fact~7.9]{Z07+}}]\label{fact_fullsubtrees}
Let $(T,r)$ be a rooted tree of order $m$ with $\ell$ leaves.
\begin{enumerate}[label={$(\roman{*})$}]
\item\label{tesco1} For each integer $m_0$, $0< m_0\le m$, there exists a
full-subtree $T_0$ of $T$ of order $\tilde{m}\in[\tfrac{m_0}2,m_0]$.
\item\label{tesco2} For each integer $\ell_0$, $0< \ell_0\le \ell$, there exists a
full-subtree $T_0$ of $T$ with $\tilde{\ell}$ proper leaves (i.e.\ leaves of $T$), where
$\tilde{\ell}\in[\tfrac{\ell_0}2,\ell_0]$. 
\end{enumerate}
\end{fact}

For each tree $F$ we write $F_\oplus$ and $F_\ominus$ for the
vertices of its two color classes with $F_\oplus$ being the
larger one. We define the {\em gap} of the tree $F$ as
$\gap(F)=|F_\oplus|-|F_\ominus|$. For a tree $F$,
a partition of its vertices into sets $U_1$ and $U_2$ is called
{\em semi-independent}\Referee{(16)} if $|U_1|\le|U_2|$ and $U_2$ is an independent set.
Furthermore, the {\em discrepancy} of  $(U_1,U_2)$ is
$\disc(U_1,U_2)=|U_2|-|U_1|$ and the discrepancy of $F$ is defined as
$$\disc(F)=\max\{\disc(U_1,U_2)\: :\: \mbox{$(U_1,U_2)$ is semi-independent} \}
\;
\mbox{.}$$
Clearly, $\gap(F) \le \disc(F)$.

The next three facts relate discrepancy to other properties of trees.
\begin{fact}[{\cite[Fact~6.9]{Z07+}}]\label{fact_ManyLeavesInUnbalanced}
Let $(U_1,U_2)$ be a semi-independent partition of a tree $T$ of order\Referee{(17)}
$v(T)>1$. Then $U_2$ contains at least $|U_2|-|U_1|+1$ leaves.
\end{fact}
\begin{fact}\label{fact_cutnodiscrepancy}
Let $r$ be a vertex of a tree $F$, and let $(U_1,U_2)$ be any semi-independent partition of $F$. Let $\mathcal{K}$ be a subset of the components of the forest $F-\{r\}$ and let~$V(\mathcal{K})$ denote all the vertices contained in the components of~$\mathcal K$. Then
\begin{enumerate}[label={$(\roman{*})$}]
\item\label{sai1} $\left||V(\mathcal{K})\cap F_\oplus|-|V(\mathcal{K})\cap F_\ominus|\right|\le\disc(F)+1$, and
\item\label{sai2} $|V(\mathcal{K})\cap U_2|-|V(\mathcal{K})\cap U_1|\le\disc(F)+1$.
\end{enumerate}
\end{fact}
\begin{proof} We focus first on~\ref{sai1}.
The statement is obvious when $|V(\mathcal{K})\cap F_\oplus|-|V(\mathcal{K})\cap F_\ominus|=0$. Suppose that $|V(\mathcal{K})\cap F_a|-|V(\mathcal{K})\cap F_b|=\ell> 0$, where $a,b\in\{\oplus,\ominus\}$, $a\not =b$ is a choice of the color classes.
It is enough to exhibit a semi-independent partition $(W_1,W_2)$ of the tree $F$ with $|W_2|-|W_1|\ge \ell-1$. Partition the components of the forest $F-\{r\}$ that are not included in $\mathcal{K}$ into two families $\mathcal{A}$ and $\mathcal{B}$ so that $\mathcal{A}$ contains those components $K\not \in\mathcal{K}$ for which $|V(K)\cap F_a|\ge|V(K)\cap F_b|$.\Referee{(18)} Then the partition below satisfies the requirements.
\begin{align*}
W_1&=\{r\}\cup (V(\mathcal{K})\cap F_b)\cup (V(\mathcal{A})\cap F_b)\cup(V(\mathcal{B})\cap F_a) \;\mbox{,}\\
W_2&=(V(\mathcal{K})\cap F_a)\cup (V(\mathcal{A})\cap F_a)\cup(V(\mathcal{B})\cap F_b) \;\mbox{.}
\end{align*}

The proof of~\ref{sai2} is similar, and we only sketch it. Again, we shall exhibit a
semi-independent partition~$(W_1,W_2)$ with $|W_2|-|W_1|\ge |V(\mathcal{K})\cap
U_2|-|V(\mathcal{K})\cap U_1|-1$. We put~$r$ into~$W_1$. On
the components of $\mathcal K$ the partition into $W_1$ and $W_2$ is inherited
from the partition $(U_1,U_2)$. Every component $K\not\in\mathcal K$ of $F-\{r\}$ is 
partitioned so that $W_2$ gets the majority color class of~$K$.
\end{proof}
\begin{fact}\label{fact:leavesB}
Suppose that $T$ is a tree with $\disc(T)\le \ell$. Let $V(T)=U_1\dcup U_2$ be a partition such that $U_2$ is independent. Then for the set $X$ of the leaves in $U_1$ that have another leaf-sibling in~$U_1$ we have $|X|\le \ell+|U_1|-|U_2|$.
\end{fact}
\begin{proof}
We have $|X|\ge 2|\parent(X)|$. Thus, if $|X|>\ell+|U_1|-|U_2|$, we consider the partition $$\big((U_1\setminus X)\cup \parent(X)  \;,\; (U_2\setminus \parent(X))\cup X \big)\;.$$
Even\Referee{(19)} though we do not necessarily have $\parent(X)\subset U_2$ this is semi-independent partition of discrepancy at least $|U_2|-|U_1|+2(|X|-|\parent(X)|)>\ell$, a contradiction.
\end{proof}

\subsection{Greedy embeddings} Given a tree $T$ and a graph $H$ there are
several situations when
one can embed $T$ in $H$ greedily. 
The simplest such setting is given in Fact~\ref{fact:greedyBUDA}.
An analogous procedure works if $H$ is bipartite, $H=(V_1,V_2;E)$, and
$ \delta(V_1,V_2)\ge |T_\oplus|, \delta(V_2,V_1)\ge
|T_\ominus|$. The facts stated below generalize the greedy procedure. 
\begin{fact}[{\cite[Fact~7.2(2)]{Z07+}}]\label{fact_embeddingsemiindependent}
Let $(U_1,U_2)$ be a semi-independent partition of a tree $T$. If
there are two disjoint sets of vertices $V_1$ and $V_2$ of a graph $H$
such that
$\min\{\delta(V_1,V_2),\delta(V_1,V_1),\delta(V_2,V_1)\}\ge |U_1|$
and $\delta(V_1)\ge v(T)-1$, then $T\subset H$.
\end{fact}
\begin{fact}[{\cite[Fact~7.2(1)]{Z07+}}]\label{fact:easyEmbedding1}
Suppose that $H$ is a graph with a bipartite subgraph $K=(W_1,W_2;J)$. If $\delta(K)>\tfrac\ell2$ and $\delta_H(W_1)\ge \ell$ then $\mathcal{T}_{\ell+1}\subset H$.
\end{fact}
\begin{fact}\label{fact:easyemb2}
Suppose $H'\subset H$ are two graphs. If $\delta(H')\ge x$ and $\delta_H(V(H'))\ge\ell$, then $F\subset H$ for each tree $F\in\mathcal T_{\ell+1}$ with at least $\ell-x$ leaves.
\end{fact}
\begin{proof}
We first embed the internal vertices of $F$ in $H'$ using the greedy procedure from Fact~\ref{fact:greedyBUDA}. We can then extend this embedding using the high degrees of $V(H')$. 
\end{proof}

The next lemma allows us to embed a tree $T$ into a graph containing a bipartite subgraph $H$ which can \emph{almost} accomodate~$T$. So, additional connecting structures $\mathcal M$, $\mathcal E$ that will allow to divert small parts of $T$ elsewhere are introduced. The main structures assumed in the lemma are shown in~Figure~\ref{fig_prop38}.\Referee{(21)}
\begin{figure}[ht]
  \centering
\includegraphics[scale=0.7]{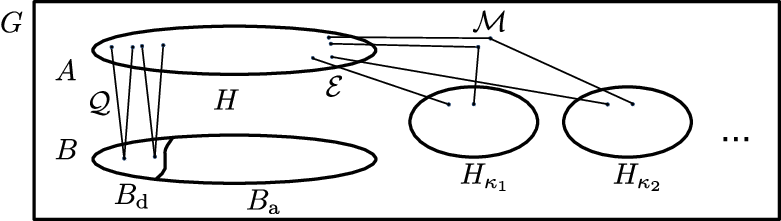}
  \caption{The situation in Lemma~\ref{prop_SE_embeddingwithfewleaves}. Most of the set $T_\ominus$ is embedded in $A$, most of the set $T_\oplus$ will be embedded in $B$. The connections $\mathcal E$ and $\mathcal M$ are used to divert parts of $T$ to the graphs $H_\kappa$.}
  \label{fig_prop38}
\end{figure}
\begin{lemma}\label{prop_SE_embeddingwithfewleaves}
Suppose that $\alpha\in(0,\frac1{10})$ is arbitrary. For each tree $T\in \mathcal{T}_{k+1}$ with less than $\alpha k$ leaves the following
holds. Suppose that a bipartite graph $H=(A,B;E)$ and graphs $\{H_\kappa\}_{\kappa\in I}$ (where $I$ is arbitrary) are pairwise vertex-disjoint subgraphs of a graph~$G$ on vertex set~$V$.\RefereeX{(22)}{No, the distinction between the graphs $H$ and $G$ was correct. The text was modified to make it clearer, and hopefully the new picture helps as well} Suppose that
the following properties are fulfilled.
\begin{enumerate}[label={$(\roman{*})$}]
\item $\delta(H_\kappa)>34\alpha k$ for each $\kappa\in I$.\RefereeX{!}{The original condition was $\delta(H_\kappa)>25\alpha k$ but when revising the proof we saw that this needs to be strengthened. Now changes in later applications of the lemma were required (as this condition is always satisfied with a lot of space).}
\item\label{it:XUN} $\delta_G(A)\ge k$.
\item There exists an $A\leftrightarrow (\bigcup_\kappa(V(H_\kappa)))$-matching
$\mathcal{E}$, and a family $\mathcal{M}$ of pairwise vertex-disjoint
$A\leftrightarrow(V\setminus V(H))\leftrightarrow(\bigcup_\kappa
V(H_\kappa))$ paths. Moreover, $V(\mathcal{E})\cap
V(\mathcal{M})=\emptyset$.
\item \label{it:sEM}$|\mathcal{E}|+|\mathcal{M}|<\alpha k$.
\item \label{it:sA}$|A|+|\mathcal{E}|\ge |T_\ominus|$.
\item \label{it:BEM}$|B|+|\mathcal{E}|+|\mathcal{M}|\ge |T_\oplus|-1$.
\item\label{it:delAB} $\delta(A,B)\ge |B|-\alpha k$.
\item The set $B$ has a decomposition $B=B_\mathrm{a}\dcup B_\mathrm{d}$,
$|B_\mathrm{d}|\le \alpha k$,
$\delta(B_\mathrm{a},A)\ge |A|-\alpha k$, and there exists a family
$\mathcal{Q}$ of $|B_\mathrm{d}|$ pairwise vertex-disjoint
$A\leftrightarrow B_\mathrm{d}\leftrightarrow A$ paths. Moreover,
$V(\mathcal{Q})\cap (V(\mathcal{E})\cup
V(\mathcal{M}))=\emptyset$.
\end{enumerate}
Then, $T\subset G$.
\end{lemma}
The proof is given in the Appendix.

\subsection{Specific notation}
A graph $H$ is said to have the {\em LKS-property} (with parameter $k$) if at least half
of its vertices have degree at least $k$, i.\,e., we have $|L|\ge
\tfrac{v(H)}2$, where $L=\{v\in V(H)\: :\: \deg_H(v)\ge k\}$.

When we refer to $q,n_0,n,k$ or $G$ in the rest of the paper, we
always refer to the objects from the statement of
Theorem~\ref{thm_main}. The vertex set of~$G$ is denoted by $V$. We partition $ V =L\dcup S$, where $L=\{v\in
 V \: :\: \deg(v)\ge k\}$ and $S=\{v\in  V \: :\: \deg(v)<k\}$. We call
the vertices from~$L$ {\em large} and the vertices from $S$ {\em small}. The hypothesis
of Theorem~\ref{thm_main} implies that $|L|\ge \tfrac{n}2$. Finally~$T$
denotes a tree of order $k+1$ that we want to embed in $G$.

We write $\alpha\ll \beta$ to express that $\alpha$ is sufficiently small
compared to $\beta$.

\section{Proof of the Main Theorem (Theorem~\ref{thm_main})}
The proof of Theorem~\ref{thm_main} is based on an iterated application
of Lemma~\ref{prop:EC-obecna} and~\ref{prop:iteration-Reg} below. To state Lemma~\ref{prop:EC-obecna} we need to introduce the notion of $(\beta,\sigma)$-extremality.\RefereeX{(23)}{Lower-cased this, and also ``deficient'' and ``abundant''} The $(\beta,\sigma)$-extremality says that a part of a graph resembles the extremal structure as in Figure~\ref{fig:ExtremalGraph}. For two reals $\beta,\sigma\in(0,1)$, a partition of the vertex set $ V
=V_1\dcup V_2\dcup\ldots \dcup V_\ell\dcup\tilde V$ is {\em $(\beta,\sigma)$-extremal} if  the following conditions are satisfied.
\begin{itemize}
\item $\ell\ge 1\;$.
\item $(1-\beta)k\le |V_i|\le(1+\beta)k$ for each $i\in[\ell]\;$.
\item  $\tilde{V}=\emptyset$ or $|\tilde{V}|>\sigma k\;$.
\item $e(V_i, V \setminus V_i)\le\beta k^2$ for each
$i\in[\ell]\;$, and
$e(\tilde{V},V\setminus \tilde{V})\le\beta k^2\;$.
\item $(\tfrac{1}2-\beta)k\le|V_i\cap L|$ for each $i\in[\ell]\;$.
\item $|\tilde{V}\cap L|\leq(\tfrac{1}2-\sigma)|\tilde{V}|\;$.
\end{itemize}

\medskip Lemma~\ref{prop:EC-obecna} below, which will be proved in Section~\ref{sec_Extremalcase}, deals with a graph that admits an extremal partition.\RefereeX{(24)}{all the propositions were turned into lemmas}
\begin{lemma}\label{prop:EC-obecna}
Given a number $q>0$, there exists a constant $c_{\mathbf E}>0$ such that the following holds.\
For each $\sigma\le c_{\mathbf E}$ there exists a number $\beta \in (0,\sigma)$ such
that if $G$ is a graph satisfying the LKS-property with $k\geq qn$\Referee{(25)} that
admits a $(\beta,\sigma)$-extremal partition $V=V_1\dcup \ldots \dcup V_\ell\dcup \tilde V$,
then $\mathcal T_{k+1}\subseteq G$, or there exists a set $Q\subset \tilde{V}$ such that
\begin{enumerate}[label={$(\roman{*})$}]
\item $|Q|>\tfrac{k}2\;$.\label{pr:velkyQ}
\item $|Q\cap L|>\tfrac{|Q|}2\;$.\label{pr:QL}
\item $e(Q, V \setminus Q)<\sigma k^2\;$.\label{pr:Qizo}
\end{enumerate}
\end{lemma}
 The next statement, which will be proved in
Section~\ref{sec_TheProof}, entails the regularity part of the proof of Theorem~\ref{thm_main}.
\begin{lemma}\label{prop:iteration-Reg}
Given numbers $q, c, \rho>0$ there are numbers $\sigmaX\in (0, \rho)$ and 
$n_0=n_0(q,c,\rho)\in \mathbb N$ such that for each graph $G$ on $n\geq n_0$
vertices satisfying the LKS-property with $k\geq qn$ with a subset $\overV\subseteq V$\RefereeX{(34)}{Changed everywhere} having the following properties
\begin{enumerate}[label={$(\roman{*})$}]
\item $|\overV|> ck\;$,
\item\label{pr:Vrtacka} $e(\overV,V\setminus \overV)\le \sigmaX k^2\;$, and
\item $|L\cap \overV|\ge \tfrac{1}2(1-\sigmaX)|\overV|\;$,
\end{enumerate}
there exists a subset $V'\subseteq \overV$ such that
\begin{itemize}
\item [$\diamond$]$(1-\rho)k\le |V'|\le (1+\rho)k\;$,
\item [$\diamond$]$|V'\cap L|\geq \tfrac{1}2|V'|\;$, and
\item [$\diamond$]$e(V',V\setminus V')\leq \rho k^2\;$,
\end{itemize}
or $\mathcal{T}_{k+1}\subseteq G$.
\end{lemma}

\bigskip

\begin{proof}[Proof of Theorem~\ref{thm_main}] 
Given $q>0$ let $c_\mathbf{E}$ be given by
Lemma~\ref{prop:EC-obecna}. Further let $\beta$ be given by
Lemma~\ref{prop:EC-obecna} with input parameters $q,c_\mathbf{E}$ and
$\sigma= c_\mathbf{E}$.\RefereeX{(26)}{Done. However, we left these subscript references in Section~\ref{ssec_Struct-Graph} and proofs of Lemma~\ref{lem:T^3} and Lemma~\ref{lem:cross-edges-1}, where we feel that they really help the reader.} Set $c= \frac {q\beta}{2}$
and $C=\lceil\frac 1q\rceil$.
We find
a sequence of parameters
\begin{equation}\label{eq:parameters-Main}
0<\sigma_1\ll\rho_1\ll\sigma_2\ll\rho_2\ll\dots\ll\rho_{C-1}
\ll\sigma_C\ll\rho_C\;,
\end{equation}
constructed as follows. Set $\rho_C= c$. Inductively for each $i=C,\dots
,1$ let $\sigma_i=\lambda(q,c,\rho_i)$\Referee{(27)} be given by Lemma~\ref{prop:iteration-Reg} for input
parameters $q,c$ and $\rho_i$.  Further let $\beta_i$ be given by
Lemma~\ref{prop:EC-obecna} with input parameters $q,c_\mathbf{E}$ and
$\frac{\sigma_i}2$. Finally for $i>1$ set $\rho_{i-1}= \frac {\beta_i}C$.
Set $n_0=\underset{i=1,\dots,C}{\max}\{n_0(q,c,\rho_i)\}$, where
the numbers $n_0(q,c,\rho_i)$ are from Lemma~\ref{prop:iteration-Reg}.

Let $G$ be a graph satisfying the conditions of Theorem~\ref{thm_main} (i.e., $q$
is fixed, $n\ge n_0$, and $k>qn$).

Recall that $\ci(x)$ denotes the closest integer to~$x$.\Referee{(28)} Let $\vartheta=\ci(\tfrac{n}k)$. We iterate the following process for at
most $\vartheta$ steps. In step $i,\; i\leq \vartheta$, we prove that $\mathcal T_{k+1}\subseteq G$ or we define a set $V_i\subseteq V\setminus \bigcup_{j<i}V_j$ such that the following conditions are fulfilled for each $j\in [i]$.
\begin{itemize}
\item [(P1)$_i$]$(1-\rho_i)k\leq |V_j|\leq (1+\rho_i)k$,
\item [(P2)$_i$]$|L\cap V_j|\geq (\tfrac{1}2-\rho_i)k$, and
\item [(P3)$_i$]$e(V_j,V\setminus V_j)\leq \rho_i k^2$.
\end{itemize}

In step $i=1$, we apply Lemma~\ref{prop:iteration-Reg} with
parameters $q,c,\rho_1$ and input set $\overV=V$. We obtain that
$\mathcal T_{k+1}\subset G$, or there exists a set $V_1$ satisfying (P1)$_{1}$,
(P2)$_{1}$, and (P3)$_{1}$. In step $i>1$, suppose that we have sets
$V_1,\dots,V_{i-1}$ satisfying (P1)$_{i-1}$, (P2)$_{i-1}$, and
(P3)$_{i-1}$.  Set $V^*=V\setminus \bigcup_{j<i}V_j$.

First assume that $|V^*|>ck$.
If $|L\cap V^*|\geq \tfrac{1}2(1-\sigma_{i})|V^*|$, the graph $G$
satisfies the conditions of Lemma~\ref{prop:iteration-Reg}
(with input parameters $q,c,\rho_i$ and input set $\overV=V^*$). Indeed,
$|V^*|>ck$ by assumption, $e(V^*,V\setminus V^*)\le (i-1)\rho_{i-1}k^2\le
\beta_ik^2<\sigma_ik^2$ because $V_1,\ldots,V_{i-1}$ satisfy (P3)$_{i-1}$, and $|L\cap V^*|\ge
\tfrac{1}2(1-\sigma_{i})|V^*|$ by assumption.\Referee{(29)}

If $|L\cap V^*|< \tfrac12(1-\sigma_{i})|V^*|$, then the partition
$V=V_1\dcup\dots\dcup V_{i-1}\dcup V^*$ is $(C\rho_{i-1},\tfrac{\sigma_{i}}2)$-extremal. 
Indeed, 
\begin{itemize}
  \item $i>1$;
  \item $(1-C \rho_{i-1})k\le (1-\rho_{i-1})k\le |V_j|\le
(1+\rho_{i-1})k\le (1+C\rho_{i-1})k$ for each $j\le i-1$ by (P1)$_{i-1}$;
\item $|V^*|>ck\ge \tfrac{\sigma_ik}2$ by assumption;
\item  $e(V_j, V\setminus V_j)\le
\rho_{i-1}k^2\le C\rho_{i-1}k^2$ for each $j\le i-1$ by (P3)$_{i-1}$ and
\\ $e(V^*,V\setminus V^*)\le (i-1)\rho_{i-1}k^2< C\rho_{i-1}k^2$;
\item $|V_j\cap L|\ge (\tfrac 12-\rho_{i-1})k\ge (\tfrac 12-C\rho_{i-1})k$
for each $j\le i-1$ by (P2)$_{i-1}$;
\item $|V^*\cap L|< \tfrac12(1-\sigma_i)|V^*|=(\tfrac
12-\tfrac{\sigma_i}2)|V^*|$.
\end{itemize}

Therefore Lemma~\ref{prop:EC-obecna} with parameters $q,c_{\mathbf{E}},
\tfrac{\sigma_{i}}2$ applies. Thus
$\mathcal
T_{k+1}\subseteq G$, or there exists a set $Q\subseteq V^*$ satisfying Lemma~\ref{prop:EC-obecna}~\ref{pr:velkyQ}--\ref{pr:Qizo}.\Referee{(30)}
It is enough to assume the latter case. Here again, the graph $G$ satisfies the
conditions of Lemma~\ref{prop:iteration-Reg} (with input parameters
$q,c,\rho_i$ and input set $\overV=Q$). Indeed, $|Q|>\tfrac k2\ge
\tfrac {q\beta}2k=ck$, $e(Q,V\setminus Q)<\tfrac {\sigma_i}2k^2<\sigma_ik^2$
and $|Q\cap L|>\tfrac{|Q|}2>\tfrac 12(1-\sigma_i)|Q|$.
Thus Lemma~\ref{prop:iteration-Reg} yields that
$\mathcal{T}_{k+1}\subseteq G$, or that there exists a set $V_i\subseteq Q$ satisfying Properties (P1)$_i$--(P3)$_i$.

It remains to deal with the case $|V^*|\le ck$. The set $V$ is decomposed into sets $V_1,\ldots,V_{i-1}$, each of which is of size approximately~$k$, and a little set~$V^*$. Thus, $i-1=\theta$.\Referee{(31)}
Having
found sets $V_1,\dots, V_\vartheta$ satisfying
(P1)$_\vartheta$--(P3)$_\vartheta$, we set $V_1'=V_1\cup V^*$ and $V_j'=V_j$ for $j\ge 2$. The thus defined partition $V=V'_1\dcup\dots\dcup V'_\vartheta\dcup\emptyset$ is
$(\beta,c_\mathbf{E})$-extremal.
Indeed, by (P1)$_\vartheta$--(P3)$_\vartheta$, we have 
\begin{itemize}
  \item $\vartheta\ge 1$;
  \item $(1-\beta)k\le (1-\rho_{\vartheta})k\le |V_j|\le |V'_j|\le
  |V_j|+|V^*|\le (1+\rho_{\vartheta}+c)k\le (1+\beta)k$ for each $j\le \vartheta$;
\item  $e(V'_j, V\setminus V'_j)\le e(V_j,V\setminus V_j)+e(V^*, V\setminus
V^*)\le  \rho_\vartheta k^2+(\vartheta-1)\rho_\vartheta k^2\le\beta k^2$ for each
$j\le \vartheta$ (the summand $e(V^*, V\setminus
V^*)$ is necessary only when $j=1$);
\item $|V'_j\cap L|\ge |V_j\cap L|\ge (\tfrac 12-\rho_\vartheta)k\ge
(\tfrac 12-\beta)k $ for each $j\le \vartheta$.
\end{itemize}

Lemma~\ref{prop:EC-obecna} with parameters $q, c_\mathbf{E}$ and
$\sigma= c_{\mathbf{E}}$ yields that $\mathcal{T}_{k+1}\subseteq G$ (as no new
set $Q$ can be found).
\end{proof} 
\section{Tools for the proof of Lemma~\ref{prop:iteration-Reg}}\label{sec_RLproof}
\subsection{Sparsity in the set of large vertices}\label{sec_Specialcase}
Suppose that $G$ is a graph with the LKS-property with parameter $k$ such that its set $L$ of
large vertices is almost independent. In this section we provide an ad-hoc
argument showing that in (a situation a bit more general than) the setting
above, we have $\mathcal T_{k+1}\subset G$. Indeed, in this case~$G$ is close to a $k$-regular bipartite graph with color classes $L$ and $S$, and thus we are roughly in the setting of Fact~\ref{fact:easyEmbedding1}.
\begin{lemma}\label{prop_SCHolds}
For every $q>0$ there exists a real $c_\mathbf{S}>0$ such that for  each
$c\in (0,c_\mathbf{S}]$ and each $n$-vertex graph $G=(V,E)$ with the
LKS-property with parameter $k>qn$, and with a set $\overV\subset V$ satisfying
\begin{enumerate}[label={$(\roman{*})$}]
\item\label{IT1} $|\overV |>\sqrt[4]{c}k\;$,
\item\label{IT2} $e(\overV ,V\setminus \overV )<c k^2\;$,
\item\label{IT3} $(\tfrac{1}2-c)|\overV |<|\overV \cap
L|\;$, and
\item\label{IT4} $e(G[\overV \cap L])<c n^2\;$,
\end{enumerate}
we have $\mathcal{T}_{k+1}\subset G$.
\end{lemma}
\begin{proof}Set $c_\mathbf{S}=q^9 10^{-8}$. Let $c\in (0,c_\mathbf{S}]$ be
arbitrary. Let~$G$ be any graph satisfying the assumptions of the lemma.
First observe that
\begin{equation}\label{eq:barVbig}
|\overV|\ge \tfrac{3k}4 \;.
\end{equation}
Indeed, suppose the  contrary. Assumptions~\ref{IT1} and~\ref{IT3} imply that $|\overV\cap L|\ge
(\tfrac{1}2-c)\sqrt[4]{c}k>\tfrac{1}4\sqrt[4]{c}k$.
By the negation\Referee{(35)} of~\eqref{eq:barVbig}, each vertex in $\overV\cap L$
emanates at least $\tfrac{k}4$ edges into $V\setminus \overV$. Therefore $e(\overV\cap L,V \setminus \overV)>\tfrac{1}{16}\sqrt[4]{c}k^2$, a contradiction
to~\ref{IT2}.

Fix a set $L_1\subset L\cap \overV$\Referee{(36)} of size
$|L_1|=(\tfrac{1}2-c)|\overV|$. Define $L_2=\{u\in L_1\:
:\: \deg(u,\overV\setminus L_1)\ge (1-2\sqrt{c})k\}$. For
each vertex $x\in L_1\setminus L_2$ we have that $\deg(x,L_1)+\deg(x,V\setminus \overV)>2\sqrt{c}k$, otherwise~$x$ would have been included in $L_2$.
Summing up~\ref{IT2} and~\ref{IT4}, we have $e(G[L_1])+e(L_1\setminus
L_2,V\setminus \overV)<2 cn^2$. Theorefore, we have that $$|L_1\setminus
L_2|\le\frac{4cn^2}{2\sqrt{c}k}\lBy{\eqref{eq:barVbig}}3\sqrt{c}q^{-2}|\overV|\le
\frac12\sqrt[4]{c} |\overV|\;.$$  Consequently,
\begin{equation}\label{eqLtilBig}
|L_2|>(\tfrac{1}2-\sqrt[4]{c})|\overV|\;.
\end{equation} We verify that the set $\tilde{S}=\{u\in
\overV \setminus L_1\: :\: \deg(u,L_2)\ge (1-\sqrt[8]{c})k\}$ covers
almost the whole set $\overV \setminus L_1$. Define $\overL=\{y\in
\overV \setminus L_1\::\:\deg(y,L_2)\ge k\}$. Observe that
$\overL\subset L\cap \tilde{S}$. By~\ref{IT4}, less than
$cn^2$ edges of $E[L_2,\overV \setminus L_1]$ are
incident with a vertex from $\overL$. Hence the number of edges in the bipartite graph $B=G[L_2,\overV \setminus (L_1\cup \overL)]$ is at least
\begin{equation}\label{eq:edgesBmany} 
e(B)\ge|L_2|(1-2\sqrt{c})k-cn^2\geBy{\eqref{eqLtilBig},\eqref{eq:barVbig}}(\tfrac12
-2\sqrt[4]{c})|\overV|k\;\mbox{.}
\end{equation}
On the other hand, we upper-bound the number of edges in the graph $B$ using the
fact that for each $x\in \tilde{S}\setminus\overL$ and for each $y\in \overV\setminus(L_1\cup\tilde S)$ we have that $\deg_B(x)< k$ and
$\deg_B(y)\le (1-\sqrt[8]{c})k$, respectively.\Referee{(37)}
\begin{align}\nonumber
e(B)&\le |\tilde{S}\setminus\overL|k+|\overV\setminus(L_1\cup\tilde
S)|(1-\sqrt[8]{c})k \qquad\big[\mbox{as $\tilde{S}\cup (\overV\setminus(L_1\cup\tilde
S))= \overV\setminus L_1$}\big]\\
\qquad&=
|\overV\setminus L_1|k-\sqrt[8]{c}|\overV\setminus(L_1\cup \tilde S)|k
\nonumber\\ 
\label{eq:edgesBfew} &=
(\tfrac{1}2+c)|\overV|k-\sqrt[8]{c}|\overV\setminus(L_1\cup \tilde S)|k \mbox{.}
\end{align}
Combining~\eqref{eq:edgesBmany} with~\eqref{eq:edgesBfew} we obtain
\begin{equation}\label{eq:koukej}
|\overV\setminus(L_1 \cup\tilde S)|\le 3\sqrt[8]{c}|\overV|\le \frac{3\sqrt[8]{c}k}q\;.
\end{equation}

By the choice of $L_2$ and $\tilde{S}$, the minimum degree of the vertices in
$L_2$ in the bipartite graph
$G_1=G[L_2,\tilde{S}]$ is at least $(1-2\sqrt{c})k-|\overV\setminus(L_1 \cup\tilde S)|$,
and of those in $\tilde{S}$ at least $(1-\sqrt[8]{c})k$. By~\eqref{eq:koukej} and
the choice of $c_\mathbf{S}$ we have that
$\delta(G_1)>\tfrac{k}2$.

Fact~\ref{fact:easyEmbedding1} applied on the graphs $B$ and $G$ yields that $\mathcal T_{k+1}\subset G$.
\end{proof}
\subsection{Cutting trees, and (un)balanced trees}\label{ssec_cut}
\RefereeX{(38)}{The notion of seeds is not needed anymore}
\Referee{(39)}

{
\renewcommand{\labelenumi}{(\roman{enumi}) }
\begin{definition}\label{def:ellfine}
An {\em $\ell$-fine partition} of a tree $T\in \mathcal{T}_{k+1}$ rooted at a vertex $R\in V(T)$ is a
quaternary $\mathcal{D}=(W_A,W_B, \mathcal{D}_A, \mathcal{D}_B)$ with the following properties.
\begin{enumerate}[label={$(\roman{*})$}]
\item $W_A$ and $W_B$ are sets of vertices in $V(T)$.
$\mathcal{D}_A$ and $\mathcal{D}_B$ are sets of subtrees in $T$. Further, $V(T)$ is a disjoint union of $W_A$, $W_B$, and the sets $V(t)$, $t\in \mathcal{D}_A\dcup \mathcal{D}_B$.\Referee{(40)}
\item The distance from each vertex in $W_A$ to each vertex in $W_B$ is
odd. The distance between each pair of vertices in $W_A$ or
between each pair of vertices in $W_B$ is even.
\item No tree from $\mathcal{D}_A$ is adjacent\footnote{a subtree $t$ is \emph{adjacent} to a vertex $v$ if there is at least one edge from $v$ to $V(t)$}\Referee{(41)} to any vertex in
$W_B$. No tree from $\mathcal{D}_B$ is adjacent to any vertex in
$W_A$.
\item $v(t)\leq \ell$ for each tree $t\in \mathcal{D}_A\cup
\mathcal{D}_B$.
\item $R\in W_A\cup W_B$.
\item\label{ellfine:max} $\max\{|W_A|,|W_B|\}\leq \frac{12k}{\ell}$.
\item $\mathcal{D}_B$ contains no internal tree.
\item We have $$\sum_{\substack{t\in \mathcal{D}_A\\t\textrm{ end-tree}}}v(t)\ge\sum_{t\in \mathcal{D}_B}v(t)\;\mbox{.}$$
\item Each internal tree from $\mathcal{D}_A$ is adjacent to two vertices of $W_A$.\RefereeX{(!)}{The last property was not present in the previous version of the manuscript. The reason we added it there was that it substantially simplifies the proof of Lemma~\ref{lem:Embedding-3}.
That means that the existence of a fine partition no longer directly follows from the work of Piguet and Stein and we therefore included a proof here. Even if the existence of a fine partition follows from~\cite{LKSsparse4}, the latter is a more recent paper and still in a preprint version}
\end{enumerate}
\end{definition}

For an $\ell$-fine partition
$\mathcal{D}=(W_A,W_B,\mathcal{D}_A,\mathcal{D}_B)$ the trees from $\mathcal{D}_A\cup\mathcal{D}_B$ are called {\em shrubs}.
For a subset $\mathcal F\subseteq \mathcal D_A\cup \mathcal D_B$, we denote the vertices contained in~$\mathcal F$ by $V(\mathcal F)$ and we write $v(\mathcal F)=|V(\mathcal F)|$.

It is proven in~\cite{PS07+} that for every $\ell$ each tree can be cut up in a way which results in a partition that satisfies (i)--(viii) of Definition~\ref{def:ellfine}. Here we extend this result by the additional requirement of (ix) from Definition~\ref{def:ellfine}. 
\begin{lemma}\label{lem:cutfine}
Let $T\in \mathcal{T}_{k+1}$ be a tree rooted at a vertex $R$ and let $\ell\in \mathbb{N}, \ell<k$. Then the rooted tree $(T,R)$ has an $\ell$-fine partition.
\end{lemma}

For the proof, we shall need the following easy claim.
\begin{fact}[{\cite[Proposition~7.11]{Z07+}}]\label{fact_FewLeavesManyDeg3}
	Let $T$ be a tree with $\ell$ leaves. Then $T$ has at most $\ell-2$
	vertices of degree at least three.
\end{fact}
\begin{proof}[Proof of Lemma~\ref{lem:cutfine}]
We first cut up the tree $T$ into components of order at most $\ell$. To this end we start with an empty set $W_1$ and place a token $v$ on the root $R$. At each step we check whether all the components of $T-v$ possibly except the one containing $R$ are of individual orders at most $\ell$. If that is the case then we insert $v$ into $W_1$, and we delete $v$ as well as all the said components from $T$. We restart with the token $v$ again on $R$. Otherwise, we move $v$ one vertex down to any component of order more than $\ell$. Obviously, at the stage when the process terminates, we have $|W_1|\le \frac{k+1}{\ell+1}$. Last, we add $R$ to $W_1$. Then $|W_1|\le \frac{k+1}{\ell+1}+1$.

Next, we want to refine the set of cut vertices $W_1$ in order to satisfy~(ix) of Definition~\ref{def:ellfine}.
To this end, consider the components $\mathcal D_{\ge 3}$ of $T-W_1$ that neighbour at least 3 vertices of~$W_1$.
Fix an arbitrary tree $t\in \mathcal D_{\ge 3}$. Let $X(t)\subset V(t)$ be the neighbors of $W_1$. Let $X'(t)$ be all the vertices of $X(t)$ with the $\preceq$-maximal element removed. 
We have $|X'(t)|=|X(t)|-1$. Consider the tree $\text{branch}(t)\subset t$ induced by the paths in $t$ connecting all the pairs of vertices of $X(t)$. Let $Y(t)$ be the vertices of degree at least 3 in $\text{branch}(t)$. By Fact~\ref{fact_FewLeavesManyDeg3}, we have $|Y(t)|\le |X(t)|-2< |X'(t)|$. Observe that a map assigning to each vertex of $\bigcup_{t\in\mathcal D_{\ge 3}} X'(t)$ any of its $\preceq$-minimal neighbors in $W_1$ is injective. Set $W_2=W_1\cup\bigcup_{t\in\mathcal D_{\ge 3}} Y(t)$. By the above, $|W_2|\le |W_1|+\sum_{t\in\mathcal D_{\ge 3}} |X'(t)|\le 2|W_1|\le \frac{2(k+1)}{\ell+1}+2$. Let $\mathcal S_A$ and $\mathcal S_B$ be a partition of all the components of $T-W_2$ where the respective membership of a component to $\mathcal S_A$ or to $\mathcal S_B$ is given by the parity of the distance of that component to $R$, and further such that 
$$\sum_{\substack{t\in \mathcal{S}_A\\t\textrm{ end-tree}}}v(t)\ge
\sum_{\substack{t\in \mathcal{S}_B\\t\textrm{ end-tree}}}
v(t)\;\mbox{.}$$
In particular, we can write $W_2=W_{2A}\dcup W_{2B}$ where $W_{2A}$ are the parents of all the components of $\mathcal S_A$ and $W_{2B}$ are the parents of all the components of $\mathcal S_B$. 

It remains to add further cut vertices in order to satisfy~(iii) and~(vii) of Definition~\ref{def:ellfine}. Initially, set $W_3=W_2$. For each internal tree $t_B\in\mathcal S_B$ we take its unique $\preceq$-maximal vertex and add it to the set $W_3$. Further, we add $\parent(W_{2B})\cap V(t_B)$ to $W_3$. For each internal tree $t_A\in\mathcal S_A$ we add $\parent(W_{2B})\cap V(t_A)$ to $W_3$. See Figure~\ref{fig:finepartition}.
\begin{figure}[t]
\centering\includegraphics[scale=0.8]{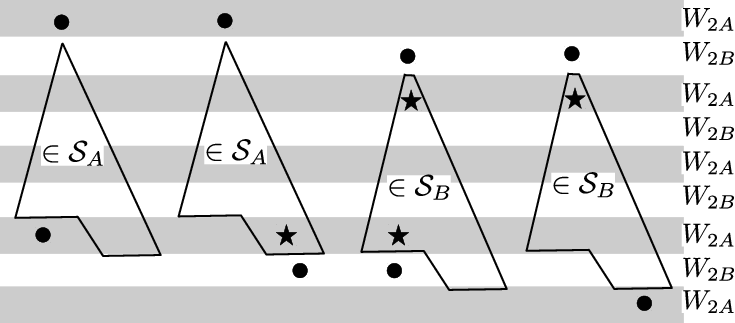}
	\caption{Obtaining the set $W_3$ from the set $W_2$ on examples of four internal trees depending on the parity of the neighbouring vertices of $W_2$ (which are denoted by dots). The newly added vertices are marked by stars.}
\label{fig:finepartition}
\end{figure}
As each vertex of $W_{2B}$ has at most one parent lying in some internal tree from $S_A\cup S_B$, we have
$$ 
|W_3|\le|W_2|+|\{\text{internal trees in $\mathcal S_B$}\}|+|W_{2B}|
\;.
$$
As each internal tree can be associated with a unique vertex of $W_2$ lying directly below it we get $|W_3|\le 3|W_2|\le \frac{6(k+1)}{\ell+1}+6\le \frac{12k}{\ell}$. It is straightforward to check that the set $W_3$ partitioned according to the bipartite colouring $W_3=W_A\dcup W_B$ with the correspondingly partitioned components $\mathcal D_A\dcup\mathcal D_B$ of $T-W_3$ satisfies all requirements of Definition~\ref{def:ellfine}.
\end{proof}

\bigskip


The next lemma will allow us to remove\Referee{(43)} trees which are locally unbalanced
from further considerations in our proof of Theorem~\ref{thm_main}. Let us
introduce the notion of (un)balanced forest now. For a real number $c\in
(0,\tfrac{1}2)$ we say that a family $\mathcal{C}$ of trees of total order at most $k+1$ is {\em
$c$-balanced} \label{p.:def-unbalanced} if the forest formed by the trees $t\in \mathcal{C}$ with
$|t_\ominus|> c\cdot v(t)$ is of order at least $c k$, i.\,e.,
$$\sum_{\substack{t\in \mathcal{C}\\|t_\ominus|> cv(t)}}v(t)\ge
ck\;\mbox{.}$$
Otherwise, we say that $\mathcal{C}$ is {\em $c$-unbalanced}.\RefereeX{(44)}{Somewhat differently incorporated}

Note that when $\mathcal{C}$ is $c$-balanced, then 
\begin{equation}\label{eq:ladvi}
\sum_{t\in \mathcal{C}}|t_\ominus|\ge
c^2k\;.
\end{equation}\RefereeX{(45)}{Somewhat differently incorporated}
\begin{lemma}\label{prop:UnbalancedCase}
For each number $q>0$ there exists a constant $c_\mathbf{U}>0$ such that the
following holds for each $n$-vertex graph $G$ with the LKS-property with
parameter $k>qn$. Suppose that $T\in \mathcal{T}_{k+1}$ is given. If there
exists a set $W\subset V(T)$, $|W|<c_\mathbf{U}k$ such that the family $\mathcal{C}$ of all components of the forest $T-W$ is
$c_\mathbf{U}$-unbalanced, then $T\subset G$.
\end{lemma}
\begin{proof}
Set $c_\mathbf{U}=\tfrac{c_\mathbf{S}}{6}$, where $c_\mathbf{S}$ is given by
Lemma~\ref{prop_SCHolds}.

If the set $L$ induces less then $c_\mathbf{S}n^2$ edges then
we have $T\subset G$ by Lemma~\ref{prop_SCHolds} with $\overV=V$. In the rest we
assume that $G[L]$ contains at least $c_\mathbf{S}n^2$ edges. A
well-known fact asserts that there exists a graph $G'\subset G[L]$
with minimum degree at least half of the average degree of $G[L]$,
i.\,e., $\delta(G')\ge c_\mathbf{S}n\ge
6 c_\mathbf{U}(k+1)$.

Let $\mathcal{C}'\subset \mathcal{C}$ be those trees $t\in
\mathcal{C}$ for which $|t_\ominus|\le  c_\mathbf{U} v(t)$. Since $\mathcal{C}$ is $c_\mathbf{U}$-unbalanced we have $\sum_{t\in \mathcal C\setminus \mathcal C'}v(t)<c_\mathbf{U} k$.\Referee{(46)} Consequently,
\begin{equation}\label{eq:Cprimbig}
\sum_{t\in \mathcal{C}'}v(t)=v(T)-|W|-\sum_{t\in \mathcal C\setminus \mathcal C'}v(t)>k+1-c_\mathbf{U} k-c_\mathbf{U} k>(1-2 c_\mathbf{U})(k+1)\;.
\end{equation}
Fact~\ref{fact_ManyLeavesInUnbalanced} gives that each tree
$t\in\mathcal{C}'$, $v(t)>1$ contains more than $(1-2 c_\mathbf{U})v(t)$ leaves. The same property holds trivially for each tree $t\in\mathcal{C}'$, $v(t)=1$. Employing~\eqref{eq:Cprimbig},\Referee{(47)} we get that there are at least
$(1-2 c_\mathbf{U})\sum_{t\in \mathcal{C}'}v(t)\ge(1-4 c_\mathbf{U})(k+1)$ leaves in the trees of $\mathcal{C}'$. A leaf of a
tree $t\in \mathcal{C}'$ is either a leaf
of~$T$ or it is adjacent to a vertex in $W$. We root $T$ at an
arbitrary vertex $r$, thus obtaining a partial order
$\preceq$. Let $X$ be the set of
vertices that are leaves of some tree $t\in \mathcal{C}'$
but not leaves of~$T$.\Referee{(48)} Each vertex in~$X$ is either a
$\preceq$-minimal or a $\preceq$-maximal vertex of some tree $t\in
\mathcal{C}$. Let $X_\mathrm{min}\subset X$ be the $\preceq$-minimal
vertices and $X_\mathrm{max}=X\setminus X_\mathrm{min}$. 
(Note that the vertices which come out from $1$-vertex trees of $\mathcal{C}'$
are included only in $X_\mathrm{min}$.)
As each
tree in $\mathcal{C}'$ has a unique $\preceq$-maximal vertex we get
$|X_\mathrm{max}|\le h$, where $h$ is the number of trees in
$\mathcal{C}'$ which have order more than 1. Observe\Referee{(49)} that each such
tree has at least $\tfrac1{c_{\mathbf{U}}}$ vertices and thus
$h\le  c_\mathbf{U}(k+1)$. For each $v\in X_\mathrm{min}$ we
have $|\children(v)\cap W|\ge 1$. Since for each $u\in W$ it holds
$|\parent(u)\cap X_\mathrm{min}|\le 1$, we have $|X_\mathrm{min}|\le
|W|< c_\mathbf{U}k$. Summing the bounds we get
$|X|<2 c_\mathbf{U}(k+1)$. Thus $T$ has at least
$(1-6 c_\mathbf{U})(k+1)$ leaves. Therefore, we can apply Fact~\ref{fact:easyemb2} on $G'\subset G$ and conclude that $T\subset G$.\Referee{(50)}
\end{proof}
 
\subsection{A matching structure}
A graph $H$ is said to be {\em
factor critical} if for each its vertex $v$ the graph $H-v$ has a
perfect matching.
The following statement is a fundamental result in Matching
theory. See~\cite[Theorem~2.2.3]{Die05}, for example.
\begin{theorem}[Gallai--Edmonds Matching Theorem]
\label{thm_GallaiEdmonds} Suppose that $H$ is a graph. Then there exist a set
$Q\subset V(H)$ and a matching $M$ of size $|Q|$ in $H$ such that
every component of $H-Q$ is factor critical and the matching $M$
matches every vertex in $Q$ to a different component of $H-Q$.
\end{theorem}
The set $Q$ in Theorem~\ref{thm_GallaiEdmonds} is called a {\em
separator}.
In order\Referee{(51)} to introduce the main result of this section, Lemma~\ref{prop_TutteType}, we
need the following setting.
\begin{setting}\label{set:TutteType}
Let $s>0$ and let $(H,\omega)$ be a weighted graph of order $N$, with
$\omega: E(H)\rightarrow (0,s]$. Let $\sigma, K$ be two positive reals with
$\tfrac1{2N}<\sigma<\min\{\tfrac{K}{32Ns},\tfrac{1}{30}\}$. Let $\mathcal{L}$ be
a\Referee{(52)} set of vertices such that
\begin{enumerate}[label={$(\roman{*})$}]
\item\label{ca:2} $V(H)\setminus \mathcal{L}$ is an independent set,
\item\label{ca:3} $|\mathcal{L}|>\tfrac{N}2-\sigma N$,
\item\label{ca:4} $\omega(u)\ge K$ for every $u\in \mathcal{L}$,
\item\label{ca:5} the set $\mathcal{L}$ induces at least one edge in $H$,
\item\label{ca:6} $\omega(u)<(1+\sigma)K$ for every $u\in V(H)\setminus
\mathcal{L}$.\RefereeX{(53)}{No, we need to work in the current mode. Our formulation is more general, and is necessary so that it fits the definition of the set $\mathcal L$ as given later, denoted by (53)* in the margin}
\end{enumerate}
\end{setting}

\begin{lemma}
\label{prop_TutteType} Let $s, N, \sigma, K, \mathcal L$, and a graph $(H,\omega)$
be as in Setting~\ref{set:TutteType}. Set $\mathcal{L}^*=\{u\in V(H)\: :\:
\omega(u)\ge\tfrac12(1+\sigma)K\}$.
Then there exist a matching $M$ such
that at least one of the following holds.
\begin{itemize}
\item[Case I]  There are two adjacent vertices $A,B\in V(H)\setminus V(M)$ with $A\in\mathcal L$, $\omega(A,V(M))\geq
K-s$, and $\omega(B,V(M)\cup \mathcal{L}^*)\ge
\tfrac12(1+\sigma)K$. For each edge $e\in M$ we have $|\neighbor(A)\cap e|\le 1$. 
\item[Case II]\RefereeX{!}{Case II and its proof were changed. Newly, there are no vertices $A,B$. This relates to comment (76).} There exists a set $\XXX\subset V(H)$ such that for each $x\in \XXX$ all but at most $2\sigma N$ neighbours of~$x$ are covered by~$M$. Furthermore, the set $\XXX\cap \mathcal{L}$ induces at least one edge, and
$|V(M')\setminus \XXX|\leq 1$, where $M'=\{xy\in M\::\: x,y\in \neighbor(\XXX)\}$.
\end{itemize}
Moreover, observe that each edge $e\in M$ intersects the set $\mathcal L$.
\end{lemma}
\begin{figure}[ht]
  \centering
\includegraphics{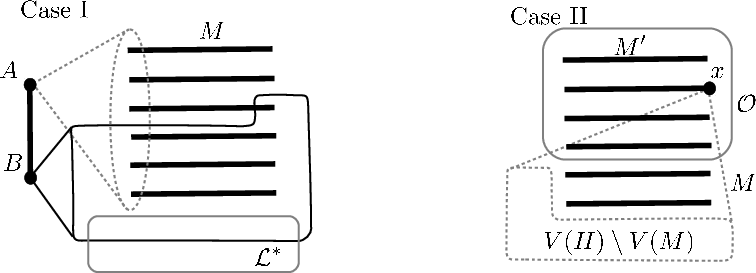}
  \caption{Two resulting matching structures from Lemma~\ref{prop_TutteType}. Dashed lines represent no connections (in Case~I), or sparse connections (in Case~II).}
  \label{fig_Tutte}
\end{figure}
\begin{proof}
Among all the matchings satisfying the conclusion of
Theorem~\ref{thm_GallaiEdmonds}, choose a matching~$M_0$ that covers the maximum number of vertices from $V(H)\setminus \mathcal{L}^*$. Let $Q$ be the corresponding separator. By definition, $M_0$ is a $Q\leftrightarrow
(V(H)\setminus Q)$-matching. Set $\mathcal L_0=\mathcal{L}\setminus Q$ and
$\mathcal{S}=V(H)\setminus \mathcal{L}$.
We distinguish three cases.

\noindent$\bullet$~~\emph{There exists an $\mathcal L_0\leftrightarrow \mathcal L_0$ edge.}\Referee{(54)}
Let $C$ be a component of $H-Q$ containing an $\mathcal L_0\leftrightarrow \mathcal L_0$ edge. 
If $V(M_0)\cap V(C)\not=\emptyset$, then we take $\{z\}=V(M_0)\cap V(C)$. Otherwise, we choose $z$ arbitrarily in $C$.\Referee{(56)}
Since $C$ is factor critical, there exists a perfect matching $M_1$ in $C-z$.
We claim that the conditions of Case~II are satisfied for $M=M_0 \cup M_1$, and
$\XXX=V(C)$. Thus, $\XXX\cap \mathcal{L}$ induces an edge. Next, let $x\in\XXX$.\Referee{(57)} We have
$\neighbor(x)\setminus\{z\}\subset V(M)$. Therefore, $\omega(x,V(M))\geq \omega(x)- s\ge \omega(x)-2\sigma Ns$. Consequently, all but at most $2\sigma N$ neighbours of~$x$ are covered by~$M$.\Referee{(58)} To check that $|V(M')\setminus \XXX|\leq 1$, it is enough to observe that each edge of $M'$ except at most one is contained entirely in $C$.\\

\noindent$\bullet$~~\emph{We have $\mathcal L_0=\emptyset$.}
Set $\XXX=V(H)$ and $M=M_0$. \RefereeX{(55)}{Not relevant, $A,B$ do not appear in the proof anymore.} Setting~\ref{set:TutteType}~\ref{ca:5} implies that there is an edge in $\XXX\cap \mathcal{L}$.
It is clear that $ V(M')\setminus\XXX=\emptyset$. Since $Q\supseteq \mathcal{L}$, $|\mathcal{L}|\ge
\tfrac{N}{2}-\sigma N$, and $|V(M)|=2|Q|$ it holds that all but at most $2\sigma
N$ vertices of $H$ are covered by $M$.\RefereeX{(58)}{Here, indeed, the bound was not needed.} The conditions of
Case~II are met.\\

\noindent$\bullet$~~\emph{$\mathcal L_0$ is an independent set and $\mathcal
L_0\not=\emptyset$.} We first derive some auxiliary properties of the
graph~$H$.
\begin{AuxiliaryCl}\label{AC:eachCompoSingle}
Each component $C$ of $H-Q$ is a singleton.
\end{AuxiliaryCl}
\begin{proof}\Referee{(59)}
Indeed, since $\mathcal{S}$ and $\mathcal L_0$ are independent, all the edges in each matching in $C$ are in the form
$\mathcal{S}\leftrightarrow \mathcal L_0$. Since $C$ is factor critical,
we have $|V(C-u)\cap \mathcal L_0|=|V(C-u)\cap \mathcal{S}|$ for each
vertex $u\in V(C)$. This is possible only when $v(C)=1$.\Referee{(60)}
\end{proof} 
Claim~\ref{AC:eachCompoSingle} implies that $M_0$ is a
maximum matching in~$H$.\RefereeX{(61)}{Similar references were added to Claims~\ref{cl:manyleaves},
\ref{cl-pdfd},
\ref{AC:assExists},
\ref{club3}} 
Define $\mathcal{\tilde{L}}=\{u\in \neighbor(\mathcal L_0)\:
:\:\omega(u)\ge K\}$. Observe that $\mathcal{\tilde{L}}\subset Q$.
 By Setting~\ref{set:TutteType}~\ref{ca:4}, we also
have
\begin{equation}\label{eq:L0ALT}
\neighbor(\mathcal L_0)\setminus \tilde{\mathcal L}\subset Q\setminus \mathcal L
\;.
\end{equation}
\begin{AuxiliaryCl}\label{AC:L(tilde)notempty}
We have $\mathcal{\tilde{L}}\not=\emptyset$.
\end{AuxiliaryCl}
\begin{proof}
Assume for contradiction that $\mathcal{\tilde{L}}=\emptyset$. Then for
every vertex $u\in\neighbor(\mathcal L_0)$ we have $\omega(u)<K$. We get $|\mathcal L_0|K\le
\omega(\mathcal L_0,\neighbor(\mathcal L_0))< K|\neighbor(\mathcal L_0)|$
(the second inequality is indeed strict because $\neighbor(\mathcal L_0)\neq \emptyset$) implying 
\begin{equation}\label{eq:L0NL0a}
|\mathcal L_0|<|\neighbor(\mathcal L_0)|\;.
\end{equation} On the other hand, from
$\mathcal{\tilde{L}}=\emptyset$ it follows that $\neighbor(\mathcal L_0)\cap \mathcal{L}=\emptyset$. Thus every vertex in $\neighbor(\mathcal L_0)$ is matched by $M_0$ to a distinct vertex
in $\mathcal L_0$, a contradiction to~\eqref{eq:L0NL0a}. 
\end{proof}

We show that the graph $V(H)$ fulfills the conditions of Case~I. 
Suppose first that  $B\in \neighbor(\mathcal L_0)$ is such that
$\omega(B,V(M_0)\cup \mathcal{L}^*)\ge \tfrac12(1+2\sigma)K$ and let $A\in \neighbor(B)\cap \mathcal L_0$ be arbitrary. Set $M=M_0\setminus\{A,B\}$. It can then be easily shown that that pair $(A,B)$ satisfies the conditions of Case~I.

So assume
that for every $B\in\mathcal{\tilde{L}}\subset \neighbor(\mathcal L_0)$ we have
\begin{align}
\label{eq_BinL*small} \omega(B,V(M_0)\cup
\mathcal{L}^*)&<\tfrac12(1+2\sigma)K \;\mbox{,}
\end{align}
which yields
\begin{align}
\label{eq_BinL*big} \omega(B,X)&>\tfrac12(1-2\sigma)K \;\mbox{,}
\end{align}
where $X=V(H)\setminus(V(M_0)\cup \mathcal{L}^*)$. 
\begin{AuxiliaryCl}\label{AC:match}
$M_0$ does not contain any edge with both end-vertices in
$\mathcal{L}$.
\end{AuxiliaryCl}
\begin{proof}
Indeed, suppose that such an edge $xy\in M_0$ exists. Then $x\in
\mathcal L_0$ and $y\in \mathcal{\tilde{L}}$. By~\eqref{eq_BinL*big},
$\omega(y,X)>\tfrac12(1-2\sigma)K$. In particular, there exists a vertex $p\in \neighbor_X(y)$. The matching $\{yp\}\cup M_0\setminus \{xy\}$ is a matching as in Theorem~\ref{thm_GallaiEdmonds} (with separator $Q$) which covers more vertices of $V(H)\setminus \mathcal L^*$ than $M_0$. This contradicts the choice of $M_0$.
\end{proof}
Observe that for each vertex $u\in X$, we have
$\omega(u,V(M))=\omega(u)<\tfrac12(1+\sigma)K$. As $\mathcal{\tilde{L}}\subset
V(M_0)$, we have $\omega(u,\mathcal{\tilde{L}})<\tfrac12(1+\sigma)K$. We bound
$\omega(\mathcal{\tilde{L}},X)$ from both sides.
$$(1-2\sigma)|\mathcal{\tilde{L}}|\frac{K}{2}\overset{\eqref{eq_BinL*big}}{\le}
\omega(\mathcal{\tilde{L}},X)\le (1+\sigma)|X|\frac{K}{2} \;\mbox{,}$$ which
yields
\begin{equation}\label{eq_XandL(tilde)}
|\mathcal{\tilde{L}}|\le \frac{1+\sigma}{1-2\sigma}|X| \;\mbox{.}
\end{equation}
We use~\eqref{eq_BinL*small} and $\mathcal L_0\subset \mathcal L^*$ to get
$\omega(\tilde{\mathcal{L}},\mathcal L_0)\le
|\tilde{\mathcal{L}}|(1+2\sigma)K/2$. Also, by the definition of $\tilde{\mathcal{L}}$, we have $\omega(\neighbor(\mathcal{L}_0)\setminus
\tilde{\mathcal{L}},\mathcal L_0)\le K|\neighbor(\mathcal L_0)\setminus
\tilde{\mathcal{L}}|$. Therefore,
\begin{align*}
|\mathcal L_0|K \le \omega(Q,\mathcal
L_0)&\le\omega(\mathcal{\tilde{L}}, \mathcal L_0)+
\omega(\neighbor(\mathcal{L}_0)\setminus
\tilde{\mathcal{L}},\mathcal L_0)
\\ 
&\le
(1+2\sigma)\frac{K}{2}|\tilde{\mathcal{L}}|+K|\neighbor(\mathcal L_0)\setminus
\tilde{\mathcal{L}}| \\ &\overset{\eqref{eq:L0ALT}}\le (1+2\sigma)\frac{K}{2}
|\tilde{\mathcal{L}}|+ K |Q\setminus \mathcal{L}|\;
\mbox{,}
\end{align*}\RefereeX{(62)}{No, actually the current computation cannot be shortened. Indeed, it is true that $|\mathcal L_0|K \le \omega(Q,\mathcal
	L_0)\le\omega(\mathcal{\tilde{L}}, \mathcal L_0)+
	\omega(Q\setminus \tilde{\mathcal{L}},\mathcal L_0)
	$. However, we do not have a way of bounding the degree of vertices in $Q\setminus \tilde{\mathcal{L}}$ by $K$. However, we added an explanation of getting the second line.}
which gives
\begin{equation}\label{eq_L0minusM}
2|\mathcal L_0|\le (1+2\sigma)|\mathcal{\tilde{L}}|+2|Q\setminus
\mathcal{L}|\;\mbox{.}
\end{equation}
Every vertex in $Q\setminus \mathcal{L}$ is matched with a vertex
in $\mathcal L_0$. The converse is true due to Claim~\ref{AC:match}: if a vertex
in $\mathcal L_0$ is matched then it is matched with a vertex in $Q\setminus\mathcal L$. Therefore,
$|Q\setminus \mathcal L|=|\mathcal L_0\cap V(M_0)|$. Combined
with~\eqref{eq_L0minusM} we get that $2|\mathcal L_0\setminus V(M_0)|\le
(1+2\sigma)|\mathcal{\tilde{L}}|$. Plugging in~\eqref{eq_XandL(tilde)} we obtain
\begin{equation}\label{eq_L'-V(M)andX}
2|\mathcal L_0\setminus V(M_0)|\le \frac{(1+2\sigma)^2}{1-2\sigma}|X| \;\mbox{.}
\end{equation}
By Setting~\ref{set:TutteType}~\ref{ca:3}, we have
$|\mathcal{L}|>|V(H)\setminus \mathcal{L}|-2\sigma N$. By Claim~\ref{AC:match},
we get $|\mathcal L_0\setminus V(M)|\ge |X|-2\sigma N$. Combined with~\eqref{eq_L'-V(M)andX} we obtain
\begin{equation*}
2|X|-4\sigma N\le \frac{(1+2\sigma)^2}{1-2\sigma}|X| \; \mbox{.}
\end{equation*}
We use the bounds $\sigma\le\min\{\tfrac{K}{32Ns},\tfrac{1}{30}\}$ to get\Referee{(63)}
\begin{equation}\label{eq:Acont}
|X|\le \frac{4\sigma N}{1-14\sigma}\le 8\sigma N\le \frac{8K}{32s}  \;
\mbox{.}
\end{equation}
On the other hand,
using~\eqref{eq_BinL*big} and Claim~\ref{AC:L(tilde)notempty}, we get $\omega(\mathcal{\tilde{L}},X)>\frac12(1-2\sigma)K|\mathcal{\tilde{L}}|$. As
$\omega(e)\le s$ for each $e\in E(H)$ we get $\omega(\mathcal{\tilde{L}},X)\le s|\mathcal{\tilde{L}}||X|$.\Referee{(64)} Combining these two bounds we arrive at
$$|X|>\frac{(1-2\sigma)K}{2s}>\frac{K}{4s}\;,$$
a contradiction to~\eqref{eq:Acont}.
\end{proof}

\subsection{Regularity Lemma}\label{ssec_RegularityLemma}\Referee{(32)}
In this section we recall briefly the Regularity
Lemma~\cite{Sze78} and establish related notation. The reader may
find more on the Regularity Method in~\cite{KS96,KSSS00,KuhnOsthusSurv}.

Let $H=(V(H);E(H))$ be a graph. For two nonempty disjoint sets $X,Y\subset V(H)$ we denote the {\em density} of the pair $(A,B)$ by $\density(A,B)=\tfrac{e(A,B)}{|A||B|}.$ The pair $(A,B)$ is \emph{$\varepsilon$-regular}, if for any subsets $X\subseteq A$, $Y\subseteq B$ with $|X|>\varepsilon |A|$ and $|Y|>\varepsilon |B|$, we have $|d(X,Y)- d(A,B)|<\varepsilon$. Such sets~$X$ and~$Y$ are called {\em significant}. We say that a vertex $v\in A$ is {\em typical} with respect to (``w.~r.~t.'') a significant set $Y\subset B$, if $\deg(v,Y)\ge (\density(A,B)-\varepsilon)|Y|$. Analogously, if $\{(A,B_i)\}_{i=1}^\ell$ are $\varepsilon$-regular pairs, and $Y_i\subseteq B_i$ are significant, a vertex $v\in A$ is \emph{typical} w.~r.~t.~$\bigcup_{i-1}^\ell Y_i$, if $\deg(v,\bigcup_{i=1}^\ell Y_i)\ge \sum_{i=1}^\ell (d(A,B_i)-\varepsilon)|Y_i|$. Note that our definitions of typicality is only one-sided; this turns out to be sufficient for our proof.
\begin{fact}\label{fact:zap}
Let $X,Y_1, Y_2,\ldots, Y_\ell$ be disjoint sets of
vertices, such that $(X,Y_1), (X,Y_2), \ldots, (X,Y_\ell)$ are $\varepsilon$-regular pairs. Suppose that sets
$W_i\subset Y_i$ are significant.
\begin{enumerate}[label={$(\roman{*})$}]
\item\label{it:za1} All but at most $\varepsilon |X|$
vertices of $X$ are typical w.~r.~t.\ $\bigcup_{i=1}^{\ell}W_i$. 
\item\label{it:za2} 
All but at most $\sqrt{\varepsilon}|X|$ vertices of $X$
are typical w.~r.~t.\ at least
$\sqrt{\varepsilon}\ell$ sets $W_i$.
\end{enumerate}
\end{fact}
The proof of~\ref{it:za2} can be found in~\cite[Proposition~4.5]{Z07+}. We prove~\ref{it:za1} in the Appendix. 
The next fact is the well-known ``slicing property'' of regular pairs.
\begin{fact}[{\cite[Fact~1.5]{KS96}}]\label{fact:slicing}
Suppose that $(X,Y)$ is an $\epsilon$-regular pair of density $d$. Let $A\subset
X$ and $B\subset Y$ be such that $|A|>\alpha |X|$, and $|B|>\alpha |Y|$ for
$\alpha>2\epsilon$. Then the pair $(A,B)$ is
$\max\left\{\frac{\epsilon}{\alpha},2\epsilon\right\}$-regular of density at
least $d-\epsilon$.
\end{fact}

A partition $V(H)=V_0\dcup V_1\dcup \ldots\dcup V_N$ of the vertex set a graph
$H$ is called {\em $(\varepsilon,N)$-regular} if
$|V_0|<\varepsilon v(H)$,
$|V_i|=|V_j|$ for every $i,j\in [N]$, and
for each $i\in [N]$ at most $\varepsilon N$ pairs $(V_i,V_j)$ (where $j\in [N]$) are not $\varepsilon$-regular.
The sets $V_1,\ldots, V_N$ are called {\em clusters}.

We are now ready to state a standard version Szemer\'edi's original result~\cite{Sze78}.\RefereeX{(33)}{We do not think it is ``multipartite''. So, as a compromise, we suggest to refer to it as a ``standard version''.}
\begin{theorem}[{\cite{Sze78}}]\label{thm_regularity}
For every $\varepsilon>0$ and every $m_0,r\in\mathbb{N}$, there
exist numbers $M_0, N_0\in\mathbb{N}$ such that every graph $H$ of
order $m\geq N_0$ whose vertex sets is partitioned into $r$ sets
$V(H)=O_1\dcup O_2\dcup\ldots\dcup O_r$ admits an
$(\varepsilon, N)$-regular partition $V(H)=V_0\dcup V_1\dcup \ldots\dcup V_N$ for some
$m_0\leq N\leq M_0$ such that for every $i\in [N]$ we have
$V_i\subset O_j$ for some $j\in [r]$.
\end{theorem}
In the above setting, let $H_d$ denote the graph obtained from $H$ by deleting the
edges incident to $V_0$, contained in some $V_i$, or in pairs of clusters that are irregular or of
density smaller than some fixed constant $d$. Let $\mathbf{H}$ denote the \emph{cluster graph} induced by
$H_d$.
That is, $\mathbf{H}$ has order $N$, its vertices are
$V(\mathbf{H})=\{V_1,\dots,V_N\}$ and edges are $$E(\mathbf{H})=\{V_iV_j \: :\: (V_i,V_j)
\mbox{ is a $\varepsilon$-regular pair with density at least $d$}\}\;
\mbox{.}$$ 
Set $\wdeg_{\mathbf{H}}(C,D):=\wdeg_{H_d}(C,D)$, for any disjoint sets $C,D\subseteq V(H)$. The function $\wdeg_\mathbf{H}$ induces a weight function on $\mathbf{H}$.

\subsection{Embedding lemmas}\label{ssec_embed}
In this section, we introduce tools for embedding trees into regular pairs. Similar results are folklore. Here we give statements tailored to our needs; their proofs are included in the Appendix. The first lemma deals with embedding a tree into one regular pair.
\begin{lemma}\label{lemma_Embedding1}
Let $(t,r)$ be a rooted tree, and $d>2\varepsilon>0$. Let $(X,Y)$ be an
$\varepsilon$-regular pair with $|X|=|Y|=s$ and density $\density(X,Y)\ge
d$. Let $P'\subset P\subset X$ and $Q'\subset Q\subset Y$ be such
that $\min\{|P|,|Q|\}\ge \Delta$ and $\max\{|P'|,|Q'|\}\ge \Delta$,
where $\Delta\ge\frac{\varepsilon s+ v(t)}{d-2\varepsilon}$. Then
there exists an embedding~$\phi$ of $t$ to $P\cup Q$ such that the
root $r$ is mapped to $P'\cup Q'$. Moreover, 
if $|P'|\ge \Delta$,  the vertex~$r$ can be mapped to $P'$, and if $|Q'|\ge \Delta$,  the vertex
$r$ can be mapped to $Q'$.
\end{lemma}
The next lemma deals with embedding a tree using a matching structure in the underlying cluster graph. A simplified picture of the situation is given in Figure~\ref{fig:L512}. \begin{figure}[ht]
	\centering
	\includegraphics{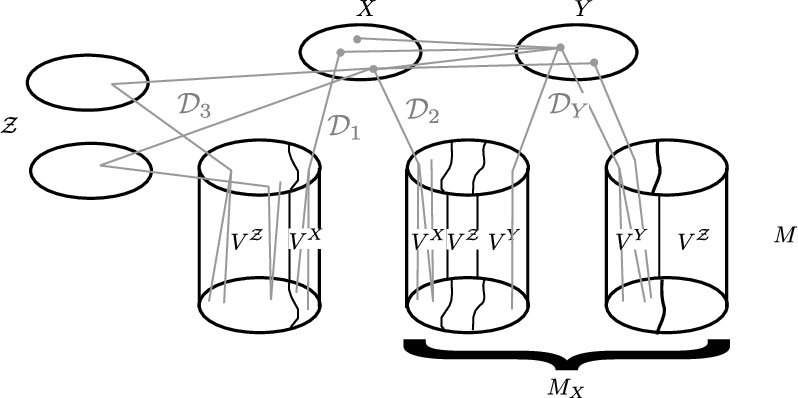}
	\caption{A simplified picture of an embedding provided by Lemma~\ref{lem:Embedding-3}. The lemma provides with an embedding of a tree with a given fine partition $(W_X,W_Y,\mathcal D_X,\mathcal D_Y)$. The cut-vertices $W_X$ and $W_Y$ are mapped to $X$ and $Y$, respectively. The shrubs $\mathcal D_Y$ are mapped to the part $V^Y$ of the regular matching $M$. The shrubs $\mathcal D_X$ are embedded using one of three different ways which is indicated by the partition $\mathcal D_X=\mathcal D_1\dcup \mathcal D_2\dcup \mathcal D_3$. The shrubs of $\mathcal D_1$ are mapped to $V^X\setminus \bigcup V(M_X)$. The shrubs of $\mathcal D_2$ which are required to be balanced are mapped to $V^X\cap \bigcup V(M_X)$. Finally, the shrubs of $\mathcal D_3$ are accommodated to a set $V^{\mathcal Z}$ with their roots placed to an additional set of clusters $\mathcal Z$.}
	\label{fig:L512}
\end{figure}
\RefereeX{(65)}{A simplified illustration of the situation was added. The sets $\mathcal D_1,\ldots,\mathcal D_3$ form indeed an arbitrary partition of $\mathcal D_X$ as is written in the text. Remark about rewording Line 6 incorporated.}
\begin{lemma}\label{lem:Embedding-3}
\setcounter{lematko511}{\value{theorem}}
Let $0<\varepsilon, \xi, d\le 1$
and $\tau, s$ be such that $\tau/s\le \varepsilon\le \xi^2 d/400$.
\RefereeX{!}{The condition $\tau/s\le \varepsilon\le \xi^2 d/400$ was changed due to changes in the proof of the lemma}
Let
$F$ be a tree of order at most $k+1$ with a $\tau$-fine partition $(W_X,W_Y,\mathcal D_X,\mathcal D_Y)$ and let $\mathcal D_1\dcup \mathcal
D_2\dcup \mathcal D_3$ be an arbitrary partition of $\mathcal D_X$. Let $\mathbf{H}$ 
be a cluster graph corresponding to an $\epsilon$-regular partition of an
$n$-vertex graph $H$, whose edges have density at
least~$d$ and clusters have size~$s$. 
Let $XY\in E(\mathbf{H})$, and 
$X'\subseteq X$, $Y'\subseteq Y$ such that $|X'|,|Y'|\ge (1-d/2) s$. Let
$\mathcal Z\subseteq V(\mathbf{H})\setminus \{X,Y\}$. Further let  $M_X\subseteq M$ be matchings in $\mathbf {H}$ disjoint from $\mathcal Z\cup \{X,Y\}$ such that for each edge of $M_X$ contains at most one vertex of $\neighbor_{\mathbf{H}}(X)$. Let  $V^X,V^Y, V^{\mathcal Z}\subseteq
\bigcup V(M)$ be pairwise disjoint sets. Suppose that
\begin{enumerate}[label={$(\roman{*})$}]
\item\label{emb3-balanced} For all $CD\in M$,  $|C\cap V^X|=|D\cap V^X|$, $|C\cap V^Y|=|D\cap V^Y|$, and  $|C\cap V^{\mathcal Z}|=|D\cap V^{\mathcal Z}|$.
\item \label{emb3-sign}For all $C\in V(M)$,  $|C\cap V^X|, |C\cap V^Y|, |C\cap V^{\mathcal Z}|\in \{0\}\cup (\varepsilon s, s]$.
\item \label{emb3-V^Y} If $\mathcal D_Y\neq \emptyset$, then $\wdeg_{\mathbf{H}}(Y,V^Y)\ge
v(\mathcal D_Y)+\xi n$\;.
\item \label{emb3-V^X} If $\mathcal D_1\neq \emptyset$, then $\wdeg_{\mathbf{H}}(X,V^X\setminus \bigcup V(M_X))\geq v(\mathcal D_1)+\xi n$\;.
\item \label{emb3-MX} If $\mathcal D_2\neq \emptyset$, then $\mathcal D_2$ is
$c$-balanced and $\wdeg_{\mathbf{H}}(X,V^X\cap \bigcup V(M_X))\geq v(\mathcal D_2)-c^2k+\xi n$\;.
\item \label{emb3-AZ} If $\mathcal D_3\neq \emptyset$, then $\wdeg_{\mathbf{H}}(X,\bigcup \mathcal
Z)\geq |V(\mathcal D_3)\cap \neighbor_F(W_X)|+\xi n$\;.
\item \label{emb3-Z} If $\mathcal D_3\neq \emptyset$, then for each $Z\in \mathcal Z$,
$\wdeg_{\mathbf{H}}(Z,V^{\mathcal Z})\geq v(\mathcal D_3)+\xi n$\;.
\end{enumerate}
 Then there is an embedding $\varphi$ of $F$ in $H$ such that $\varphi(W_X)\subseteq X'$, $\varphi(W_Y)\subseteq Y'$, $\varphi(V(\mathcal D_Y))\subseteq V^Y$, and  $\varphi(V(\mathcal D_X))\subseteq (V^X\cup V^{\mathcal Z}\cup \bigcup \mathcal Z)$.
\end{lemma}

\section{Proof of Lemma~\ref{prop:iteration-Reg}}\label{sec_TheProof}
Suppose that $q,c$, and $\rho$ are given. Let  $c_{\mathbf S}$ be given by
Lemma~\ref{prop_SCHolds} for input parameter $q$. Further, let
$c_\mathbf{U}$ be given by Lemma~\ref{prop:UnbalancedCase} for input
parameter $q$.
Set reals $\alphaX, \betaX,\gammaX,\epsilonX,\etaX,\sigmaX,\omegaX$ so that
\[
0<\betaX\ll\epsilonX\ll\gammaX\ll \sigmaX\ll\alphaX\ll\etaX\ll\omegaX\ll
\min\{\rho,c,c_\textbf{S},c_\mathbf{U}, q\}.
\]

Let $n_0$ (the minimal order of the graph) and $\Pi_1$ (the upper bound for the number of clusters) be the
numbers given by the Regularity Lemma~\ref{thm_regularity} for input
parameters $\betaX$ (for precision), $\Pi_0=\tfrac2\betaX$ (for the\Referee{(66)} minimum number of
clusters) and $4$ (for the number of pre-partition classes).

Let $G$ be a graph of order $n\geq n_0$ that has the LKS-property. 
We can assume that $G$ is
LKS-minimal, that is, there is no proper spanning subgraph $G'\subset G$ with
the LKS-property. Then clearly,
\begin{align}
\label{ass_Sindep}
\mbox{the set $S$ is independent}\;.
\end{align}
Let $\overV\subseteq V$ satisfy the assumptions of Lemma~\ref{prop:iteration-Reg} and
let $T\in \mathcal{T}_{k+1}$ be arbitrary. Our goal is to show that $T\subset G$. Root $T$ at an arbitrary vertex~$R$, and consider any $\tau$-fine partition $(W_A,W_B,\mathcal{D}_A,\mathcal{D}_B)$ of $(T,R)$, with $\tau=\tfrac{\betaX k}{\Pi_1}$. The existence of such a partition follows from Lemma~\ref{lem:cutfine}.

Prepartition the vertex set $V$ into
$\overV\cap L,\overV\cap S, L\setminus \overV$, and $S\setminus \overV$. By the Regularity Lemma~\ref{thm_regularity}, there
exists a partition $ V =C_0\dcup C_1\dcup\dots\dcup C_N$ satisfying
the following.
\begin{enumerate}[label=\textbf{(R\arabic{*})}]
\item \label{RL:bound}$\Pi_0\leq N\leq \Pi_1$,
\item \label{RL:clustersize}$|C_i|=s$ for each $i\in [N]$,\Referee{(67)}
\item \label{RL:garbageset}$|C_0|\leq \betaX n$,
\item\label{RL:mindegree} for each $i\in [N]$, all but at most $\betaX N$ pairs
$(C_i,C_j)$ (where $j\in [N]$) are $\betaX$-regular,
\item \label{RL:prepartition}for each $i\in[N]$, if $C_i\cap L\neq \emptyset$ then $C_i\subseteq
L$, and if $C_i\cap \overV\neq \emptyset$ then
$C_i\subseteq \overV$.
\end{enumerate}

Let $G_{\gammaX}$ denote the graph obtained from $G$ by deleting the
edges incident to $C_0$, contained in some $C_i$, or in pairs of clusters that are irregular or of
density smaller than $\gammaX$ and let $\mathbf{G}$ be the corresponding cluster graph with weight function $\wdeg_\mathbf{G}$. \RefereeX{(68)}{I do not see what would these be used for. Certainly not for deriving~\eqref{eq:edgesGvsGg}} 
Observe that by~\ref{RL:bound}--\ref{RL:mindegree} we have\begin{equation}\label{eq:edgesGvsGg}
e(G_{\gammaX})\ge e(G)-2\betaX n^2-\gammaX n^2\ge e(G)-\sigmaX
k^2\;.
\end{equation}

Denote by $\mathcal{L}$ the set
of clusters contained in $L\cap \overV$ which have large average degree in
$\overV$: \Referee{(53)*}
$$\mathcal{L}=\{C\in V(\mathbf{G})\: :\: C\subseteq L\cap
\overV,\ \wdeg_{\mathbf{G}}(C,\overV)\ge k-\sqrt{\sigmaX} n\}\;
\mbox{.}$$ 
Note that~\ref{RL:prepartition} supports\Referee{(69)} the definition and observation below. Let
$\overcalV$ be the set of clusters contained in $\overV$; we
write $\overN=|\overcalV|$. Observe that each cluster inside
$L\cap \overV$ is in $\mathcal L$, unless it sends many edges to $V\setminus \overV$. 
To estimate the size of $\mathcal L$,\RefereeX{(70)}{different formulation than suggested used} we set $\mathcal B=\{C \in V(\mathbf{G}):
C\subseteq \overV, \wdeg_\mathbf{G}(C,V\setminus \overV)\ge
\frac{\sqrt{\sigmaX} n}2\}$. It follows from the assumptions of
Lemma~\ref{prop:iteration-Reg} that 
\begin{equation}\label{eq:boB}
|\mathcal B|\le 3\sqrt{\sigmaX} N\;.
\end{equation}
Further,
observe that we have 
\begin{equation}\label{eq:Lsup}
\mathcal L\supset \{C\in V(\mathbf{G}): C\subset
\overV \cap L\}\setminus \mathcal B\;.
\end{equation}
The ratio $|L\cap \overV|:|\overV|$ approximately corresponds to $|\mathcal{L}|:|\overcalV|$. More precisely, we use later the following lower-bound on $|\mathcal{L}|$,\RefereeX{(71)}{The precise argument is exactly the calculation below, the previous sentence was just to emphasize our starting point.}
\begin{equation}\label{eq:boundScriptL}
|\mathcal{L}|\geq
\tfrac12(1-2\sigmaX)\overN-|\mathcal B|\geq \tfrac{\overN}2-4\sqrt{\sigmaX}N\ge \tfrac{\overN}2-\alphaX\overN\;,
\end{equation}
where we use $\sigmaX\ll \alphaX\ll c,q$.



Let $\mathbf{H}$ be the subgraph of $\mathbf{G}$ induced by $\overcalV$ such that all the edges induced by the set $\overcalV\setminus\mathcal L$ are removed. The cluster graph $\mathbf{H}$ naturally inherits  the function $\wdeg_{\mathbf{G}}$ of $\mathbf{G}$ (which is denoted by~$\wdeg_{\mathbf{H}}$). The next lemma gives some simple properties of $\mathbf{H}$.\Referee{(72)}
\begin{lemma}\label{lem:por}$~$
\begin{enumerate}[label=(\roman*)]
\item\label{por1} For each $C\in \mathcal L$, we have
$
\sum_{D\in \neighbor_{\mathbf H}(C)}\wdeg_\mathbf{H}(C,D)=
\sum_{D\in \neighbor_{\mathbf G}(C, \overV)}\wdeg_\mathbf{G}(C,D)$.
\item\label{por2}
All but at most $3\sqrt{\sigmaX}N$
clusters $C\in \overcalV\setminus \mathcal L$ satisfy
$$\sum_{D\in \neighbor_{\mathbf H}(C)}\wdeg_\mathbf{H}(C,D)\ge
\Big(\sum_{{D\in \neighbor_{\mathbf
G}(C,\overcalV)}}\wdeg_\mathbf{G}(C,D)\Big)-3\sqrt{\sigmaX}n\;.$$
\end{enumerate}
\end{lemma}
\begin{proof}
Part~\ref{por1} follows directly.\RefereeX{!}{The proof was simplified, constants improved}
 Let us now deal with part~\ref{por2}. By~\eqref{eq:boB}, for any cluster~$C\in \overcalV\setminus (\mathcal L\cup \mathcal B)$, we have $\wdeg_{\mathbf{H}}(C)\ge \wdeg_\mathbf{G}(C,\overV)-|\mathcal B|s\ge \wdeg_\mathbf{G}(C,\overV)-3\sqrt{\lambda}n$, as edges of $\mathbf G$ sent by~$C$ go either to $\mathcal B$ or are kept in $\mathbf H$. At most $3\sqrt{\lambda}N$ clusters in $\overcalV\setminus \mathcal L$ may be contained in $\mathcal B$.
\end{proof}

\subsection{Matching structure in the cluster graph}
\label{ssec_Struct-Graph}
Set $c'=\min\{c_{\mathbf{S}},c^4\}$. If $e(G[\overV\cap L])<
c'n^2$, then the conditions of Lemma~\ref{prop_SCHolds} are satisfied for
the set $\overV$ and parameter $c_{\mathrm L\ref{prop_SCHolds}}=c'$. Indeed,
the assumptions~\ref{IT1}--\ref{IT3} of Lemma~\ref{prop_SCHolds} follow from the assumptions of
Lemma~\ref{prop:iteration-Reg}, and the fact that $\sigmaX\ll c'$. Then, by
Lemma~\ref{prop_SCHolds} we get $\mathcal{T}_{k+1}\subseteq G$. Therefore,
 we  assume in the rest of the proof that $e(G[\overV\cap L])\geq c'n^2$.
By~\eqref{eq:edgesGvsGg} as $\sigmaX\ll c'$, we get
$e(G_{\gammaX}[\overV\cap L])\geq \frac {c'}2n^2$. 
\begin{lemma}\label{lem:Ledge}
The set $\mathcal{L}$ induces at least one edge in $\mathbf{H}$.
\end{lemma}  
\begin{proof}
By~\eqref{eq:Lsup} and~\eqref{eq:boB} at most $4\sqrt{\sigmaX}n^2$ of the edges
of $E(G_{\gammaX}[\overV\cap L])$ are not induced by the vertices of $\bigcup
\mathcal L$. As $e(G_{\gammaX}[\overV\cap L])\geq\frac {c'}2n^2\gg 4\sqrt{\sigmaX}n^2$, $\mathcal{L}$ induces at least one
edge in $\mathbf{G}$. This edge is also an edge in $\mathbf{H}$.
\end{proof}

The weighted graph $(\mathbf{H},\wdeg_\mathbf{H})$ satisfies the conditions
of Lemma~\ref{prop_TutteType}  with parameters $s$, $N=\overN$, $\sigma=
\alphaX$, and $K=k-\sqrt{\sigmaX} n$.\Referee{(73)} Let us verify\RefereeX{(74)}{``Let us verify'' instead of the suggested ``In fact''}  Conditions~\ref{ca:2}--\ref{ca:6} of Setting~\ref{set:TutteType}.
Condition~\ref{ca:2} is satisfied by the way $\mathbf{H}$ was derived from $\mathbf{G}$. Condition~\ref{ca:3} follows from~\eqref{eq:boundScriptL}. Condition~\ref{ca:4} is given by the definition of~$\mathcal L$. Condition~\ref{ca:5} was derived in Lemma~\ref{lem:Ledge}. Finally, Condition~\ref{ca:6} follows from the definitions of~$L$ and~$\mathcal L$. Lemma~\ref{prop_TutteType} ensures that one of the two specific matching structures in $\mathbf{H}$ exists. 

\bigskip
\noindent {\em Case I: }There are two adjacent clusters  $A,B$ and a matching
$M$ in $\mathbf{H}-\{A,B\}$  such that:
\begin{enumerate}[label={$(\alph{*})$}]
\item\label{hippiA} We have $\wdeg_\mathbf{H}(A,V(M))\geq
k-2\sqrt{\sigmaX} n$.\Referee{(75)}
\item\label{hippiB}  For each edge $e\in M$ we have
$|\neighbor_{\mathbf{H}}(A)\cap e|\le 1$.
\item\label{hippiC} There is a set $\mathcal{L}^*\in V(\mathbf H)$ such that for all $C\in \mathcal{L}^*$ we have $
\wdeg_\mathbf{H}(C)\geq (1+\tfrac\alphaX2)\frac k2$ and 
\begin{equation}\label{eq:WF}
\wdeg_\mathbf{H}(B,V(M)\cup\mathcal{L}^\ast)\geq (1+\tfrac \alphaX 2)\frac
k2\;.
\end{equation}
\end{enumerate}

\noindent {\em Case II: }There exist a set of clusters $\XXX\subseteq V(\mathbf H)$ and a matching $M$ in $\mathbf{H}$ such that:\RefereeX{(76)}{No, actually the clusters $A$ and $B$ cannot be defined here yet. The wording of Lemma~\ref{prop-CaseIIpart1} hints why.}
\begin{enumerate}[label={$(\alph{*})$}]
\item\label{hypnousA} $ \XXX\cap \mathcal{L}$ induces at least one edge in $\mathbf{H}$.
\item\label{hypnousB} $|V(M_\XXX)\setminus \XXX|\leq 1$, where $M_\XXX=\{CD\in M\: :\: C,D\in
\neighbor_{\mathbf{H}}(\XXX)\}$.
\item\label{hypnousC} All clusters of $\XXX\cap\mathcal L$ and all but at most
$3\sqrt\sigmaX N$ clusters $C\in \XXX\setminus \mathcal L$ satisfy
\begin{equation*}
\wdeg_{\mathbf H}(C,V(M))\geq \wdeg_{\mathbf{G}}(C, \overV)-3\alphaX n\;.
\end{equation*}
To see this, recall that by the assertion of Lemma~\ref{prop_TutteType} we have
that\Referee{(77)}
$$\sum_{D\in\neighbor_\mathbf{H}(C)\cap V(M)}\wdeg_\mathbf{H}(C,D)\ge
\sum_{D\in\neighbor_\mathbf{H}(C)}\wdeg_\mathbf{H}(C,D)-2\alphaX n$$for each
$C\in \XXX$. Thus the assertion follows from Lemma~\ref{lem:por}.
\item\label{hypnousD} Each edge of $M$ intersects $\mathcal L$.
\end{enumerate}

We partition $\mathcal D_A=\mathcal{T}_F\dcup \mathcal{T}_A$, where $\mathcal{T}_F$ are the internal shrubs and by~$\mathcal{T}_A$ are the end-shrubs of $\mathcal D_A$. Recall that~$\mathcal{D}_B$ contains only end-shrubs and  that $v(\mathcal D_B)\leq v(\mathcal T_A)$. 
We shall assume that $\mathcal{T}_F\cup \mathcal{T}_A\cup \mathcal{D}_B$ is $c_\mathbf{U}$-balanced, otherwise~$T\subseteq G$ by Lemma~\ref{prop:UnbalancedCase}.\Referee{(78)}
 
As we shall show shortly, the proof of Lemma~\ref{prop:iteration-Reg} follows from the following three statements, proofs of which are postponed to subsequent sections.
\begin{lemma}\label{prop-caseI}
If we have Case~I, then~$T\subseteq G$.
\end{lemma}

\begin{lemma}\label{prop-CaseIIpart1}
If we have Case~II, then~$T\subseteq G$, or for any two clusters $A,B\in \XXX\cap \mathcal L$ that are adjacent in $\mathbf{H}$, there exists a matching $M_A\subseteq M-\{A,B\}$  such that $M_A$ and $V_A=\bigcup V(M_A)$ satisfy the following properties.
\begin{enumerate}[label={$(\roman{*})$}]
\item\label{aSS:1} $\wdeg_{\mathbf H}(A,C), \wdeg_{\mathbf H}(A,D)>(1-2\etaX)s$ and $\wdeg_{\mathbf H}(A,CD)>(2-3\etaX)s$, for all $CD\in M_A$.
\item\label{aSS:2} $\wdeg_{\mathbf H}(A,V(M_A))\ge  (1-8\etaX) k$.
\item\label{aSS:3} $(1-8\etaX)k\le |V_A|\le k$.
\item\label{aSS:4} $V(M_A)\subseteq \XXX$.
\item\label{aSS:5} If $v(\mathcal D_B)\ge \sqrt[4]{\zeta}k$, then $\wdeg_{\mathbf H}(B,V(M_A))\ge (1-9\etaX)k$. 
\item\label{aSS:6} If $v(\mathcal D_B)< \sqrt[4]{\zeta}k$, then there exists a matching $M_B\subseteq M- (V(M_A)\cup\{ A, B\})$ such that $|M_B|\le  \sqrt[4]{\zeta}N$ and $v(\mathcal D_B)+\sigmaX k\le \wdeg_{\mathbf H}(B,V(M_B))\le v(\mathcal D_B)+\sigmaX k +2s$.
\item\label{aSS:7} $|V_A\cap L|\ge \frac 12 |V_A|$.
\end{enumerate}
\end{lemma}

\begin{lemma}\label{prop-CaseIIpart2}
Suppose we have Case~II and let $A,B\in \XXX\cap \mathcal L$, $AB\in E(\mathbf{H})$. Suppose that $M_A$, $M_B$ and $V_A$ satisfy~\ref{aSS:1}--\ref{aSS:7} from Lemma~\ref{prop-CaseIIpart1}. (For convenience, we take $M_B=\emptyset$ if the assumption of Lemma~\ref{prop-CaseIIpart1}~\ref{aSS:6} is not satisfied.)\Referee{(79)} If $|e_{G_\gamma}(V_A,V\setminus V_A)|\ge \frac{\kappa n^2}{2}$, then~$T\subseteq G$. 
\end{lemma}

Given Lemmas~\ref{prop-caseI}--\ref{prop-CaseIIpart2}, Lemma~\ref{prop:iteration-Reg} follows. 
Indeed, we get that $\mathcal T_{k+1}\subset G$, or $e_{G_\gamma}(V_A,V\setminus V_A)< \kappa^2n^2/2$, with $V_A$ from Lemma~\ref{prop-CaseIIpart1}. 
In the latter case, the assertions of Lemma~\ref{prop:iteration-Reg} are fullfilled with the set $V':=V_A$.  
Indeed, by Lemma~\ref{prop:iteration-Reg}~\ref{pr:Vrtacka} and by~\eqref{eq:edgesGvsGg}, we have $e_G(V_A,V\setminus V_A)\le e_G(\overV, V\setminus \overV)+e_{G_\gamma}(V_A,V\setminus V_A)+e(G)-e(G_\gamma)\le 2\lambda k^2+\kappa n^2/2\ll \rho k^2$.

\subsection{Proof of Lemma~\ref{prop-caseI}}
\label{ssec_CaseI}
We shall partition each cluster $C\in V(\mathbf H)$ so that the
partition defines two disjoint sets $V^F,V^B\subset V(G)$.
The embedding $\varphi: V(T)\rightarrow  V(G) $ of~$T$ will be defined in three phases. In the first phase, we shall embed the
subtree $T'=T[W_A\cup W_B\cup V(\mathcal T_F\cup \mathcal T^M_B)]$, where~$\mathcal T_B^M\subseteq \mathcal D_B$ will be defined later. 
 The trees~$\mathcal T_F$ will be embedded in~$V^F$ and the trees~$\mathcal T^M_B$ in~$V^B$.
In the second phase, we shall embed $\mathcal T_B^L=\mathcal D_B\setminus \mathcal T_B^M$ in~$V^B$. In the last phase, we shall embed~$\mathcal T_A$
 in~$V(G)$. From now on, we write~$\varphi$ for the partial embedding (at the current stage) of~$T$.
   
The difference between the present proof of Theorem~\ref{thm_main} and its approximate version Theorem~\ref{thm_PiguetStein} is that in the proof of Theorem~\ref{thm_main} we have to fight to gain back small loses caused by the use of the Regularity Lemma. However, this is not necessary when we have the matching structure of Case~I. Indeed, we can reduce this situation to the ``approximate version'', i.e., to a setting of similar nature as in Theorem~\ref{thm_PiguetStein}.

\paragraph{Preparation.} We partition each cluster $C\in V(\mathbf H)$ into sets~$C^F$ and~$C^B$ in an arbitrary way so that $|C^F|=(1-y)|C|$ and $|C^B|=y|C|$, where
\begin{equation}\label{eq:def_y}
y=\frac
{v(\mathcal T_A\cup \mathcal D_B)}{k}\cdot\frac 1{1+\tfrac{\alphaX}4}+\sigmaX\geq \frac
{2v(\mathcal D_B)}{k}\cdot\frac 1{1+\tfrac{\alphaX}4}+\sigmaX.
\end{equation}
Note that\Referee{(83)}
\begin{align}\label{eq:1-ystronger}
1-y&\ge \frac{v(\mathcal T_F)}k+\frac{\alphaX}{8}\cdot\frac{v(\mathcal T_A\cup \mathcal D_B)}k-\lambda\\
\label{eq:1-y}&\ge
\frac{v(\mathcal T_F)}k-\lambda
\;.
\end{align}
Set\RefereeX{(80)}{In Lemma~\ref{prop_TutteType}CaseI, newly $M$ is made disjoint from $A$ and $B$. Consequently, the matching $\tilde{M}$ is indeed not needed. However, such a change cannot be done in CaseII (which is a subject of Lemma~\ref{prop-CaseIIpart1} and Lemma~\ref{prop-CaseIIpart2})}
\begin{align*}
&V^B=\bigcup_{C\in V(\mathbf{H})}C^B\;,\quad  
V^F=\bigcup_{C\in V(\mathbf{H}}C^F \;, \\
 &M^B=V^B\cap \bigcup V(M)\; ,\quad   M^F=V^F\cap \bigcup V(M)\;\mbox{, and }
 \quad L^B=V^B\cap \bigcup ( \mathcal L^*\setminus \{A,B\})\;\mbox{.}
\end{align*}
Observe that~\eqref{eq:def_y} gives $y\in (\sigmaX,1-\sigmaX)$. Indeed, the
lower bound is trivial and the upper bound follows from $\frac{1}{1+\frac\alphaX4}< 1-2\sigmaX$. 

Let $\mathcal{T}^M_B\subseteq \mathcal{D}_B$ be a maximal subject to
\begin{equation}\label{eq:defTBM}
\sum_{t\in \mathcal{T}^M_B}v(t)\leq \wdeg_{\mathbf H}(B,M^B)-{\sigmaX n}.
\end{equation}
\Referee{(87)} Let $\mathcal{T}_B^L=\mathcal{D}_B\setminus \mathcal{T}_B^M$. 
From the maximality, we have
\begin{equation}\label{eq:defTBMcomp}
\sum_{t\in \mathcal{T}^M_B}v(t)\geq \wdeg_{\mathbf H}(B,M^B)-{\sigmaX n}-\tau k\quad \mbox{or} \quad \mathcal{T}_B^L=\emptyset\;.
\end{equation}

\smallskip
We now proceed with the three-phase embedding outlined above.

\paragraph{Phase 1 of the embedding.} 
Let~$A'\subseteq A$ be the set of typical vertices w.~r.~t.\ all but at most~$\beta N$ sets $C\in V(M)$ and
let ~$B'\subseteq B$ be the set of typical vertices  w.~r.~t.~$L^B$. From Fact~\ref{fact:zap},\Referee{(81)}
\begin{equation}\label{eq:A'B'}
\min\{|A'|, |B'|\}\ge
(1-\sqrt{\alpha})s\;.
\end{equation}

We use Lemma~\ref{lem:Embedding-3} to embed the tree $T'=T[W_A\cup W_B\cup
V(\mathcal T_F\cup \mathcal T^M_B)]$ with the following setting. 
The cluster graph is $\mathbf
H$, with $AB\in E(\mathbf{H})$ and $A'\subseteq A$,  $B'\subseteq B$. 
 The set~$\mathcal Z$ is empty.  The tree $T'$ has a
$\tau$-fine partition $(W_A, W_B, \mathcal T_F, \mathcal T^M_B)$. 
We have disjoint sets
$M^F\dcup M^B\dcup \emptyset\subseteq \bigcup V(M)$. The sets $M^F$, $M^B$ and $\emptyset$ play the roles of $V^X$, $V^Y$, and $V^\mathcal{Z}$ from Lemma~\ref{lem:Embedding-3}.\Referee{(82)}
If $\mathcal T_F$ is
$c_{\mathbf{U}}/2$-balanced, we set $\mathcal
D_2=\mathcal T_F$, $\mathcal D_1=\mathcal D_3=\emptyset$ and $M_X=M$. If $\mathcal T_F$ is not
$c_{\mathbf{U}}/2$-balanced, we set $\mathcal
D_1=\mathcal T_F$, $\mathcal D_2=\mathcal D_3=\emptyset$, and
$M_X=\emptyset$. \RefereeX{(89)}{preparation. The main change is one page below.} In particular, note that
\begin{equation}\label{eq:vyhled}
\varphi(V(T')\setminus V(\mathcal T_B^M))\cap M^B=\emptyset\;.
\end{equation}

\def\LemmaEmbedT{{\mathrm{L\ref{lem:Embedding-3}}}}
We now verify the assumptions of Lemma~\ref{lem:Embedding-3}, where we use $d_\LemmaEmbedT=\gamma, \xi_\LemmaEmbedT=\sigmaX,\epsilon_\LemmaEmbedT=\alpha$. The parameters $0<\betaX\ll \sigmaX\ll \gamma<1$ and $\tau,s$ satisfy
$\tau/s<\betaX<\sigmaX^2 \gamma/400$. The bound~\eqref{eq:A'B'} guarantees that $A'$ and $B'$ have sizes as required by the lemma.
Condition~\ref{emb3-balanced} follows from the way $V^F$
and $V^B$ were defined and Condition~\ref{emb3-sign} holds as $y\in
(\sigmaX,1-\sigmaX)$. Conditions~\ref{emb3-AZ} and~\ref{emb3-Z} hold
trivially. Condition~\ref{emb3-V^Y} follows from~\eqref{eq:defTBM}. If $\mathcal T_F$ is $c_{\mathbf{U}}/2$-balanced, Condition~\ref{emb3-V^X} holds
trivially and for Condition~\ref{emb3-MX} observe that\Referee{(83)}
\begin{align*}
\wdeg_{\mathbf H}(A,M^F)&\ge (1-y)(\wdeg_{\mathbf H}(A, V(M))-\betaX n \;\qquad\text{[by~\eqref{eq:1-y} and Case~I\ref{hippiA}]}\\
&\ge v(\mathcal T_F)-3\sqrt{\sigmaX }n\ge v(\mathcal T_F)-\tfrac {c_{\mathbf{U}}^2}{4}k+\sigmaX n\;.
\end{align*}

As $\mathcal T_A\cup\mathcal T_F\cup \mathcal D_B$ is\Referee{(84)} $c_{\mathbf{U}}$-balanced we get that if $\mathcal T_F$ is not $c_{\mathbf{U}}/2$-balanced then $\mathcal T_A\cup \mathcal D_B$ is
$c_{\mathbf{U}}/2$-balanced. Condition~\ref{emb3-MX} holds trivially and for
Condition~\ref{emb3-V^X} observe that\Referee{(85)}
\begin{align*}
\wdeg_{\mathbf H}(A,M^F)&\ge (1-y)(\wdeg_{\mathbf H}(A, V(M))-\betaX n\;\qquad\text{[by~\eqref{eq:1-ystronger} and Case~I\ref{hippiA}]}\\
&\ge v(\mathcal T_F)+v(\mathcal T_A\cup \mathcal D_B)\frac
{\zeta}{8}-3\sqrt{\sigmaX} n\ge v(\mathcal T_F)+\sigmaX n\;.
\end{align*}

\paragraph{Phase 2 of the embedding.} Phase~2 is skipped when $\mathcal{T}_B^L=\emptyset$. We label the shrubs of~$\mathcal{T}_B^L$ as $t_1,\ldots, t_{|\mathcal{T}_B^L|}$. In step $i\geq 1$,
we define the embedding for the shrub~$t_i$ in a suitable edge $CD\in
E(\mathbf{G})$. Set $U_i=\varphi(V(\mathcal T_F\cup \mathcal T_B^M)\cup \bigcup_{j<i}V(t_j)).$ Let
$x_i\in W_B$ be the parent of the root of the shrub~$t_i$. The
vertex~$\varphi(x_i)$ is typical w.~r.~t.\ $L^B$ and hence
by~\eqref{eq:WF},~\eqref{eq:def_y} and~\eqref{eq:defTBMcomp},\Referee{(87)} we have\RefereeX{(86)}{Indeed, $2\alpha$ changed to $\alpha$.}
\begin{align*}
\deg(\varphi(x_i),L^B)&\geq \wdeg_{\mathbf H}(B,L^B)-\betaX n\ge \wdeg_{\mathbf H}(B,M^B\cup L^B)-\wdeg_{\mathbf H}(B,M^B)-\betaX n\\
&\geq v(\mathcal D_B)+\tfrac{\zeta k}4-v(\mathcal T_B^M)-\sigmaX n-\tau k-2\betaX n\geq
v(\mathcal T_B^L)+\sigmaX n\; .
\end{align*}
Thus there is a cluster $C\in \mathcal{L}^*$ with $$|\neighbor(\varphi(x_i))\cap
C\setminus U_i|\geq \frac {\sigmaX n}N\geq \frac {\betaX
s+\tau}{\gammaX-2\betaX}\; .$$ 
From the definition of $\mathcal{L}^*$,~\eqref{eq:def_y}, and~\eqref{eq:vyhled}\Referee{(88)} we obtain\Referee{(89)}
$$\wdeg_{\mathbf H}(C,V^B\setminus U_i)\geq \wdeg_{\mathbf H}(C,V^B) -|\varphi(V(\mathcal D_B))\cap U_i|\geq  \tfrac{\sigmaX k}4\; .$$ Therefore there
is a cluster $D\in\neighbor_{\mathbf H}(C)$ with $|D\setminus U_i|\geq \frac {\betaX
s+\tau}{\gammaX-2\betaX}$. We use Lemma~\ref{lemma_Embedding1} to embed
$t_i$ in $(C\cup D)\setminus U_i$ so that the root of the shrub~$t_i$ is mapped to $\neighbor(\varphi(x_i))\cap C\setminus U_i$.

\paragraph{Phase 3 of the embedding.} 
We label $\mathcal T_A$ as $t_1, \dots, t_{|\mathcal T_A|}$. In
step $i= 1, \dots, |\mathcal T_A|$, we define the embedding for the shrub~$t_i$. Let
$x_i\in W_A$ be the parent of the root~$r_i$ of~$t_i$. Set $U_i=\varphi(V(\mathcal T_F\cup
\mathcal D_B)\cup \bigcup_{j<i}V(t_j))$. For an edge $CD\in M$ with $C\in \neighbor_{\mathbf H}(A)$ we define
\[
\Upsilon^i_{CD}=\min\{|\neighbor(\varphi(x_i))\cap C\setminus U_i|,|D\setminus
U_i|\}\;.\]
By Lemma~\ref{lemma_Embedding1}, the shrub~$t_i$ can be embedded in unused
vertices of an edge $CD\in M$ so that~$r_i$ is mapped to a neighbor of $\varphi(x_i)$, whenever $CD$ satisfies
$\Upsilon^i_{CD}\ge \sigmaX s$. If $\mathcal{T}_F\cup \mathcal{D}_B$ is
$\tfrac{c_{\mathbf{U}}}2$-balanced then  by~\eqref{eq:ladvi} we have\Referee{(90)}
$$ \sum_{\substack{CD\in  M\\C\in\neighbor_{\mathbf H}(A)}}\max\{|C\cap U_i|,|D\cap U_i|\}\le v(\mathcal T_A)+v(\mathcal{T}_F\cup \mathcal{D}_B)-\sum_{t\in\mathcal{T}_F\cup \mathcal{D}_B}|t_\ominus|\le v(\mathcal T_A)+v(\mathcal{T}_F\cup \mathcal{D}_B)-\tfrac{c_\mathbf{U}^2k}4\;.$$
By Fact~\ref{fact:zap} we have
\begin{align*}
\sum_{\substack{CD\in M\\C\in\neighbor_{\mathbf H}(A)}} \Upsilon^i_{CD}
\geq &
\sum_{\substack{CD\in  M\\C\in\neighbor_{\mathbf H}(A)}}\left(|\neighbor(\varphi(x_i))\cap
C|-\max\{|C\cap U_i|,|D\cap U_i|\}\right)\\
\ge & \wdeg_{\mathbf H}(A,V(M))-2\sqrt{\betaX}n-(v(\mathcal T_F\cup
\mathcal D_B)-\tfrac{c_\mathbf{U}^2k}4)-v(\mathcal T_A) \geq  \sigmaX n\;.
\end{align*}
If $\mathcal{T}_F\cup \mathcal{D}_B$ is $\tfrac{c_{\mathbf{U}}}2$-unbalanced,
then $\mathcal{T}_A$ is $\tfrac{c_{\mathbf{U}}}2$-balanced. Then by~\eqref{eq:ladvi}, \Referee{(91)}
$\max\{|V(\mathcal T_A)\cap T_\oplus|,|V(\mathcal T_A)\cap T_\ominus|\}\leq
v(\mathcal T_A)-(\tfrac{c_{\mathbf{U}}}2)^2 k$. We get
\begin{align*}
\sum_{\substack{CD\in M\\C\in\neighbor_{\mathbf H}(A)}} \Upsilon^i_{CD}
\geq &
\sum_{\substack{CD\in  M\\C\in\neighbor_{\mathbf H}(A)}}\left(|\neighbor(\varphi(x_i))\cap
C|-\max\{|C\cap U_i|,|D\cap U_i|\}\right)\\
\ge & \wdeg_{\mathbf H}(A,V(M))-2\sqrt{\betaX}n-v(\mathcal T_F\cup
\mathcal D_B)-(v(\mathcal T_A)-\tfrac{c_\mathbf{U}^2k}4) \geq \sigmaX n\;.
\end{align*}
In both cases, there is an edge $CD\in M$ with $\Upsilon^i_{CD}\ge \sigmaX s$.

\def\LPDV{{\mathrm{L\ref{lem:Embedding-3}}}}
\subsection{Proof of Lemma~\ref{prop-CaseIIpart1}}\label{ssec_CaseIIp1}
Let~$\tilde M\subset M$ be the minimum matching covering clusters~$A$ and~$B$.
 We claim that
\begin{equation}\label{eq:deg(A,tildeM)}
\min\{\wdeg_{\mathbf H}(A,V(M\setminus\tilde M)),\wdeg_{\mathbf H}(B,V(M\setminus\tilde M))\}\geq k-4\alphaX n\;.
\end{equation}
As $A,B\in \XXX\cap \mathcal L$, $\min\{\wdeg_{\mathbf G}(A, \overV),\wdeg_{\mathbf G}(B, \overV)\}\ge k-\sqrt{\sigmaX} n$. From Case~II~\ref{hypnousC} and the fact that $|V(\tilde M)|\le 4$,~\eqref{eq:deg(A,tildeM)} follows. \RefereeX{!}{additional correction here}
 
The proof of \emph{(i)}--\emph{(vi)} corresponds to Lemma~6.11 from~\cite{Z07+}. The hypotheses of \cite[Lemma~6.11]{Z07+} and the present Lemma~\ref{prop-CaseIIpart1} are almost identical. We describe the correspondence and slight differences. Our Case~II~\ref{hypnousB} implies hypothesis given by Claim~6.7(3) in~\cite{Z07+}. Our Case~II~\ref{hypnousC} is weaker than the corresponding hypothesis given in Claim~6.7(2). In his proof, Zhao only uses Claim~6.7(2) to deduce that the clusters~$A$ and~$B$ have a large weight to the matching~$\mathcal M_\mathsf{Zhao}$ (which corresponds to our matching $M$). For the adaptation of the proof, we can use~\eqref{eq:deg(A,tildeM)}, instead. To help the reader comparing both statements, we indicate the differences in the notation
\begin{align}
\begin{split}\label{eq:trZh}
\sigmaX \approx 3\gamma_\mathsf{Zhao} \quad \alphaX \approx d_\mathsf{Zhao} \quad \etaX \approx \eta_\mathsf{Zhao} \quad N \approx 2k_\mathsf{Zhao}\quad s \approx N_\mathsf{Zhao}\quad k \approx n_\mathsf{Zhao} \quad n \approx 2n_\mathsf{Zhao}\\
M_A \approx \mathcal M_{\textrm{in},\mathsf{Zhao}}\quad M \approx \mathcal M_\mathsf{Zhao} \quad V_A \approx \mathcal V_{1,\mathsf{Zhao}}\quad v(\mathcal D_B) \approx f_{b,\mathsf{Zhao}}\quad M_B \approx \mathcal M_{b,\mathsf{Zhao}}\;.
\end{split}
\end{align}
The bound in~\emph{(iii)} is phrased in \cite[Lemma~6.11(iii)]{Z07+} in terms of the cluster graph however this is an inessential difference.

\medskip 
It remains to prove~\emph{(vii)}. This follows from Case~II~\ref{hypnousD} as $\bigcup \mathcal L \subseteq L$.

\subsection{Proof of Lemma~\ref{prop-CaseIIpart2}}\label{ssec_CaseIIp2}
Let~$\tilde M\subset M$ be the minimum matching covering clusters~$A$ and~$B$.
Lemma~\ref{prop-CaseIIpart2} follows from the following Lemmas~\ref{lem:cross-edges-1}, \ref{lem:cross-edges-2} and~\ref{lem:fewL-LinM*(A)}.
Set $\tilde {\mathcal{S}}=\{C \::\: CD\in M_A,\; C\notin  \mathcal L\}$,
$\tilde S=\bigcup \tilde {\mathcal{S}}$, and  $M_L=\{CD\in M_A\: :\: \{C,D\}\subseteq \mathcal L\}$. \

\begin{lemma}\label{lem:cross-edges-1}If
\setcounter{lematko66}{\value{theorem}}
$e_{G_{\gammaX}}(\tilde S, V\setminus V_A)\ge 53 \etaX n^2$, then $T\subseteq G\;$.
\end{lemma}

\begin{lemma}\label{lem:cross-edges-2} 
\setcounter{lematko67}{\value{theorem}}
If  $e_{G_{\gammaX}}(\tilde S, V\setminus V_A)< 53 \etaX
n^2$, then $T\subseteq G$ or $|M_L|\ge 9\etaX N$.
\end{lemma}

\begin{lemma}\label{lem:fewL-LinM*(A)}If
\setcounter{lematko68}{\value{theorem}}
$|M_L|\ge 9\etaX N$, then $T\subseteq G\;$.
\end{lemma}

 To prove Lemmas~\ref{lem:cross-edges-1}--\ref{lem:fewL-LinM*(A)} we use auxiliary Lemmas~\ref{lem:forref},~\ref{lem:MaMb} and~\ref{lem:T^3}.
 
\begin{lemma}\label{lem:forref}\Referee{(99)}
Let $\mathcal P\subset V(M_A)$ such that $e_{G_\gamma}(\bigcup \mathcal P,V\setminus V_A)\ge \xi n^2$. Then there exists $\xi N/2-6\sqrt{\lambda} N$ clusters $C\in \mathcal P$ with $\wdeg_{\mathbf H}(C,V(M\setminus (M_A\cup M_B)))\ge \xi n/2-2\sqrt[4]{\zeta}n$. 
\end{lemma}

Set $\mathcal{T}^{\geq 3}=\{t\in \mathcal{D}_A\: : \:
|V(t)\setminus \neighbor_T(W_A)|\geq 2\}$. For $i=1,2$, set $\mathcal{T}^i=\{t\in \mathcal T_A\::\: v(t)=i\}$. 

\begin{lemma}\label{lem:MaMb} 
\setcounter{lematko69}{\value{theorem}}
Let $M^-\subseteq M_A$ and $\mathcal T_A^*\subseteq \mathcal D_A$. 
If $v(\mathcal T_A^*)>2|M^-|s+10\etaX n$, then there exist disjoint matchings $M_a,M_b\subseteq (M_A\cup M_B)\setminus M^-$ such that 
\begin{align}\label{eq:MaMb-a}
\wdeg_{\mathbf H}(A,V(M_a))&\ge v(\mathcal D_A)-v(\mathcal T_A^*)+\sigmaX k\;,\mbox{ and}\\
\label{eq:MaMb-b}
\wdeg_{\mathbf H}(B,V(M_b))&\ge v(\mathcal D_B)+\sigmaX k\;.
\end{align}
\end{lemma}
\RefereeX{(92)}{We did not do this change, and would prefer to leave those statements as lemmas. The reason for this is that inside those ``lemmas'' there are further ``claims'' (with subordinate numbering), and we believe that this hierarchy may make it easier to understand the structure of the argument.}
\begin{lemma}\label{lem:T^3}
\setcounter{lematko610}{\value{theorem}}
If $v(\mathcal{T}^{\geq 3}) \ge  51\etaX n$ or $v(\mathcal{T}_{1}) \ge  10\etaX n$, then $T\subseteq G\;$. \RefereeX{(93)}{several changes below to this end.}
\end{lemma}
In the proof of Lemma~\ref{lem:MaMb} we use the following fact.
\begin{fact}[{\cite[Lemma~9]{PS07+}}]\label{f_numbertheoretic}
Let $J$ be a finite nonempty set, and let $a,b,\Delta>0$. For $i\in J$, let
$\alpha_i,\beta_i\in (0,\Delta]$. Suppose that
$$\frac{a}{\sum_{i\in J}\alpha_i}+\frac{b}{\sum_{i\in J}\beta_i}\le
1\;\mbox{.}$$
Then $J$ can be partitioned into two sets $J_a$ and $J_b$ so that
$\sum_{i\in J_a}\alpha_i>a-\Delta$, and
$\sum_{i\in J_b}\beta_i\ge b$.
\end{fact}

\begin{proof}[Proof of Lemma~\ref{lem:forref}]
At least $\tfrac
{\xi N}{2}$ clusters $C\in \mathcal P$ satisfy  $\wdeg_{\mathbf{G}}(C,V\setminus V_A)\ge
\tfrac{\xi n}{2}$. 
From~\eqref{eq:boB} we have that all but most $3\sqrt{\lambda}N$ clusters $C$ of $\mathcal P$ satisfy $\wdeg_{\mathbf G}(C,V\setminus \overV)<
\sqrt{\lambda}n/2$.  Therefore, all but at most $(\tfrac{\xi}{2}-3\sqrt{\lambda})N$ clusters $C\in \mathcal P$ satisfy $\wdeg_{\mathbf{G}}(C,\overV\setminus V_A)\ge \tfrac{\xi n}2-\tfrac{\sqrt{\lambda}n}2$.

\Referee{(95)}By Case~II~\ref{hypnousC} and by Lemma~\ref{prop-CaseIIpart1}~\ref{aSS:4}, all but at most $3\sqrt{\lambda}N$ clusters $C\in \mathcal P$  satisfy $\wdeg_\mathbf{H}(C,V(M))\ge \wdeg_\mathbf{G}(C, \overV)-3\zeta n$. As $\wdeg_\mathbf{H}(C,V_A)\le \wdeg_\mathbf{G}(C,V_A)$, at least $\tfrac {\xi N}2-6\sqrt{\lambda}N$ clusters $C\in \mathcal P$ satisfy $\wdeg_{\mathbf H}(C,V(M)\setminus V_A)\ge \tfrac {\xi n}{2}-4\zeta n$.
  
By
Lemma~\ref{prop-CaseIIpart1}~\ref{aSS:5},~\ref{aSS:6}, for all clusters $C\in V(\mathbf{H})$ we have $\wdeg_\mathbf{H}(C,V(M_B))\le \sqrt[4]{\zeta}n$. This proves the lemma.
\end{proof}

\begin{proof}[Proof of Lemma~\ref{lem:MaMb}]
\setcounter{theorem}{\value{lematko69}}
\setcounter{AuxiliaryCl}{0}
If $v(\mathcal D_B)<\sqrt[4]{\alphaX}k$, set $M_a=M_A\setminus M^-$ and $M_b=M_B$. From the assumption of the lemma, we have $M_B\cap M^-\subseteq M_B\cap M_A=\emptyset$.\RefereeX{(94)}{This led to elimination of $\tilde M$ at several places.}
Condition~\eqref{eq:MaMb-b} follows from Lemma~\ref{prop-CaseIIpart2}~\emph{(vi)}. For~\eqref{eq:MaMb-a}, Lemma~\ref{prop-CaseIIpart2}~\ref{aSS:2} gives
\begin{align*}
\wdeg_{\mathbf H}(A,V(M_a))&\ge \wdeg_{\mathbf H}(A,V(M_A))-2|M^-|s >k-8\etaX n-2|M^-|s\\
&> v(\mathcal D_A)-v(\mathcal T_A^*)+\sigmaX k\;.
\end{align*}
If $v(\mathcal D_B)\ge\sqrt[4]{\alphaX}k$, we get $M_a,M_b\subseteq M_A\setminus M^-$ satisfying~\eqref{eq:MaMb-a} and~\eqref{eq:MaMb-b} using Fact~\ref{f_numbertheoretic} with the following setting: $\Delta=2s, a=v(\mathcal D_A)-v(T_A^*)+2\sigmaX k, b=v(\mathcal D_B)+\sigmaX k, J=M_A\setminus M^-$ and for every $CD\in J$,  $\alpha_{CD}=\wdeg_{\mathbf H}(A,CD)$ and $\beta_{CD}=\wdeg_{\mathbf H}(B,CD)$. By~\ref{aSS:2} and~\ref{aSS:5} of Lemma~\ref{prop-CaseIIpart2},
\[
\frac{v(\mathcal D_A)-v(\mathcal T_A^*)+2\sigmaX k}{\wdeg_{\mathbf H}(A,V(M_A\setminus M^-))}+\frac {v(\mathcal D_B)+\sigmaX k}{\wdeg_{\mathbf H}(B,V(M_A\setminus M^-))}\le \frac {k-v(\mathcal T_A^*)+3\sigmaX k}{k-9\etaX n-2|M^-|s}\le 1\;,
\]as required for an application of Fact~\ref{f_numbertheoretic}.
\end{proof}

\begin{proof}[Proof of Lemma~\ref{lem:T^3}]
\setcounter{theorem}{\value{lematko610}}
\setcounter{AuxiliaryCl}{0}

\begin{AuxiliaryCl}\label{cl:manyleaves}
If $v( \mathcal{T}^1)\geq 10\etaX n$, then $T\subseteq G$.
\end{AuxiliaryCl}
\begin{proof}
By Lemma~\ref{lem:MaMb}, with $\mathcal T_{A,\LPDV}^*=\mathcal{T}^1$ and $M^-_\LPDV=\emptyset$, there exists a partition $M_a\cup M_b=M\setminus \tilde
M$ satisfying~\eqref{eq:MaMb-a} and~\eqref{eq:MaMb-b}. We embed the tree $T'=T-V(\mathcal T^1)$  using
Lemma~\ref{lem:Embedding-3} with $\mathcal D_{Y,\LPDV}=\mathcal D_B$ and $\mathcal
D_{1,\LPDV}=\mathcal D_{X,\LPDV}=\mathcal D_A\setminus \mathcal T^1$. \Referee{(104)} It is easy to check that the conditions of Lemma~\ref{lem:Embedding-3} are met. The trees of $\mathcal{T}^1$ are leaves of $T$ whose parent vertices are mapped to $A\subseteq L$,
and can be then embedded greedily. 
\end{proof}

We use Lemma~\ref{lem:forref}\Referee{(99)} with setting $\mathcal P=V(M_A)$ and $\xi=\kappa/2$, and obtain a set 
$\mathcal C\subseteq V(M_A)$ with $|\mathcal C|=20\etaX N$ such that for all $C\in \mathcal C$ we have
\begin{equation}\label{eq:bridge}
\wdeg_{\mathbf H}(C, V(M\setminus (M_A\cup M_B)))\ge \tfrac{\kappa n}{8}\;.
\end{equation} 
 Set $M^-=\{CD\in M_A\::\: \{C,D\}\cap \mathcal C\neq \emptyset\}$. 
 Let $\mathcal T_A^*\subset \mathcal T^{\ge 3}$ be maximal, subject to $v(\mathcal T_A^*)\le 50\etaX  n+\tau$. Hence,\Referee{(96)} $v(\mathcal T_A^*)\ge 50\etaX n>2|M^-|s+10\etaX n$.
 By Lemma~\ref{lem:MaMb} there are disjoint matchings $M_a,M_b\subseteq (M_A\cup
 M_B)\setminus M^-$ satisfying~\eqref{eq:MaMb-a} and~\eqref{eq:MaMb-b}.

 We use Lemma~\ref{lem:Embedding-3}  to embed the tree~$T$ with the
 $\tau$-fine partition $(W_A,W_B, \mathcal D_A, \mathcal D_B)$
 in~$G$ with the following setting: $\mathbf H_{\LPDV}=\mathbf H$, $X'_{\LPDV}=X_{\LPDV}=A$, $Y_{\LPDV}=Y'_{\LPDV}=B$, $\mathcal Z_{\LPDV}=\mathcal C$, $M_{X,\LPDV}=\emptyset$,  $M_{\LPDV}=M\setminus (\tilde M\cup M^-)$, $\mathcal D_{1,\LPDV}=\mathcal D_A\setminus \mathcal T_A^*$, $\mathcal D_{2,\LPDV}=\emptyset$, and $\mathcal D_{3,\LPDV}=\mathcal T_A^*$, $\mathcal D_{Y,\LPDV}=\mathcal {D}_B$,\Referee{(97)} $V^X_{\LPDV}=\bigcup V(M_a)$, $V^Y_{\LPDV}:=\bigcup V(M_b)$, and $V^{\mathcal Z}_{\LPDV}=\bigcup V(M\setminus (M_A\cup
 M_B))$. The parameters $\epsilon_\LPDV=\betaX, \xi_\LPDV=\sigmaX q, d_\LPDV=\gamma, \tau$, and $s$ satisfy
 $ \tau/s\le \betaX\le \sigmaX^2 q^2 \gamma/400$. Let us now verify the conditions of Lemma~\ref{lem:Embedding-3}.
 Conditions~\ref{emb3-balanced},~\ref{emb3-sign}, and~\ref{emb3-MX} trivially
 hold. Conditions~\ref{emb3-V^X} and~\ref{emb3-V^Y} follow
 from~\eqref{eq:MaMb-a} and~\eqref{eq:MaMb-b}, respectively.
 Condition~\ref{emb3-Z} follows from~\eqref{eq:bridge}. 
 
For
 Condition~\ref{emb3-AZ} first observe that\Referee{(98)}
$|\mathcal T_A^*|+|W_A|\ge |V(\mathcal T^*_A)\cap \neighbor_T(W_A)|$. This is because each vertex in $V(\mathcal T^*_A)\cap \neighbor_T(W_A)$ is either a root of a shrub, or a predecessor of a vertex in $W_A$. Moreover, each vertex in $W_A$ is a predecessor of at most one such vertex. As $\mathcal T_A^*\subset \mathcal T^{\ge 3}$,
 \begin{align*}
\wdeg_\mathbf{H}(A,\bigcup \mathcal{C})\geq (1-2\etaX)20\etaX n\geq
v(\mathcal T_A^*)/3+|W_A|+\sigmaX k\ge |V(\mathcal T^*_A)\cap \neighbor_T(W_A)|+\sigmaX  k\;.
 \end{align*}
\end{proof}

\begin{proof}[Proof of Lemma~\ref{lem:cross-edges-1}]
\setcounter{theorem}{\value{lematko66}}
\setcounter{AuxiliaryCl}{0}
Using Lemma~\ref{lem:forref}\Referee{(99)} with the setting $\mathcal P=\tilde {\mathcal S}$ and $\xi=53 \etaX$ and obtain a set 
$\mathcal C'\subseteq \tilde {\mathcal S}$ of size $18\etaX N$ such that for every $C\in \mathcal C'$,
\begin{equation}\label{eq:degS}
\wdeg_\mathbf{H}(C, M\setminus (M_A\cup M_B))\ge 25\etaX n\;,
\end{equation}
At
least $9\etaX N$ such clusters are in different edges of~$M$.
Let $\mathcal C$ be the set of such clusters. Set $
M^-=\{CD\in M_A\: :\: \{C,D\}\cap \mathcal{C}\neq \emptyset\}$ and $\mathcal
C^-=V(M^-)\setminus \mathcal C$. Note that $|M^-|=9\etaX N$ and that $\mathcal C^-\subset \mathcal L$.\Referee{(101)}

Lemma~\ref{lem:T^3} tells us that $T\subset G$ if $v(\mathcal{T}^1)\ge 10\etaX n$ or $v(\mathcal{T}^{\geq 3})\ge 51\etaX n$.  Therefore, suppose that $v(\mathcal{T}^1)< 10\etaX n$ and $v(\mathcal{T}^{\geq 3})<51\etaX n$.\Referee{(100)}

Observe that $\mathcal{D}_A\setminus (\mathcal{T}^{\geq 3}\cup\mathcal T^2\cup \mathcal{T}^1)$ consists of those internal shrubs that have at most one vertex that is not adjacent to~$W_A$. Consider a shrub $t$ in $\mathcal{D}_A\setminus (\mathcal{T}^{\geq 3}\cup\mathcal T^2\cup \mathcal{T}^1)$. Any vertex in $t$ is either a predecessor of $W_A$, or the only vertex of $t$ not adjacent to $W_A$, or the only root in~$t$. Moreover,~$t$ always contains a predecessor of $W_A$, and each vertex in $W_A$ is a predecessor of at most one vertex in such shrubs. Hence, $v(\mathcal{D}_A\setminus (\mathcal{T}^{\geq 3}\cup\mathcal T^2\cup \mathcal{T}^1)\le 3|W_A|$.\Referee{(102)} Therefore
\begin{align*}
v( \mathcal{T}^2)&=
v(\mathcal D_A)-v(\mathcal T^{\geq 3})  -v(\mathcal T^1)-v(\mathcal{D}_A\setminus (\mathcal{T}^{\geq 3}\cup\mathcal T^2\cup \mathcal{T}^1))\\
&\geq \tfrac{k}2-|W_A\cup W_B|-51\etaX n-10\etaX n-3|W_A|>29\etaX n\;.
\end{align*}
Let~$\mathcal T^*_A\subseteq \mathcal T^2$ be maximal subject to $v(\mathcal T_A^*)\le 29\etaX n$. Then $v(\mathcal T^*_A)\geq 28\etaX n\ge 2|M^-|s+10\etaX n$ and that $T-V(\mathcal T_A^*)$ is a tree. 
By Lemma~\ref{lem:MaMb} there exist disjoint matchings
 $M_a,M_b\subseteq (M_A\cup M_B)\setminus  M^-$ satisfying~\eqref{eq:MaMb-a} and~\eqref{eq:MaMb-b}.

Set $B'=B$ and let $A'\subseteq A$ be the set of typical vertices w.~r.~t.\ $\bigcup
V(M^-)$. By Fact~\ref{fact:zap}~\ref{it:za1} $\min\{|A'|,|B'|\}\ge (1-\betaX)s$. We use
Lemma~\ref{lem:Embedding-3} to embed $T-V(\mathcal T_A^*)$ in\Referee{(103)} $A'\cup B'\cup \bigcup
V(M_a\cup M_b)$ with $\mathcal D_{Y,\LPDV}=\mathcal D_B$ and $\mathcal D_{1,\LPDV}=\mathcal
D_{X,\LPDV}=\mathcal D_A\setminus \mathcal T_A^*$.\Referee{(104)} It is easy to check that the conditions of Lemma~\ref{lem:Embedding-3} are met. It remains to embed~$\mathcal T_A^*$.

 Let $\tilde C\subseteq \bigcup \mathcal C$ be the
set of typical vertices w.~r.~t.\  $\bigcup V(M\setminus (M_A\cup
M_B)$. By Fact~\ref{fact:zap}~\ref{it:za1} $|\bigcup \mathcal C\setminus \tilde
C|\le \betaX n$. As the current embedding satisfies $\varphi (W_A)\subset A'$, we get for every\Referee{(105)} $x\in W_A$,
\begin{equation*}
\deg(\varphi(x), \tilde C\cup \bigcup \mathcal C^-)\ge
(1-2\etaX)2|M^-|s-2\betaX n\ge 17\etaX n\ge v(\mathcal T_A^*)/2+\sigmaX n\;.
\end{equation*} 
We map the roots of the trees in~$\mathcal T^*_A$ to $\tilde C\cup \bigcup \mathcal C^-$.
The rest of the trees in~$\mathcal T^*_A$ can be then embedded greedily using the
typicality of the vertices in $\tilde C$,~\eqref{eq:degS} and that $\bigcup
\mathcal C^-\subseteq L$. Thus, $T\subset G$ as needed.
\end{proof}

\begin{proof}[Proof of Lemma~\ref{lem:fewL-LinM*(A)}]
\setcounter{theorem}{\value{lematko68}}
\setcounter{AuxiliaryCl}{0}
The proof is similar (and actually simpler) to that of Lemma~\ref{lem:cross-edges-1} and we provide only the needed adaptations. We use~$M_L$ instead of~$M^-$. When $\mathcal T^1$ and $\mathcal T^{\ge 3}$ are small we use the property that $\bigcup V(M_L)\subseteq L$ instead of~\eqref{eq:degS} to embed greedily $\mathcal T_A^*$. 
\end{proof}

\begin{proof}[Proof of Lemma~\ref{lem:cross-edges-2}]
\setcounter{theorem}{\value{lematko67}}
\setcounter{AuxiliaryCl}{0} 
\begin{AuxiliaryCl}\label{cl-pdfd}
There exists a set~$\mathcal C\subseteq \neighbor_\mathbf{H}(A)\cap \mathcal{L}\cap \XXX$ of size $\tfrac {\kappa}{20}N$ such that for every $C\in \mathcal C$, we have $\wdeg_{\mathbf H}(C,\overV\setminus V_A)\geq \tfrac{\omegaX n}8$ and the clusters of $\mathcal C$ lie in different edges of~$M$.
\end{AuxiliaryCl}
\begin{proof} We have\Referee{(106)}
\begin{align*}
e_{G_{\gammaX}}(V_A\setminus \tilde S,\overV\setminus V_A)&\geq e_{G_\gamma}(V_A,V\setminus V_A)-e_{G_{\gammaX}}(\tilde S,
V\setminus V_A)-e_{G_{\gammaX}}(\overV,V\setminus \overV)\\
&\ge \tfrac {\kappa n^2}{2}-53\etaX n^2-\lambda k^2> \tfrac{\omegaX n^2}4 \;.
\end{align*}
Thus,~$\tfrac{\omegaX  N}8$ clusters $C$ of $V(M_A)\setminus \tilde{\mathcal{S}}$ satisfy $\wdeg_{G_\gamma}(C,\overV\setminus V_A)\geq \tfrac{\omegaX n}8$. Pick $\tfrac {\kappa N}{16}$ of them in different edges of $M$, and denote them by $\mathcal C$. As $\mathcal V_*\setminus \mathcal L$ is independent, $\mathcal C\subset\mathcal{L}$. 
Moreover, by Lemma~\ref{prop-CaseIIpart1}~\ref{aSS:4},\Referee{(107)} we have $\mathcal C\subseteq \XXX$. By Lemma~\ref{lem:por}\ref{por1}, we have $V(M_A)\subseteq N_{\mathbf H}(A)$ and thus $\mathcal C$ satisfies the assertion of the claim.
\end{proof}
For each $X\in V(\mathbf H)$, we define $M^*_{X}=\{CD\in M\::\: |\wdeg_{\mathbf H}(X,C)-\wdeg_{\mathbf H}(X,D)|\ge \etaX s\}$.
\begin{AuxiliaryCl}\label{auxZhao}
For each cluster $X\in \XXX\cap \mathcal L\cap \neighbor_{\mathbf H}(\XXX\cap \mathcal L)$, we have  $|M^*_X|<\etaX N/2$, or $T\subseteq G$. 
\end{AuxiliaryCl}
We do not prove Claim~\ref{auxZhao} here. The proof can be taken verbatim from~\cite[Lemma~6.15 (Case~1)]{Z07+}. 
There, Zhao considers two adjacent clusters $A_\mathsf{Zhao},B_\mathsf{Zhao}$ with high average degree in a matching. He shows that if for some $X\in\{A_\mathsf{Zhao},B_\mathsf{Zhao}\}$, the matching $M^*_X$ is substantial, then $T\subseteq G$. (He uses notation $\mathcal M_{unbal,\mathsf{Zhao}}\approx M^*_X$; recall~\eqref{eq:trZh} for further vocabulary). The condition of Case~II~\ref{hypnousC} is the counterpart of the property~\cite[(6.14)]{Z07+}.


Let $\mathcal C$ be given by Claim~\ref{cl-pdfd}. Set $\mathcal{D}=V(M\setminus M_A)\cap \XXX\cap \mathcal
L$. 

\begin{AuxiliaryCl}\label{cl-prdifuk} We have
$T\subseteq G$ or
$|\mathcal{D}|>\tfrac{\omegaX N}{17}$  and $e_{G_{\gammaX}}(\bigcup \mathcal{C},\bigcup \mathcal{D})\geq \tfrac{\omegaX^2n^2}{340}$.
\end{AuxiliaryCl}
\begin{proof}
For each $C\in \mathcal C$, we apply Claim~\ref{auxZhao}. We get that $|M^*_C|\le \etaX N/2$ as otherwise $T\subseteq G$ and we are done. Hence,
$\wdeg_{\mathbf H}(C,V(M\setminus (M_A\cup M^*_C)))\geq \tfrac{\omegaX
n}8-\etaX n$. 
Let \[M_C^-=\{D_1D_2\in M\::\: \wdeg_{\mathbf H}(X,D_1)<\etaX s\mbox{ or }\wdeg_{\mathbf H}(X,D_2)<\etaX s\}\;.\]\Referee{(108)}
By the definition of $M^*_C$, the weight~$C$ sends to both end-clusters of $M\setminus  M^*_C$ differs\Referee{(109)}
 by at most $\vartheta s$.
Thus, $\wdeg_{\mathbf H}(C,V(M\setminus (M_A\cup M^*_C\cup M_C^-)))\geq \tfrac{\omegaX
n}8-4\etaX n$. By Case~II~\ref{hypnousD}, all edges in $M\setminus (M_A\cup M^*_C\cup M_C^-) $ meet $\mathcal L$. The definition of $M_C^*$ tells us that   \RefereeX{(110)}{We do not agree. The correct main term is indeed roughly $\tfrac {\kappa n}{16}$, not $\tfrac {\kappa n}{8}$. More explanation was added.}
\begin{align*}
\wdeg_{\mathbf{H}}(C,\mathcal L\cap V(M\setminus (M_A\cup M^*_C\cup M_C^-)))&\ge \frac{1}{2+\etaX}\wdeg_{\mathbf{H}}(C,V(M\setminus (M_A\cup M^*_C\cup M_C^-))) \\&\ge(1-\etaX)\tfrac {\kappa n}{16}-4\vartheta n\;.\end{align*}
Case~II~\ref{hypnousB} gives that $|V(M\setminus M_C^-)\setminus \XXX|\le 1$. Therefore,
$\wdeg_{\mathbf H}(C,\mathcal{D})> \tfrac{\omegaX n}{17}$, implying $|\mathcal D|\ge \tfrac {\kappa N}{17}$. The assertion follows from the bound on $|\mathcal C|$ given by Claim~\ref{cl-pdfd}.
\end{proof}

\begin{AuxiliaryCl}\label{cl-prckofuk}
We have $T\subseteq G$ or $|M_L|\ge \tfrac {\kappa^3N}{2\cdot10^4}$.
\end{AuxiliaryCl}
\begin{proof}
Let us assume that $T\not\subset G$. In particular, the second assertion of Claim~\ref{cl-prdifuk} applies. At least $\kappa N/680$ clusters $D\in \mathcal D$ satisfy $\wdeg_{G_\gamma}(D,\mathcal C)\ge \kappa^2 n/680$.
By Claim~\ref{auxZhao}, we may assume that each of these chosen clusters satisfy $\wdeg_{G_\gamma}(D,\mathcal{C}\setminus V(M^*_D))\geq
\tfrac{\omegaX^2n}{680}-\etaX n$, as otherwise $T\subseteq G$.\Referee{(111)} By Lemma~\ref{lem:por}\ref{por1}, these clusters satisfy $\wdeg_{\mathbf H}(D,\mathcal{C}\setminus V(M^*_D))\geq
\tfrac{\omegaX^2n}{690}$. Let $\mathcal C^-=V(M)\setminus \mathcal C$. By the definition of $M^*_D$, we get $\wdeg_{\mathbf H}(D,\mathcal{C}^-\setminus V(M^*_D))\geq
\tfrac{\omegaX^2n}{690}-\etaX n>\tfrac{\omegaX^2n}{700}$. Observe that $\mathcal{C}^-\setminus V(M_L)\subseteq \mathcal{\tilde S}$.
As $|\mathcal D|\ge \tfrac {\kappa N}{17}$ we get,
 \begin{align*}
\tfrac{\omegaX^3n^2}{12\cdot10^3}< e_{G_\gamma}\left(\bigcup \mathcal D, \bigcup\mathcal{C}^-\right)\le e_{G_\gamma}\left(\bigcup \mathcal D, \tilde S\right)+e_{G_\gamma}\left(\bigcup \mathcal D,  V(M_L)\right)\le 53\etaX n^2 +|M_L|sn\;,
 \end{align*}
 implying $|M_L|\ge \tfrac{\kappa^3N}{2\cdot 10^4}$.
 \end{proof}
Claim~\ref{cl-prckofuk} gives the statement of the lemma (recall that $\kappa\gg\vartheta$).
 \end{proof}

This finishes the proof of the Lemma~\ref{prop:iteration-Reg}.

\section{Proof of Lemma~\ref{prop:EC-obecna} (Extremal case)}\label{sec_Extremalcase}
\def\Sstrong[#1]{S^{#1}_0}
\def\Sweak[#1]{S^{#1}}
\def\Supper[#1]{S^{#1}}
\def\Supptilde[#1]{\tilde S^{#1}}

Let $c_{\mathbf{E}}$ be sufficiently small compared to $q$. Given
$\sigma\in(0,c_{\mathbf{E}}]$, let $\beta$ and $\gamma$ be chosen so that
$\beta\ll \gamma\ll \sigma$. 
Given a
$(\beta,\sigma)$-extremal partition\Referee{(112)} $V=V_1\dcup \ldots\dcup V_\ell\dcup \tilde{V}$
we show that $\mathcal{T}_{k+1}\subset G$, or there exists a set $Q\subset \tilde{V}$ satisfying Properties~\ref{pr:velkyQ}--\ref{pr:Qizo} of Lemma~\ref{prop:EC-obecna}.

The proof of
Lemma~\ref{prop:EC-obecna} is split into two statements, Lemma~\ref{prop_ECFewLeaves} and Lemma~\ref{prop_ECManyLeaves}, according\Referee{(113)} to the number of leaves of the tree $T\in \mathcal{T}_{k+1}$ considered.
\begin{lemma}\label{prop_ECFewLeaves}
Let $T\in \mathcal{T}_{k+1}$ be a tree that has at most $60\gamma k$ leaves. Suppose that $G$ admits a $(\beta,\sigma)$-extremal partition $V=V_1\dcup \ldots\dcup V_\ell\dcup \tilde{V}$. Then $T\subset G$,
or there exists a set $Q\subset \tilde{V}$ satisfying Properties~\ref{pr:velkyQ}--\ref{pr:Qizo} of Lemma~\ref{prop:EC-obecna}.
\end{lemma}
\begin{lemma}\label{prop_ECManyLeaves}
Let $T\in \mathcal{T}_{k+1}$ be a tree that has more than $60\gamma k$ leaves. Suppose that $G$ admits a $(\beta,\sigma)$-extremal partition $V=V_1\dcup \ldots\dcup V_\ell\dcup \tilde{V}$. Then $T\subset G$.
\end{lemma}
Lemma~\ref{prop:EC-obecna} follows Lemmas~\ref{prop_ECFewLeaves} and~\ref{prop_ECManyLeaves}. The proofs of these lemmas occupy Sections~\ref{ssec_FewLeavesEmbedding}, and~\ref{ssec_ManyLeavesEmbedding}. First however, we establish some basic properties of a
$(\beta,\sigma)$-extremal partition.
Throughout this section we write $m=\ci(\tfrac{n}k)$ for the integer closest to~$\tfrac{n}k$.\Referee{(27)} The sets $V_i$, $i\in [\ell]$ are called {\em clumps}.

Suppose that $G$ admits a $(\beta,\sigma)$-extremal partition
$V=V_1\dcup \ldots\dcup V_\ell\dcup \tilde{V}$. Then $\ell\le m$.
\begin{lemma}\label{l:e}
For each $i\in[\ell]$ the following holds.\Referee{(114)}
\begin{enumerate}
\item For all but at most $\sqrt{\beta}k$ vertices
 $v\in V_i\cap L$, we have that $\deg(v,V_i)\ge k-\sqrt{\beta}k$.
\item For all but at most $2\sqrt{\beta}k$ vertices
 $v\in V_i\cap S$, we have that $\deg(v,V_i\cap L)\ge |V_i\cap
L|-\sqrt{\beta}k$.
\item For all but at most $\sqrt{\beta}k$ vertices
 $v\in V\setminus V_i$, we have that $\deg(v,V_i)< \sqrt{\beta}k$.\RefereeX{!}{This was added because of (119)}
\end{enumerate}
\end{lemma}
\begin{proof}
\begin{enumerate}
\item Let $U=\{v\in V_i\cap L\: :\: \deg(v,V_i)< k-\sqrt{\beta}k\}$. Since
every vertex $v\in U$ sends at least $\sqrt{\beta}k$ edges outside $V_i$, we
deduce from $e(V_i,V\setminus V_i)<\beta k^2$ that $|U|\le
\sqrt{\beta}k$.
\item Let $W=\{v\in V_i\cap S\: :\: \deg(v,V_i\cap L)<|V_i\cap L|-\sqrt{\beta}k\}$. From
\begin{align*}
e(V_i\cap L, V_i\cap S)&> |V_i\cap L|k-|V_i\cap L|^2-\beta k^2>|V_i\cap
L||V_i\cap S|-2\beta k^2\; \mbox{, and}\\
e(V_i\cap L,V_i\cap S)&=e(V_i\cap L,W)+e(V_i\cap L,V_i\cap S\setminus
W)\\
&\le(|V_i\cap L|-\sqrt{\beta}k)|W|+|V_i\cap L|(|V_i\cap S|-|W|)\\
& =|V_i\cap L||V_i\cap S|-\sqrt{\beta}k|W|
\end{align*}
we infer that $|W|<2\sqrt{\beta}k$.
\item Let $Z=\{v\in V\setminus V_i:\deg(v,V_i)\ge \sqrt{\beta}k\}$. We have $$\beta k^2>e(V_i,V\setminus V_i)\ge\sum_{v\in Z}\deg(v,V_i)\ge |Z|\sqrt{\beta}k\;,$$
which proves the statement.
\end{enumerate}
\end{proof}

For each $i\in[\ell]$, we set $L^i=\{u\in L\: :\: \deg(u,V_i)>
(1-\tfrac{\gamma}4)k\}$. For every $A\subset V_i$,\Referee{(115)} Lemma~\ref{l:e}(i) and the assumption $|V_i\cap L|\ge (\frac12-\beta )k$ give that\RefereeX{(116)}{We prefer to leave it there as this is then compactly referred to in the proof of Lemma~\ref{lemma_consideronlysmalldiscrepancy}.}
\begin{equation}\label{eq:Am}
|L^i|\geq (1-\tfrac{\gamma}2 )\frac k2\quad
\mbox{and}\quad \delta(L^i,A)\geq |A|-\frac{\gamma k}{2}\;.
\end{equation}

\RefereeX{(120)}{By eliminating what used to be Lemma 7.7 (see our explanation to (139)), we got rid of $S^i_\sharp$. Then, we unified the definition of two symbols $S^j$ (introduced in what used to be Lemma~7.8 and is Lemma~\ref{lemma_StartDiana} under the current numbering) and $S^j_\heartsuit$ (Lemma~7.9 in the old numbering, Lemma~\ref{lemma_StartDiana} in the new numbering) into one symbol: $\Sweak[j]$. The original symbol $S^j_\diamond$ was changed to $\Sstrong[j]$}
For each $i\in[\ell]$, we set $\Sstrong[i]=\{v\in S\cap V_i\: :\:
\deg(v,L^i)>|L^i|-\tfrac{\gamma k}2\}$. As the sets $V_i$ are pairwise disjoint, so are the sets $\Sstrong[1],\Sstrong[2],\ldots,\Sstrong[\ell]$.
Any vertex $v\in S\cap V_i$ with $\deg(v, V_i\cap L)\ge |V_i\cap L|-\sqrt{\beta}k$ satisfies $\deg(v, L^i)\ge |V_i\cap L^i|-\sqrt{\beta}k-|(V_i\cap L)\setminus L^i|\ge |L^i|-\sqrt{\beta}k-|(V_i\cap L)\setminus L^i|-|L^i\setminus V_i|$. \Referee{(118)} Therefore by Lemma~\ref{l:e}(i),(iii) any such vertex~$v$ belongs to $\Sstrong[i]$. \Referee{(119)}By Lemma~\ref{l:e}(ii) and by~\eqref{eq:Am} we have 
\begin{equation}\label{eq:LSdiamond}|L^i\cup \Sstrong[i]|\geq (1-\tfrac{\gamma}2) k\;.\end{equation}

\smallskip
The next lemma allows to discard trees with substantial discrepancy from further considerations.
\begin{lemma}\label{lemma_consideronlysmalldiscrepancy}
Suppose that $G$ admits a $(\beta,\sigma)$-extremal partition
$V=V_1\dcup \ldots\dcup V_\ell\dcup \tilde{V}$.  Then each tree $T\in \mathcal{T}_{k+1}$ with
discrepancy at least $2\gamma k$ is a subgraph of $G$.
\end{lemma}
\begin{proof}
Fix $i\in[\ell]$.\Referee{(121)} Choose $L^*\subseteq L^i$ with $|L^*|=(1-\tfrac{\gamma}2)\frac k2$, and set
$S^*=(L^i\cup \Sstrong[i])\setminus L^*$. By~\eqref{eq:LSdiamond}, $|S^*|\geq
(1-\tfrac{\gamma}2 )\frac k2$. Using~\eqref{eq:Am} and the definition of $\Sstrong[i]$,\Referee{(122)} we have\Referee{(123)}
$$\min\{\delta(L^*,S^*),\delta(S^*,L^*),\delta(L^*,L^*)\}\geq
(1-\tfrac{3\gamma}2)\frac{k}2\; .$$ Take a semi-independent partition
$(U_1,U_2)$ of $T$ witnessing that $\disc(T)\ge 2\gamma k$. We apply
Fact~\ref{fact_embeddingsemiindependent} to embed $T$ in $G$ using the sets $L^*$ and $S^*$.
\end{proof}

\begin{lemma}\label{lemma_deficientL^idecompose}
\begin{enumerate}[label={$(\roman{*})$}]
\item\label{part_Ldisjoint} The sets $\{L^i\}_{i\in[\ell]}$ are mutually disjoint, or $\mathcal{T}_{k+1} \subset G$.
\item\label{part_two} Suppose that $\tilde{V}=\emptyset$. If there exists a vertex $u\in L\setminus (\bigcup_i L^i)$, then
$\mathcal{T}_{k+1}\subset G$.
\end{enumerate}
\end{lemma}
\begin{proof}
For each $i\in[\ell]$, fix a set $A_i\subset L^i$ of size
$(\tfrac12-\tfrac{\gamma}4)k$, and set $B_i=(L^i\cup \Sstrong[i])\setminus
A_i$. By~\eqref{eq:Am},~\eqref{eq:LSdiamond} and the definition of the set $\Sstrong[i]$ we have
\begin{equation}\label{eq:newly}
\delta(G[A_i,B_i])\ge \left(\frac12-\frac{5\gamma}4\right)k\;. 
\end{equation}

\noindent\emph{Proof of Part~\ref{part_Ldisjoint}.}~~
Suppose that there exist distinct indices
$i,j\in[\ell]$ and a vertex $u\in L^i\cap L^j$. Let $T\in\mathcal{T}_{k+1}$
be arbitrary. By Lemma~\ref{l:e}(iii), we have\Referee{(124)}
\begin{equation}\label{eq:interSm}
|L^i\cap L^j|<\frac{k}{100}\;.
\end{equation}

By Lemma~\ref{lemma_consideronlysmalldiscrepancy} we can assume in
the following that $\disc(T)< 2\gamma k$.  By Fact~\ref{fact_fullsubtrees} there exists a
full-subtree $\tilde{T}\subset T$ rooted at a vertex $r$ such that
$v(\tilde{T})\in [\tfrac{k}6,\tfrac{k}3]$. We map $r$ to $u$, and embed the tree $\tilde{T}$ in $G[A_i,B_i]$ greedily. This is possible since
$$\max\{|T_\oplus\cap V(\tilde{T})|,|T_\ominus\cap
V(\tilde{T})|\}<\tfrac{v(\tilde{T})}2+2\gamma k\le \tfrac{k}6+2\gamma k$$ by
Fact~\ref{fact_cutnodiscrepancy}, and the graph $G[A_i,B_i]$ satisfies~\eqref{eq:newly}.\Referee{(125)} It remains to embed the tree~$T-\tilde T$.\Referee{(126)} By Fact~\ref{fact_cutnodiscrepancy}, we have $\min\{|T_\oplus\cap V(T-\tilde{T})|,|T_\ominus\cap
V(T-\tilde{T})\}|>\tfrac{v(T-\tilde{T})}2-2\gamma k$, and thus
$\max\{|T_\oplus\cap V(T-\tilde{T})|,|T_\ominus\cap
V(T-\tilde{T})|\}<\tfrac{5k}{12}+2\gamma k$. We embed $T-\tilde{T}$ in $G[A_j,B_j]$ greedily (avoiding the previously used vertices of $L^i\cap L^j$;
we use~\eqref{eq:interSm} to bound the number of occupied vertices).

\noindent\emph{Proof of Part~\ref{part_two}.}~~
 Suppose that there exists a vertex $u\in L\setminus \bigcup_i L^i$. By Part~\ref{part_Ldisjoint} of the lemma, we may assume that the sets $L^i$ are pairwise disjoint. 
Let\Referee{(127)}
\begin{align*}
X_i&=\{u\in A_i\: :\: \deg(u,V_i)> (1-\tfrac{\gamma}{13m})k\}\;\mbox{,
and}\\ Y_i&=\{u\in B_i\: :\: \deg(u,L^i)>|L^i|-\tfrac{\gamma
k}{13m}\}\;\mbox{.}
\end{align*}
(In applications, we use that $\deg(u,X_i)>|X_i|-\tfrac{\gamma k}{13 m}$ for every $u\in Y_i$.) Applying Lemma~\ref{l:e}~(i)--(ii) to $L^i,\Sstrong[i],X_i$ and $Y_i$, we get that\Referee{(128)}
\begin{equation}\label{eq_clusterfilledwholeXY}
|V_i\setminus (X_i\cup Y_i)|<\tfrac{\gamma k}{6m^2} \; \mbox{.}
\end{equation}
As $X_i\subset L^i$ and $Y_i\subset \Sstrong[i]$, all the sets $X_i$ and $Y_i$
are pairwise disjoint.
Without loss of generality, we assume that $\deg(u,X_1\cup Y_1)\ge \ldots \ge \deg(u,X_m\cup Y_m)$. As $u\in L\setminus L^1$ we have
\begin{align*}
 k &\le \deg(u,L)\le \sum_{i=1}^m \deg(u,X_i\cup Y_i)+\tfrac{\gamma k}{6 m}\le (1-\tfrac\gamma2)k+\sum_{i=2}^m \deg(u,X_i\cup Y_i)+\tfrac{\gamma k}{6 m}\\
 &\le
(1-\tfrac\gamma3)k+(m-1) \deg(u,X_2\cup Y_2)\;.
\end{align*}
This yields that
\begin{equation}\label{eq:degTWo}
\deg(u,X_1\cup Y_1)\ge\deg(u,X_2\cup Y_2)\ge \frac{\gamma k}{3(m-1)}\ge 2\;.
\end{equation}

 Let $T\in \mathcal{T}_{k+1}$ be arbitrary. Analogously as
in the proof of Lemma~\ref{lemma_consideronlysmalldiscrepancy} we have
$T\subset G$ if $\disc(T)\ge\tfrac{\gamma k}{6m}$. Therefore we assume
that $\disc(T)<\tfrac{\gamma k}{6m}$. By Fact~\ref{fact_fullsubtrees}
there exists a full-subtree $\tilde{T}\subset T$ rooted at a vertex $r$ such
that $v(\tilde{T})\in [0.3k,0.6k]$. 
Let $D$ be the set of leaves of $T$ in $\neighbor_T(r)$.\Referee{(129)} We
first embed the tree $T-D$, mapping $r$ to $u$, as described below. The embedding is then
extended to an embedding of $T$ using the fact that $u\in L$.\Referee{(130)}

A {\em $2^+$-component} is a component of the forest $T-r$ of order at least
two. Let $\mathcal{C}$ be the family of all $2^+$-components. For each subfamily $\mathcal{C'}\subset \mathcal C$, we have by Fact~\ref{fact_cutnodiscrepancy}
and by the assumption $\disc(T)\le \tfrac{\gamma k}{6m}$ that
\begin{equation}\label{eq_2cutRbalanced}
\max\{|V(\mathcal{C}')\cap T_\ominus|,|V(\mathcal{C}')\cap
T_\oplus|\}<\tfrac{|V(\mathcal{C}')|}2+\tfrac{\gamma k}{12m}+1 \;
\mbox{.}
\end{equation}

By~\eqref{eq_clusterfilledwholeXY} at most $\tfrac{\gamma k}{6m}$\Referee{(131)}
vertices of the graph $G$ are not contained in $\bigcup_i (X_i\cup Y_i)$. Thus,
$\deg(u,\bigcup_i (X_i\cup Y_i))\ge (1-\tfrac{\gamma}{6m})k$. We 
assign each $2^+$-component $C\in \mathcal{C}$ an index $i_C\in [m]$ such that $C$ will be mapped to the clump $V_{i_C}$.\Referee{(132)} For each $j\in [m]$ we shall require:\Referee{(133)}
\begin{align}
\label{eq_bounddegreecluster}\deg(u,X_j\cup Y_j) &\ge \left|\{C\in
\mathcal{C}\::\: i_C=j\}\right|\;\mbox{, and}\\ \label{eq_boundToEmbedIntoVi}   \sum_{\substack{C\in\mathcal{C}\\i_C=j}} v(C)
&\le (1-\tfrac{\gamma}3)k \;\mbox{.}
\end{align}
\begin{AuxiliaryCl}\label{AC:assExists}
There exists a family $\{i_C\}_{C\in\mathcal C}$ such
that~\eqref{eq_bounddegreecluster} and~\eqref{eq_boundToEmbedIntoVi} are
satisfied.\RefereeX{(134)-(135)}{The proof rewritten completely.}
\end{AuxiliaryCl}
\begin{proof}
We order the $2^+$-components as $C_1,\ldots,C_{|\mathcal{C}|}$ so that $v(C_1)\ge v(C_2)\ge\ldots\ge v(C_{|\mathcal{C}|})$.  For $j=1,\ldots,|\mathcal C|$, take the smallest index $i\in [m]$ with the property that after assigning $i_{C_j}=i$, the properties~\eqref{eq_bounddegreecluster} and \eqref{eq_boundToEmbedIntoVi} are satisfied for the partial assignment $\{i_{C_{j'}}\}_{j'\le j}$. If for a given $j$ there exists no such value $i$ we just mark $C_j$ as unassigned and proceed with $j+1$. 

We thus need to check that actually each $2^+$-component $C_j$ was assigned. Suppose for a contradiction that $C_g$ was not. We have $v(C_1)\le 0.7k$, and for $\ell\ge 2$ we have $v(C_\ell)\le \frac k \ell$. These bounds and~\eqref{eq:degTWo} guarantee us that $C_1,\ldots,C_4$ can always be assigned; one assignment satisfying~\eqref{eq_bounddegreecluster} and \eqref{eq_boundToEmbedIntoVi} is $i_{C_1}=i_{C_4}=1,i_{C_2}=i_{C_3}=2$. Thus $g>4$, and consequently $v(C_g)\le 0.2k$.

To finish the argument, we distinguish two cases. First, assume that $\deg(u,X_1\cup Y_1)\ge 0.5k$. Since $v(C)\ge 2$ for each $C\in\mathcal C$, property~\eqref{eq_bounddegreecluster} for $j=1$ holds trivially. As $C_g$ could not be assigned with $i_{C_g}=1$,  by~\eqref{eq_boundToEmbedIntoVi} we get that $\sum_{i_C=1} v(C)> (1-\tfrac{\gamma}3)k -v(C_g)$. In particular, the number of $2^+$-components $C$ that are unassigned, or have $i_C\neq 1$ is less than $1+\tfrac{\gamma k}6$. Further, the total order of the $2^+$-components to be assigned to other clumps is at most $v(C_g)+\frac{\gamma k}3<0.4k$. Thus,~\eqref{eq_boundToEmbedIntoVi} holds trivially for $j>1$. The reason why the component~$C_g$ was not assigned is that it did not satisfy~\eqref{eq_bounddegreecluster} for any $j>1$. Hence, by~\eqref{eq_clusterfilledwholeXY} we have
\[
1+\tfrac{\gamma k}6>\sum_{j>1}\deg(u,X_j\cup Y_j)\ge k-\deg(u,X_1\cup Y_1)-\sum_{j=1}^m|V_j\setminus(X_j\cup Y_j)|\ge \frac k3\;,
\]
a contradiction with the choice of $\gamma$.

Now, consider the case that $\deg(u,X_1\cup Y_1)< 0.5k$. Then $\deg(u,X_2\cup Y_2)< 0.5k$. Observe that for $j=1,2$ we have
$\sum_{i_C=j} v(C) \ge 2\deg(u,X_j\cup Y_j)-\frac{\gamma k}{3}-v(C_g)$, as otherwise we could have assigned $i_{C_g}=j$ without violating~\eqref{eq_bounddegreecluster} and~\eqref{eq_boundToEmbedIntoVi}. For $j>2$ by similar arguments we have $\sum_{i_C=j} v(C) \ge \min\{0.7k,2\deg(u,X_j\cup Y_j)\}$. Summing these bounds, we get that
\begin{equation}\label{eq:mamvelmiradsvouzenu}
\sum_{C\in\mathcal C} v(C) \ge 2\deg(u,X_1\cup Y_1)+2\deg(u,X_2\cup Y_2)-2\frac{\gamma k}{3}-2v(C_g)+\sum_{j=3}^m  \min\{0.7k,2\deg(u,X_j\cup Y_j)\}\;.
\end{equation}
Suppose that for some $j>2$ we have $0.7k\le 2\deg(u,X_j\cup Y_j)$. Then $2\deg(u,X_1\cup Y_1)\ge 2\deg(u,X_2\cup Y_2)\ge 0.7k$, and thus
$$\sum_{C\in\mathcal C} v(C) \ge 0.7k+0.7k-2\frac{\gamma k}{3}-2v(C_g) +0.7k\ge 1.6k>k\;,
$$
where we used that $v(C_g)\le 0.2k$. This is a contradiction. Thus, we can assume that for all $j>2$, $0.7k> 2\deg(u,X_j\cup Y_j)$. Plugging into~\eqref{eq:mamvelmiradsvouzenu} we get
$$
\sum_{C\in\mathcal C} v(C) \ge \sum_{j=1}^m 2\deg(u,X_j\cup Y_j)-2\frac{\gamma k}{3}-2v(C_g)\ge 2\cdot 0.9k-2\frac{\gamma k}{3}-2\cdot 0.2k>k\;,
$$
which again gives a contradiction.
\end{proof}

We embed the tree $T-D$ as follows. Let us consider the indices $\{i_C\}_{C\in\mathcal C}$ from Claim~\ref{AC:assExists}. The vertex $r$ is mapped to $u$. For each
component $C\in \mathcal{C}$ we map its root $r_C\in V(C)\cap \neighbor_T(r)$ to one vertex from $(X_{i_C}\cup Y_{i_C})\cap \neighbor_G(u)$ (so that distinct
roots are mapped to distinct vertices). We denote the image of the root $r_C$ by
$\varphi(r_C)$. The mapping of the roots is extended to an embedding of
all $2^+$-components. This can be done greedily since each of the graphs
$G[X_i,Y_i]$ has minimum degree at least
$(\tfrac12-\tfrac{\gamma}{12m})k+1$, and we have by a double
application of~\eqref{eq_2cutRbalanced} that\Referee{(136)}
\begin{align*}
\sum_{\substack{C\in\mathcal{C}\\ \varphi(r_C)\in X_i}}|V(C)\cap
T_\oplus|+\sum_{\substack{C\in\mathcal{C}\\ \varphi(r_C)\in Y_i}}|V(C)\cap
T_\ominus|&<(1-\tfrac{\gamma}3)\tfrac{k}2+2(\tfrac{\gamma
k}{12m}+1)\le\delta(G[X_i,Y_i]) \;\mbox{, and}\\
\sum_{\substack{C\in\mathcal{C}\\ \varphi(r_C)\in X_i}}|V(C)\cap
T_\ominus|+\sum_{\substack{C\in\mathcal{C}\\ \varphi(r_C)\in Y_i}}|V(C)\cap
T_\oplus|&<(1-\tfrac{\gamma}3)\tfrac{k}2+2(\tfrac{\gamma
k}{12m}+1)\le\delta(G[X_i,Y_i]) \;\mbox{.}
\end{align*}
\end{proof}

Much of the work for proving Lemma~\ref{prop:EC-obecna} splits according to the following distinction. A $(\beta,\sigma)$-extremal partition is said to be {\em abundant} if
there exists an index $i\in[\ell]$ with $|L^i|\ge \tfrac{k+1}2$. It is called {\em deficient} otherwise.\Referee{(117)}

We now derive properties of $G$ in the deficient case. First, we
observe that $G$ is decomposed into clumps.
\begin{lemma}\label{lem:def}
Suppose that $G$ admits a $(\beta,\sigma)$-extremal deficient partition
$V=V_1\dcup\ldots\dcup V_\ell\dcup\tilde{V}$. Then $\tilde V=\emptyset$, and
$\ell=m$. Further,
\begin{align}\label{eq_rkn}
m(k+1)> n\;\mbox{.}
\end{align}
\end{lemma} 
\begin{proof}
Since the partition is deficient we have $|L\cap V_i|\le \frac k2$ for all $i\in[\ell]$. Thus by the definition of $(\beta,\sigma)$-extremality,
\Referee{(137)}
we have $|L|\le\ell\tfrac{k}2+(\tfrac12-\sigma)|\tilde V|$, and $|S|>\ell
(1-\beta)\tfrac{k}2+(\tfrac12+\sigma)|\tilde V|$. Since $|L|\ge |S|$, we infer that\Referee{(138)}
$|\tilde V|<\frac{\gamma \ell k}{4\sigma}$. This in turn implies that $\tilde V=\emptyset$.
Thus, $\ell=m$. To get the bound~\eqref{eq_rkn}, we observe that
$$n=|L|+|S|\le 2|L|=2\sum_{i=1}^m |L\cap V_i|<2m
\frac{k+1}2\;.$$
\end{proof}

Lemmas~\ref{lemma_StartDiana} and~\ref{prop_deficientCones} deal with the
deficient case. It may happen that none of the
clumps is suitable for the embedding of the tree $T\in\mathcal{T}_{k+1}$. For this
reason, we must find connecting structures that allow us to distribute parts of
$T$ to different clumps. Each lemma is used for a different type of trees.\RefereeX{(139)-(141)}{What used to be Lemma 7.7 is not needed anymore. That is, during the previous revision, parts relying on that lemma were reworked. We realized the fact that the lemma was actually abundant only now. (That brought us to crosscheck if we really need all the lemmas in the paper)}

For $j\in [m]$, set $\Supper[j]=\{v\in S\::\: \deg(v,L^j)\geq \tfrac{k}{5m}\}$.

\begin{figure}[t]
  \centering
\subfigure[Connecting structure guaranteed by Lemma~\ref{lemma_StartDiana}.]{
\includegraphics[scale=1.6]{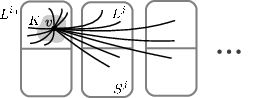}
  \label{fig_StartDiana}
}
  \hspace{.2in}
\subfigure[Connecting structure guaranteed by Lemma~\ref{prop_deficientCones}.]{
\includegraphics[scale=1.6]{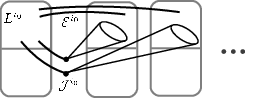}
  \label{fig_deficientCones}
}\caption{Structures in Lemma~\ref{lemma_StartDiana} and Lemma~\ref{prop_deficientCones}}
\end{figure}
\begin{lemma}\label{lemma_StartDiana}
Suppose that $G$ admits a $(\beta,\sigma)$-extremal deficient partition
$V=V_1\dcup\ldots\dcup V_m$, such that
$\{L^i\}_{i=1}^m$ is a partition of $L$.  Then there exist an index $i_0\in
[m]$ such that we have $|K|\ge k/10$ for the set 
\begin{equation}\label{eq:defK}K=\left\{v\in L^{i_0}\::\:\deg(v,L^{i_0})+\deg(v,\bigcup_{j\neq i_0}(L^j\cup \Supper[j]))\geq \frac{k+1}2\right\}\;.
\end{equation}
\end{lemma}
\begin{proof}We partition $\bigcup_j\Supper[j]$ into sets $\Supptilde[j]$, $j\in [m]$ such
that $\Supptilde[j]\subseteq \Supper[j]$. As $|L|\geq |S|$, there exists an index $i\in
[m]$ such that $|\Supptilde[i]|\leq |L^i|\leq \tfrac{k}2$. Without loss of
generality, assume that $\tfrac{k}2-|\Supptilde[1]|$ is the maximum value among all
the values $\tfrac{k}2-|\Supptilde[i]|$ ($i\in[m]$); then $i_0=1$ is the index asserted by the lemma. We have that $\tfrac{k}2-|\Supptilde[1]|$ is non-negative. For each vertex $v\in L^1\setminus K$, we have\RefereeX{(142)}{This simply follows from the fact that $\Supptilde[j]\subset \Supper[j]$. In other words, $S\setminus \Supptilde[j]\supset S\setminus \Supptilde[j]$ (and the same for a union over several $j$'s. We do not think this needs to be explained.}
$$\deg(v,S\setminus \bigcup_{j\neq 1}\Supptilde[j])\geq \deg(v,S\setminus
\bigcup_{j\neq 1}\Supper[j])\ge\tfrac{k}2.$$ Thus $\deg(v,S^-)>\tfrac{k}2-|\Supptilde[1]|$, where $S^-=\{u\in S\::\: \deg(u,L^i)<\tfrac{k}{5m},\forall i=1,\dots,m\}$. We have\Referee{(143)}
\begin{equation}\label{eq:velikostS-}
|S^-|\frac {k}{5m}>e(L^1\setminus K,S^-)\ge|L^1\setminus K|\left(\frac k2-|\Supptilde[1]|\right)\;.
\end{equation}
On the other hand, as $\sum_{j}|L^j|=|L|\geq |S|=\sum_{j}|\Supptilde[j]|+|S^-|$,
there exists an index $i\in [m]$ such that $|L^i|\geq |
\Supptilde[i]|+\tfrac{|S^-|}{m}$. From the maximality of $\tfrac{k}2-|\Supptilde[1]|$ and from~\eqref{eq:velikostS-} we deduce that $$\frac k2-|\Supptilde[1]|\geq \frac k2-|\Supptilde[i]|\geq |L^i|-|\Supptilde[i]|\geq \frac {|S^-|}m>\frac {5|L^1\setminus K|}{k}\left(\frac k2-|\Supptilde[1]|\right).$$This implies that $k>5|L^1\setminus K|$, and the asserted bound on $|K|$ follows from~\eqref{eq:Am}.
\end{proof}

\begin{lemma}\label{prop_deficientCones}
Suppose that $G$ admits a $(\beta,\sigma)$-extremal deficient partition $V=V_1\dcup\ldots\dcup V_m$. Furthermore, suppose that the sets $\{L^{i}\}_{i\in[m]}$ partition the set $L$. 

Then there exists an index $i_0\in[m]$ and
matchings $\mathcal{E}^{i_0}$, and $\mathcal{J}^{i_0}$ such that the following hold.
\begin{enumerate}[label={$(\roman{*})$}]
\item\label{IWHS1} $\mathcal{E}^{i_0}$ is an $L^{i_0}\leftrightarrow (L\setminus L^{i_0})$-matching, $\mathcal{J}^{i_0}$ is an $L^{i_0}\leftrightarrow \bigcup_{i\neq i_0}\Sweak[i]$-matching.
\item\label{IWHS3} $V(\mathcal{E}^{i_0})\cap V(\mathcal{J}^{i_0})=\emptyset$.
\item\label{IWHS4} $|L^{i_0}|+|\mathcal{E}^{i_0}|+|\mathcal{J}^{i_0}|\ge \frac{k+1}2$.
\item\label{IWHS5} $|\mathcal{E}^{i_0}|+|\mathcal{J}^{i_0}|< \gamma k$.\Referee{(144)}
\end{enumerate}
\end{lemma}
\begin{proof}
By Lemma~\ref{l:e} we have that $|\Sweak[i]|>(\tfrac12-\gamma)k$. We first find for each $i\in[m]$ two vertex-disjoint matchings $\mathcal E^i$ and $\mathcal D^i$, such that $\mathcal E^i$ is an $L^i\leftrightarrow (L\setminus L^i)$-matching, $\mathcal D^i$ is an $L^i\leftrightarrow (S\setminus \Sweak[i])$-matching, and such that
the matchings $\{\mathcal D^i\}_{i\in[m]}$ are pairwise vertex-disjoint.

For each $i\in[m]$, take $\mathcal E^i$ to be a maximum $L^i\leftrightarrow (L\setminus L^i)$
matching. If $|L^i|+|\Sweak[i]|+|\mathcal E^i|>k+1$, we truncate $\mathcal E^i$ so that
$|L^i|+|\Sweak[i]|+|\mathcal E^i|=\max\{k+1,|L^i|+|\Sweak[i]|\}$. Let us assume that
\begin{equation}\label{eq_sizesordered}
|L^1|+|\Sweak[1]|+|\mathcal E^1|\ge|L^2|+|\Sweak[2]|+|\mathcal E^2|\ge\ldots\ge|L^m|+|\Sweak[m]|+|\mathcal E^m| \; \mbox{.}
\end{equation}
Start with $i=1$, and increase the index $i$ gradually. Take $\mathcal D^i$ to be a
maximum $(L^i\setminus V(\mathcal E^i))\leftrightarrow (S\setminus (\Sweak[i]\cup\bigcup_{j<i}V(\mathcal D^j)))$
matching and truncate it so that $|L^i|+|\Sweak[i]|+|\mathcal E^i|+|\mathcal D^i|=\max\{k+1,|L^i|+|\Sweak[i]|+|\mathcal E^i|\}$. Such a matching $\mathcal D^i$ exists. Indeed, if $|L^i|+|\Sweak[i]|+|\mathcal E^i|\ge k+1$, then set $\mathcal D^i=\emptyset$. Otherwise, we find a matching $\mathcal D^i$ of size $d_i=k+1-|L^i|-|\Sweak[i]|-|\mathcal E^i|$ as follows.\Referee{(145)} Set $B_i=S\cap\bigcup_{j<i}V(\mathcal D^j)$. From the sizes of the matchings $\mathcal D^j$ ($j<i$) and the ordering given by~\eqref{eq_sizesordered} we get $|B_i|< m d_i$. Each vertex $u\in L^i$ has at least $d_i$ neighbors outside $L^i\cup \Sweak[i]\cup V(\mathcal E^i)$. Color arbitrary $d_i$ edges emanating from each vertex $u\in L^i$ outside $L^i\cup \Sweak[i]\cup V(\mathcal E^i)$ by black, and the remaining edges incident with $u$ by grey. We have
\begin{equation}\label{eq_manyblackedgesemanating}
e_\mathrm{black}\big(L^i\setminus V(\mathcal E^i),S\setminus (\Sweak[i]\cup
B_i)\big)>d_i(\tfrac12-3\gamma)k-m d_i\frac{k}{5m}>\frac{d_ik}5
\;\mbox{.}
\end{equation}
Since the maximum degree in the graph $G_\mathrm{black}[L^i\setminus
V(E^i),S\setminus (\Sweak[i]\cup B_i)]$ is upper-bounded by
$\max\{\tfrac{k}{5m},d_i\}=\tfrac{k}{5m}$, there
is no vertex cover of $G_\mathrm{black}[L^i\setminus V(E^i),S\setminus
(\Sweak[i]\cup B_i)]$ of size less than $(\tfrac{d_ik}5)/(\tfrac{k}{5m})\ge d_i$. Hence, by
K\"onig's Matching Theorem, there exists a matching $\mathcal D^i$ of size $d_i$ with the desired properties. We set $X_i=V(\mathcal D^i)\setminus L^i$.

Let us summarize the properties of the obtained structure. For each $i\in [m]$ we have
\begin{align}
\label{eq_clustersatleastk+1}|L^i|+|\Sweak[i]|+|\mathcal E^i|+|X_i|&\ge k+1 \mbox{, and}\\
\label{eq_Xidisjoint}X_i\cap \bigcup_{j\not=i} X_j=\emptyset
\quad&\mbox{and}\quad \Sweak[i]\cap X_i=\emptyset \; \mbox{.}
\end{align}
There is an index $i_0\in[m]$ such that sufficiently many vertices from $\Sweak[i_0]\cup X^{i_0}$ are contained in $\bigcup_{j\not =i_0}\Sweak[j]$, giving the desired bridges from the clump $V_{i_0}$. Indeed,
\begin{align*}
n-|L|&\ge\left|\bigcup_i(\Sweak[i]\cup X_i)\right|
     \overset{\eqref{eq_Xidisjoint}}\ge \sum_i |\Sweak[i]|+\sum_i |X_i|-\sum_i
\left|(\Sweak[i]\cup X_i)\cap \bigcup_{j\not=i}\Sweak[j]\right|\\
     &\overset{\eqref{eq_clustersatleastk+1}}\ge m(k+1)-|L|-\sum_i
\left|(\Sweak[i]\cup X_i)\cap \bigcup_{j\not=i}\Sweak[j]\right|-\sum_i|\mathcal E^i| \; \mbox{,}
\end{align*}
which yields
\begin{align*}
\sum_i \left(|L^i|+|\mathcal E^i|+\left|(\Sweak[i]\cup
X_i)\cap\bigcup_{j\not=i}\Sweak[j]\right|\right)&\ge |L|+m(k+1)-n\ge
m(k+1)-\frac{n}2\\
&\overset{\eqref{eq_rkn}}\ge \frac{m(k+1)}2 \; \mbox{.}
\end{align*}
By averaging, there exists an index $i_0\in[m]$ such that
\begin{equation}\label{eq_LiHunted}
|L^{i_0}|+|\mathcal E^{i_0}|+\left|(\Sweak[i_0]\cup X_{i_0})\cap\bigcup_{j\not={i_0}}\Sweak[j]\right|\ge \frac{k+1}2 \;
\mbox{.}
\end{equation}
It remains to define $\mathcal{J}^{i_0}$. Let $\mathcal{J}_1=\{e\in \mathcal D^{i_0}\::\: e\cap \bigcup_{j\not=i_0}\Sweak[j]\not=\emptyset\}$. Set $Q=\Sweak[i_0]\cap \bigcup_{j\not=i_0}\Sweak[j]$. Let~$\mathcal{J}_2$ be any matching in $G[Q,L^{i_0}\setminus V(\mathcal{E}^{i_0}\cup \mathcal{J}_1)]$ that covers~$Q$. Since $|Q|<\gamma k$, we can find such a matching greedily. Set $\mathcal{J}^{i_0}=\mathcal{J}_1\cup \mathcal{J}_2$. 

Properties~\ref{IWHS1}--\ref{IWHS3} of the lemma are clear from the construction. Property~\ref{IWHS4} follows from~\eqref{eq_LiHunted}, and using that $|\mathcal{J}_1|=|X_{i_0}\cap\bigcup_{j\not={i_0}}\Sweak[j]|$ and $|\mathcal{J}_2|=|\Sweak[i_0]\cap\bigcup_{j\not={i_0}}\Sweak[j]|$.

Last,~\eqref{eq:Am} tells us that we can truncate $\mathcal{E}^{i_0}$ and $\mathcal{J}^{i_0}$ so that~\ref{IWHS5} is satisfied without violating~\ref{IWHS4}. This truncation preserved properties~\ref{IWHS1}--\ref{IWHS3}.\Referee{(144)}
\end{proof}

\subsection{Proof of Lemma~\ref{prop_ECFewLeaves}}\label{ssec_FewLeavesEmbedding}
\RefereeX{(146)}{The initial part of the proof was rewritten quite a bit. As a cosequence, (147), (149) and (151) are not relevant anymore (the text was replaced).}
Suppose that $T$ and $G$ satisfy the hypothesis of Lemma~\ref{prop_ECFewLeaves}. By Lemma~\ref{lemma_consideronlysmalldiscrepancy}, we can assume that $T$ has discrepancy less than $2\gamma k$. In particular, 
\begin{equation}\label{eq:doyouwantme}
|T_\oplus|\le\frac k2+\gamma k\;. 
\end{equation}

Recall that if $G$ is deficient then by Lemma~\ref{lem:def} we have $\tilde V=\emptyset$.
For each $i\in[\ell]$ we define $X^i=\{v\in V_i\: :\:
\deg(v,L^i)>\tfrac{k}{5m}\}$. 
If $G$ is abundant, we set $\Lambda\subset[\ell]$ to be
the set of indices $i_0$ such that $|L^{i_0}|\ge \tfrac{k+1}2$, and set
$\mathcal{E}^{i_0}=\mathcal{J}^{i_0}=\emptyset$. If $G$ is deficient, we apply
Lemma~\ref{prop_deficientCones} to obtain sets $\Sweak[i]$, an index $i_0$, and two matchings
$\mathcal{E}^{i_0}$ and $\mathcal{J}^{i_0}$ such that
\begin{equation}\label{eq:Biza}
|L^{i_0}|+|\mathcal{E}^{i_0}|+|\mathcal{J}^{i_0}|\ge \tfrac{k+1}2\ge|T_\ominus| \;.
\end{equation} We
then set $\Lambda=\{i_0\}$.

For each $i_0\in \Lambda$ individually,\Referee{(148)} we shall try to embed the tree $T$ so that most of the vertices of~$T$ are embedded in
$V_{i_0}$. We show that if all the attempts fail, then there exists a
set $Q$ satisfying the assertions of Lemma~\ref{prop:EC-obecna}.

Fix $i_0\in\Lambda$. Let $F^{i_0}=V(\mathcal{E}^{i_0})\cup V(\mathcal{J}^{i_0})$. By Lemma~\ref{prop_deficientCones}\ref{IWHS5}, $|F^{i_0}\cap L^{i_0}|\le \gamma k$. Take a maximum family
$\mathcal{P}$ of vertex-disjoint $(L^{i_0}\setminus F^{i_0})\leftrightarrow (V\setminus
(L^{i_0}\cup \Sstrong[i_0]\cup F^{i_0}))\leftrightarrow(L^{i_0}\setminus F^{i_0})$-paths.

\begin{Claim71}\label{cla1}
If $|L^{i_0}\cup \Sstrong[i_0]\cup F^{i_0}|+|\mathcal P|\ge k-1$ then $T\subset G$.
\end{Claim71}
\begin{proof}
Consider a family of paths $\mathcal{P}'\subset\mathcal{P}$ by truncating $\mathcal{P}$ so that $|\mathcal{P}'|=\min\{|\mathcal{P}|,30\gamma k\}$. By~\eqref{eq:LSdiamond}, $|L^{i_0}\cup \Sstrong[i_0]\cup F^{i_0}|+|\mathcal P'|\ge k-1$. Observe that $V(\mathcal P')\setminus L^{i_0}$ are the middle vertices of~$\mathcal{P}'$. Fix a set $A\subset L^{i_0}$ of size $|T_\ominus|-|\mathcal{J}^{i_0}|-|\mathcal{E}^{i_0}|$ and which contains
$(F^{i_0}\cup V(\mathcal P'))\cap L^{i_0}$. This is possible by~\eqref{eq:Biza} and by\Referee{(150)}
$$|(F^{i_0}\cup V(\mathcal P'))\cap L^{i_0}|\lBy{L\ref{prop_deficientCones}\ref{IWHS5}} 31\gamma k<
\tfrac k2-2\gamma k \lBy{\eqref{eq:doyouwantme}} |T_\ominus|-|\mathcal{J}^{i_0}|-|\mathcal{E}^{i_0}|\;.$$

We apply Lemma~\ref{prop_SE_embeddingwithfewleaves}, setting the parameters of the lemma as $\alpha=60\gamma$, $A, B_\mathrm{a}=(L^{i_0}\setminus
A)\cup \Sstrong[i_0],
B_\mathrm{d}=V(\mathcal P')\setminus L^{i_0},\mathcal{Q}=\mathcal{P}',\mathcal{E}=\mathcal{E}
^{i_0}\cup \mathcal{J}^{i_0},\mathcal{M}=\emptyset, I=[\ell]\setminus\{i_0\}$, and $H_\kappa=G[L^\kappa\cup \Sweak[\kappa]]$ (for each $\kappa\in I$) to get $T\subseteq G$. To check Condition \ref{it:delAB} of Lemma~\ref{prop_SE_embeddingwithfewleaves}, let us consider an arbitrary vertex $v\in A$.
\begin{align}
\begin{split}\deg(v,B_\mathrm{a}\cup B_\mathrm{d})&\geBy{\eqref{eq:Am}}  |(B_\mathrm{a}\cup B_\mathrm{d})\cap V_i|-\gamma k \\
&\ge |B_\mathrm{a}\cup B_\mathrm{d}|-|L^{i_0}\setminus V_{i_0}|-|\Sstrong[i_0]\setminus V_{i_0}|-|\mathcal P'|-\gamma k\\
&\ge |B_\mathrm{a}\cup B_\mathrm{d}|-\sqrt{\beta}k-\sqrt{\beta}k-30\gamma k-\gamma k\ge |B_\mathrm{a}\cup B_\mathrm{d}|-60\gamma k\;,
\label{eq:degvB}
\end{split}
\end{align}
where the last line follows from Lemma~\ref{l:e}(iii) combined with the definition of $L^{i_0}$, $\Sstrong[i_0]$, and~\eqref{eq:Am}. Other conditions of Lemma~\ref{prop_SE_embeddingwithfewleaves} are easy to check.
\end{proof}

It remains to consider the case that $|L^{i_0}\cup \Sstrong[i_0]\cup F^{i_0}|+|\mathcal P|\le k-2$. From~\eqref{eq:LSdiamond}, we have\Referee{(152)} $|\mathcal{P}|<\gamma k$. Consider an arbitrary vertex $u\in L^{i_0}\setminus
(F^{i_0} \cup V(\mathcal{P}))$. Since $\deg(u)\ge k$,\Referee{(153)} there are at least two edges $ux^1_u$ and $ux^2_u$ that emanate into $V\setminus (L^{i_0}\cup \Sstrong[i_0]\cup F^{i_0})$. By the maximality of $\mathcal{P}$ all the vertices
$x^1_u,x^2_u$, $u\in L^{i_0}\setminus (F^{i_0} \cup V(\mathcal{P}))$,\Referee{(154)} are pairwise distinct. Set $R_{i_0}=\bigcup_{u\in L^{i_0}\setminus (F^{i_0} \cup
V(\mathcal{P}))}\{x^1_u,x^2_u\}$ and $\tilde R_{i_0}= R_{i_0}\cap \tilde V$.
\begin{Claim71}\label{claMatching}
For an arbitrary set $U\subset R_{i_0}$ there exists a $U\leftrightarrow (L^{i_0}\setminus (F^{i_0} \cup V(\mathcal{P}))$ matching $\mathcal F_{i_0}$ with $|\mathcal F_{i_0}|\ge \frac{|U|}2$.
\end{Claim71}
\begin{proof}
For $q=1,2$, let $U_q=\{u\in L^{i_0}\setminus (F^{i_0} \cup V(\mathcal{P})\::\:x^q_u\in U\}$. There exists $q\in[2]$ such that $|U_q|\ge \frac{|U|}2$. The desired matching $\mathcal F_{i_0}$ is then $\{ux^q_u\}_{u\in U_q}$.
\end{proof}
\begin{Claim71}\label{cla2}\RefereeX{(156),(155)}{The set $Y_{i_0}$ is not defined anymore}
If $|\tilde R_{i_0}|\le 2|L^{i_0}|-7m\gamma k$ then $T\subset G$.
\end{Claim71} 
\begin{proof}
Observe that 
\begin{align*}
\left|R_{i_0}\cap \bigcup_{i\in [\ell]} (L^i\cup X^i)\right|&\ge |R_{i_0}|-|\tilde R_{i_0}|-\left|V\setminus\left(\tilde V\cup
\bigcup_{i\in[\ell]}X^i\right)\right|\\
&\geBy{L\ref{l:e}} 2|L^{i_0}\setminus (F^{i_0}\cup V(\mathcal P))|-(2|L^{i_0}|-7m\gamma k)-m\sqrt{\beta} k\\
&\ge 2|L^{i_0}|-4\gamma k - 2|L^{i_0}|+7m\gamma k-m\sqrt{\beta} k\ge 2\gamma k\;.
\end{align*}
By Claim~\ref{claMatching}, there exists an $(L^{i_0}\setminus F^{i_0}) \leftrightarrow (R_{i_0}\cap \bigcup_{i\in [\ell]} (L^i\cup X^i))$ matching $\mathcal N$ of size $\gamma k$.
Fix a set $A\subset L^{i_0}$ of size $|T_\ominus|-|\mathcal{J}^{i_0}|-|\mathcal{E}^{i_0}|$ and which contains
$(F^{i_0}\cup V(\mathcal N))\cap L^{i_0}$. We apply Lemma~\ref{prop_SE_embeddingwithfewleaves}  with parameters
$\alpha=60\gamma$, $A, B_\mathrm{a}=(L^{i_0}\setminus
A)\cup \Sstrong[i_0],
B_\mathrm{d}=\mathcal{Q}=\mathcal{M}=\emptyset,\mathcal{E}=\mathcal{E}
^{i_0}\cup \mathcal{J}^{i_0}\cup\mathcal N, I=[\ell]\setminus\{i_0\}$, and $H_\kappa=G[L^\kappa\cup \Sweak[\kappa]]$ (for each $\kappa\in I$) 
and get that $T\subset G$. Condition~\ref{it:BEM} of Lemma~\ref{prop_SE_embeddingwithfewleaves} follows from~\eqref{eq:LSdiamond}. Condition~\ref{it:delAB} is checked analogously as in~\eqref{eq:degvB}. Other conditions are easy to verify.
\end{proof}

Putting Claim~\ref{cla1} and Claim~\ref{cla2} together, we can assume that for each $i\in\Lambda$,  we have \Referee{(157)}
\begin{equation}\label{eq:Ri}
|\tilde R_{i}|> 2|L^{i}|-7m\gamma k\;.
\end{equation}

Suppose that there exists an index $i_0\in \Lambda$ such that 
\begin{equation}\label{eq:wc}
\left|\tilde R_{i_0}\cap \bigcup_{i\in \Lambda\setminus\{i_0\}}\tilde R_i\right|\ge 8\gamma k\;.
\end{equation} 
Claim~\ref{claMatching} gives an $(L^{i_0}\setminus F^{i_0})\leftrightarrow(\tilde R_{i_0}\cap \bigcup_{i\in\Lambda\setminus\{i_0\}}\tilde R_i)$ 
matching $\mathcal M_1$ of size $4\gamma k$. Further applications of Claim~\ref{claMatching} for indices in $i\in \Lambda\setminus\{i_0\}$ and sets $U=V(M_1)\cap \tilde R_{i_0}\cap \tilde R_i$ yield a $\left(V(\mathcal M_1)\cap \tilde R_{i_0}\cap \bigcup_{i\in\Lambda\setminus\{i_0\}}\tilde R_i\right)\leftrightarrow \left(\bigcup_{i\in
\Lambda\setminus\{i_0\}}L^i\right)$ matching $\mathcal M_2$ of size $2\gamma k$. From this matching choose a matching $\mathcal M_3$ of size $\gamma k$ that is disjoint from $F^{i_0}$. Extend the edges of $\mathcal M_3$  by edges of~$\mathcal M_1$. This leads to $\gamma k$ vertex-disjoint $L^{i_0}\leftrightarrow(\tilde R_{i_0}\cap \bigcup_{i\in\Lambda\setminus\{i_0\}}\tilde R_i)\leftrightarrow \left(\bigcup_{i\in
\Lambda\setminus\{i_0\}}L^i\right)$-paths, denoted by $\mathcal M$. \Referee{(158)}
 Fix a set $A\subset L^{i_0}$ of size $|T_\ominus|-|\mathcal{J}^{i_0}|-|\mathcal{E}^{i_0}|$ and which contains
$(F^{i_0}\cup V(\mathcal M))\cap L^{i_0}$. This is possible by~\eqref{eq:Biza} and by
$$|(F^{i_0}\cup V(\mathcal M))\cap L^{i_0}|\lBy{L\ref{prop_deficientCones}\ref{IWHS5}} 2\gamma k \lBy{\eqref{eq:doyouwantme}} |T_\ominus|-|\mathcal{J}^{i_0}|-|\mathcal{E}^{i_0}|\;.$$
We apply Lemma~\ref{prop_SE_embeddingwithfewleaves},  setting the parameters of the lemma as $\alpha=60\gamma$, $A, B_\mathrm{a}=(L^{i_0}\setminus
A)\cup \Sstrong[i_0],
B_\mathrm{d}=\mathcal{Q}=\emptyset,\mathcal{E}=\mathcal{E}
^{i_0}\cup \mathcal{J}^{i_0},\mathcal{M}, I=[\ell]\setminus\{i_0\}$, and $H_\kappa=G[L^\kappa\cup \Sweak[\kappa]]$ (for each $\kappa\in I$) to get $T\subseteq G$. Condition~\ref{it:BEM} of Lemma~\ref{prop_SE_embeddingwithfewleaves} follows from~\eqref{eq:LSdiamond}. Condition~\ref{it:delAB} is checked as in~\eqref{eq:degvB}. Consequently, $T\subset G$.

We assume in the rest that no index $i_0$ satisfies~\eqref{eq:wc}. We have 
\begin{equation}\label{eq_Y2sum}
\left|\bigcup_{i\in\Lambda}\tilde R_i\right|\ge \sum_{i\in \Lambda}(|\tilde R_i|-|\tilde R_i\cap
\bigcup_{j\in \Lambda\setminus\{i_0\}}\tilde R_j|)\geBy{\eqref{eq:Ri},$\lnot$\eqref{eq:wc}} 2\sum_{i\in\Lambda}|L^{i}|-15m^2\gamma k
\;\mbox{.}
\end{equation} Set
$Y=\bigcup_{i\in\Lambda}\tilde R_i$.

We distinguish three cases:
\begin{itemize}
\item[$(\clubsuit 1)$] {\em We have that $|L\cap Y|\le \tfrac{k}8$ and
$e(Y,\tilde{V}\setminus Y)<\sigma k^2/2$.}\\
{\em Solution of $(\clubsuit 1)$}: We show that the set
$Q=\tilde{V}\setminus Y$ satisfies the assertions of
Lemma~\ref{prop:EC-obecna}.

First, we prove the property of
Lemma~\ref{prop:EC-obecna}~\ref{pr:QL}. By the hypothesis of $(\clubsuit
1)$, not many vertices in $Y$ are large. Thus the ratio of the large vertices in the graph $G[\bigcup_{i\in\Lambda}V_i\cup Y]$ is substantially
smaller
than one half. Then there must be substantially more than
half of the large vertices in the complementary set $Q$,
and the property follows. We make the idea rigorous by
the following calculations. For each $i\in\Lambda$ set $l_i=|L^i|$.
\begin{align*}
\frac12 n&\le |L|\le(\ell-|\Lambda|)\tfrac{k}2+\sum_{i\in\Lambda}l_i+|L\cap
Y|+|L\cap Q|+|L\setminus (\tilde{V}\cup\bigcup_{j\in[\ell]}L^j)|\\
&<(\ell-|\Lambda|)\tfrac{k}2+\sum_{i\in\Lambda}l_i+\tfrac{k}8+|L\cap
Q|+\gamma n\;\mbox{.}
\end{align*}
Thus,
\begin{equation}\label{eq:stress}
\begin{aligned}
|L\cap Q|&>
\frac12n-(\ell-|\Lambda|)\tfrac{k}2-\sum_{i\in\Lambda}l_i-\tfrac{k}8-\gamma n\\
&>
\frac12\left(|\tilde{V}|-2\sum_{i\in\Lambda}l_i\right)+|\Lambda|\tfrac{k}2-\tfrac{k}8-2\gamma
n\\
&\overset{\eqref{eq_Y2sum}}>\frac12|Q|+|\Lambda|\tfrac{k}2-\tfrac{k}7\ge\frac12|Q|+\frac5{14}k\;\mbox{,}
\end{aligned}
\end{equation}
which was needed to show the property of
Lemma~\ref{prop:EC-obecna}~\ref{pr:QL}. Looking back at~\eqref{eq:stress},
we see that $|Q|\ge\frac12|Q|+\frac5{14}k$, and thus also the property of
Lemma~\ref{prop:EC-obecna}~\ref{pr:velkyQ} follows.

Finally, to infer the property of
Lemma~\ref{prop:EC-obecna}~\ref{pr:Qizo} we write $$e(Q,V\setminus Q)\le
e(Y,\tilde{V}\setminus Y)+e(\tilde{V},V\setminus \tilde{V})<\sigma k^2/2+\beta
k^2\le \sigma k^2\;.$$ The bound on the first summand follows from the
hypothesis of $(\clubsuit 1)$, the bound on the second summand follows from the
$(\beta,\sigma)$-extremality.

\item[$(\clubsuit 2)$] {\em We have that $|L\cap Y|> \tfrac{k}8$ and
$e(Y,\tilde{V}\setminus Y)<\sigma k^2/2$}.\\
{\em Solution of $(\clubsuit 2)$}: We show that $T\subset G$. The hypothesis of\Referee{(159)}
$(\clubsuit 2)$ gives $e(G[Y])\ge \frac12|L\cap Y|k-e(Y,\tilde{V}\setminus Y)\ge \frac{k^2}{20}$. 
The average degree in $G[Y]$ is $\frac{2e(G[Y])}{|Y|}\ge \frac{k^2}{10 n}\ge \tfrac{qk}{10}$. 
There exists a subgraph $H_*\subset G[Y]$ with $\delta(H_*)\ge
\tfrac{q k}{20}$. By averaging, there exists an index $i_0\in \Lambda$ such that
\begin{equation}\label{eq_H*}
|\tilde{R}_{i_0}\cap V(H_*)|>\tfrac{qk}{20m}\; \mbox{.}
\end{equation}
Fix such an index $i_0$. By~\eqref{eq_H*} there exists an
$L^{i_0}\leftrightarrow V(H_*)$-matching $\mathcal{E}$ of size $30 \gamma  k$. 
Fix a set~$A\subseteq L^{i_0}$ of size $|T_\ominus|-|\mathcal{E}|$ containing~$V(\mathcal{E})\cap L^{i_0}$. Such a set exists by~\eqref{eq:Am}.
By
Lemma~\ref{prop_SE_embeddingwithfewleaves} (with $\alpha=60\gamma$, $A, B_\mathrm{a}=\Sstrong[i_0],
B_\mathrm{d}=\mathcal{Q}=\mathcal{M}=\emptyset,\mathcal{E},$  and $H_*$, $I=\{*\}$) we
get that $T\subset G$. To check Condition~\ref{it:BEM}, observe that, by~\eqref{eq:LSdiamond} and the fact that we are the deficient case, we have $|\Sstrong[i_0]|+|\mathcal E|\ge \frac k2-\gamma k + 30 \gamma k\ge |T_\oplus|$. Condition~\ref{it:delAB} follows from~\eqref{eq:Am}. Other conditions are straightforward.

\item[$(\clubsuit 3)$]{\em We have that $e(Y,\tilde{V}\setminus Y)\ge
\sigma k^2/2$.}\\ {\em Solution of $(\clubsuit 3)$}: We show that $T\subset G$.
The average degree of the bipartite graph $G[Y,\tilde{V}\setminus Y]$ is at least $q \sigma k$. Thus there exists a
graph $H_*\subset G[Y,\tilde{V}\setminus Y]$ with
$\delta(H_*)\ge \tfrac{q\sigma k}2$.
There must be an index $i_0\in\Lambda$ such that $|\tilde{R}_{i_0}\cap
V(H_*)|>\tfrac{\sigma q k}{2m}$. Fix such an index $i_0$, find
a matching $\mathcal{E}$ and set $A$ as in $(\clubsuit 2)$. We apply 
Lemma~\ref{prop_SE_embeddingwithfewleaves} as in $(\clubsuit 2)$.
\end{itemize}

\subsection{Proof of Lemma~\ref{prop_ECManyLeaves}}\label{ssec_ManyLeavesEmbedding}
In order to prove Lemma~\ref{prop_ECManyLeaves} we need the following  auxiliary lemma.\RefereeX{(160),(161)}{The Lemma containing these referee's comments is not needed. See the remark regarding comments (139)-(141)}

\begin{lemma}\label{lem:vnor-obecPartition}
Let $F$ be a rooted forest with a partition $V(F)=O_1\cup O_2$, such that $O_2$
is independent. Let $W$ be the set of leaves of $F$ and set $P=\{u\in O_2\::\: |W\cap \children(u)|=1\}$. Let $H$ be a graph and let $A,B\subseteq V(H)$ be two disjoint sets such that for\Referee{(162)} some $f\ge 0$ we have $|A|\geq |O_1|$,  $\min\{\delta(A,A),\delta(B,A)\}> |O_1|-f$, $\delta(A,B)> |B|-f$, $|B|\geq |O_2\setminus W|$, and $\delta(A)\geq v(F)-1$. If $|P|\geq 2f$, then there exists an embedding $\varphi$ of $F$ in $H$ such that $\varphi(O_1)\subseteq A$.
\end{lemma}
\begin{proof}\Referee{(163)}
Choose a subset $P'\subseteq P$ of size $|P'|=2f$. Consider the subtree $F'=F-W'$, where $W'=W\cap (O_2\cup \neighbor(P'))$. We embed greedily the tree $F'$ in $A\cup B$, so that $V(F')\cap O_1$ maps to~$A$ and $V(F')\cap O_2$ maps to~$B$. Denote this embedding by~$\varphi'$. Next we want to embed the leaves $W'\cap O_1$ in~$A$. Let $A'=A\setminus \varphi (V(F'))$. We have $|A'|\geq 2f=|\varphi'(P')|$, 
$\delta(\varphi(P'),A')> f=\tfrac{|P'|}2$, and $\delta(A',\varphi(P'))>
f=\tfrac{|P'|}2$. By K\"onig's matching theorem, there exists a matching~$M$ in $H[A',\varphi' (P')]$ that covers $\varphi' (P')$.

We extend $\varphi'$ to an embedding $\varphi$ of~$F$, by embedding $W'\cap O_1$ according to the matching~$M$, and by embedding $W\cap O_2$ greedily (this is guaranteed by the minimum degree condition of the set~$A$).\end{proof}

A semi-independent partition $(U_1,U_2)$ of a tree $F$ is {\em $p$-ideal} if each of the vertex sets~$U_1$ and~$U_2$ contains at least $p$ leaves of $F$.
If $\disc(T)\ge 2\gamma k$, then Lemma~\ref{lemma_consideronlysmalldiscrepancy}
 ensures that $T\subset G$. Therefore, the proof of Lemma~\ref{prop_ECManyLeaves} follows from Lemma~\ref{lemma:firststep} and~\ref{lemma:secondstep} below.
\begin{lemma}\label{lemma:firststep}
\addtocounter{lemmafirststep}{\value{theorem}}
If we are in the setting of Lemma~\ref{prop_ECManyLeaves} and $\disc(T)< 2\gamma k$, then $T$ has an $8\gamma k$-ideal semi-independent partition, or $T\subset G$.
\end{lemma}
\begin{lemma}\label{lemma:secondstep}
\addtocounter{lemmasecondstep}{\value{theorem}}
If we are in the setting of Lemma~\ref{prop_ECManyLeaves}, $\disc(T)< 2\gamma k$, and $T$ has an $8\gamma k$-ideal semi-independent partition then $T\subset G$.
\end{lemma}
\begin{proof}[Proof of Lemma~\ref{lemma:firststep}]
\setcounter{theorem}{\value{lemmafirststep}}
We partition the set $W$ of leaves of $T$ into $W_\oplus=W\cap T_\oplus$ and $W_\ominus=W\cap T_\ominus$. Set $w_\oplus=|W_\oplus|$ and
$w_\ominus=|W_\ominus|$. We have that $w_\oplus+w_\ominus\ge 60\gamma k$. We
distinguish three cases based on the values of $w_\oplus$ and $w_\ominus$.

\noindent $\bullet$ \emph{We have $w_\oplus\ge 8\gamma k$ and $w_\ominus\ge 8\gamma k$.}\\ Then $(T_\ominus,T_\oplus)$ is an $8\gamma k$-ideal semi-independent partition.

\smallskip
\noindent $\bullet$ \emph{We have $w_\oplus< 8\gamma k$.}\\ Then we have $w_\ominus\ge 52\gamma k$. We distinguish two subcases.
  \begin{itemize}
   \item If $|\parent(W_\ominus)|\le 16\gamma k$, we consider the sets $U_1=T_\ominus\div (W_\ominus\cup\parent(W_\ominus))$ and $U_2=T_\oplus\div (W_\ominus\cup\parent(W_\ominus))$. The partition $(U_1,U_2)$ is semi-independent with $|U_2|-|U_1|\geq 72\gamma k$, a contradiction with the assumption $\disc(T)<2\gamma k$.
   \item If $|\parent(W_\ominus)|> 16\gamma k$, we choose an arbitrary subset $P'\subset \parent(W_\ominus)$ with $|P'|=8\gamma k$ and set $W'_\ominus=\neighbor(P')\cap W_\ominus$. The partition $(U_1,U_2)$, defined by $U_1=T_\ominus\div (W'_\ominus\cup P')$ and $U_2=T_\oplus\div (W'_\ominus\cup P')$, is an $8\gamma k$-ideal semi-independent partition.
  \end{itemize}
  
\smallskip
\noindent $\bullet$ \emph{We have $w_\ominus<8\gamma k$.}\\We use Fact~\ref{fact_fullsubtrees}~\ref{tesco2} to
 find a full-subtree $\tilde{T}\subset T$ rooted at a vertex $r$ with $p$
 proper leaves, where $p\in[20\gamma k, 40\gamma k]$. The choice of
 $\tilde{T}$ has the property that
     \begin{equation}\label{eq:10gammak}
     \min\{|W_\oplus\cap V(\tilde T)|,|W_\oplus\setminus V(\tilde
     T)|\}\geq 12\gamma k\;.
     \end{equation}
     Set $d=|V(\tilde{T})\cap T_\oplus|-|V(\tilde{T})\cap T_\ominus|$. We distinguish six subcases.\\
     \begin{tabular}{ll}
{\bf(C1)}~$r\in T_\oplus$ and $d\le \tfrac{\gap(T)}2$,
&{\bf(C2)}~$r\in T_\ominus$ and $d\ge \tfrac{\gap(T)}2$,\\
{\bf(C3)}~$r\in T_\oplus$ and $d\ge \tfrac{\gap(T)}2+1$,
&{\bf(C4)}~$r\in T_\ominus$ and $d\le \tfrac{\gap(T)}2-1$,\\
{\bf(C5)}~$r\in T_\oplus$ and $d= \tfrac{\gap(T)+1}2$,
&{\bf(C6)}~$r\in T_\ominus$ and $d= \tfrac{\gap(T)-1}2$.
\end{tabular}\\

In cases {\bf(C1)}--{\bf(C4)} we obtain a semi-independent
partition by flipping either $V(\tilde{T})$ (in cases {\bf(C1)} and {\bf(C2)}) or
$V(\tilde{T})\setminus\{r\}$ (in cases {\bf(C3)} and {\bf(C4)}) in the original
partition $(T_\ominus,T_\oplus)$. In these cases, the obtained partition is indeed
$8\gamma k$-ideal by~\eqref{eq:10gammak}. 

In the rest, we consider only the case {\bf(C5)}, the case {\bf(C6)} being
analogous. Notice that $\gap(T)$ has the same parity as $v(T)=k+1$. Thus, the integrality of $d$ gives that $k$ is even.\Referee{(164)} We set $O_1=T_\ominus\div V(\tilde{T})$ and $O_2=T_\oplus\div
V(\tilde{T})$. We have that $|O_1|=\tfrac{k+2}2$, and $|O_2|=\tfrac{k}2$.
\begin{AuxiliaryCl}\label{club1}\Referee{(165)}
We have $\parent(O_1\cap W)\subset O_2$, or $T$ has an $8\gamma k$-ideal semi-independent partition.
\end{AuxiliaryCl}
\begin{proof}
The existence of a vertex $u\in O_1\cap W$ whose parent lies in $O_1$ would yield a semi-independent partition $(O_1\setminus \{u\},O_2\cup\{u\})$, which would be by~\eqref{eq:10gammak} $8\gamma k$-ideal.
\end{proof}
\begin{AuxiliaryCl}\label{club2}
If there exist two distinct leaves
$z_1,z_2\in O_1$ with a common neighbor $\{x\}=\parent(\{z_1,z_2\})$, then $T$ has an $8\gamma k$-ideal semi-independent partition.
\end{AuxiliaryCl} 
\begin{proof}
By Claim~\ref{club1} we can assume that $x\in O_2$. Set $U_1=O_1\div\{x,z_1,z_2\}$ and
$U_2=O_2\div\{x,z_1,z_2\}$. Then $|U_1|=\tfrac{k}2$, $|U_2|=\tfrac{k}2+1$, $|U_1\cap W|=|O_1\cap W|-2$, and $|U_2\cap
W|=|O_2\cap W|+2$. By~\eqref{eq:10gammak}, the partition $(U_1,U_2)$ is
$8\gamma k$-ideal semi-independent.
\end{proof}By the two claims above, we restrict ourselves to the case that $\parent(O_1\cap W)\subset O_2$, and the leaves in $O_1$ have pairwise distinct parents.

\begin{AuxiliaryCl}\label{club3}For the set $O_1^*=\{y\in O_1\cap W\::\:\deg(\parent(y))=2\}$,\Referee{(166)} we have $|O_1^*|> 1.5\gamma k$.
\end{AuxiliaryCl}
\begin{proof}Recall that every vertex in $\parent(O_1\cap W)$ has exactly one leaf-child in~$O_1$. Set $W_*=V(\tilde{T})\cap W_\oplus$ and $T_*=\tilde{T}- W_*$. By~\eqref{eq:10gammak}, we have $|W_*|\ge 12\gamma k$. \RefereeX{(167)}{Actually, even with the suggested change, the calculation was wrong. We also changed $|W_*|\ge 10\gamma k$ to $|W_*|\ge 12\gamma k$.}
\begin{align*}
|V(T_*)\cap T_\ominus|&=|V(\tilde{T})\cap
T_\ominus| \overset{\text{Fact~\ref{fact_cutnodiscrepancy}}}>
|V(\tilde{T})\cap T_\oplus|-2.5\gamma k\\ &=|V(T_*)\cap T_\oplus|+|W_*|-2.5\gamma k \ge |V(T_*)\cap T_\oplus|+9.5\gamma k\;.
\end{align*}By Fact~\ref{fact_ManyLeavesInUnbalanced}, the tree~$T_*$ contains at least $9.5\gamma k$ leaves from~$T_\ominus$. These leaves are also leaves of~$\tilde{T}$, with $|O_1^*|$ exceptions caused by $\parent(O_1^*)$. Since $w_\ominus<8\gamma k$, we must have $|O_1^*|>1.5\gamma k$.
\end{proof}

We show that $T\subseteq G$ in two cases {\bf ($\diamondsuit$1)} and {\bf ($\diamondsuit$2)} separately, based on whether $G$ is in the abundant or deficient configuration.

{\bf ($\diamondsuit$1)}~~If $G$ admits an abundant partition, then there exists
an index $i\in[\ell]$ such that $|L^i|\geq \tfrac{k+1}2$. As $k$ is even,
$|L^i|\geq \tfrac{k+2}2$. Choose $L_*\subset L^i$ such that
$|L_*|=\tfrac{k+2}2$. Define $Z=\{u\in W\cap O_1: \parent(u)\in O_2 \}$. 
Suppose that $|(W\cap O_1)\setminus Z|> \gamma k$. Then consider the partition $(U_1,U_2)$ with $U_1=O_1\setminus ((W\cap O_1)\setminus Z)$ and $U_2=O_2\cup ((W\cap O_1)\setminus Z)$.\Referee{(169)}  We have $|U_2|-|U_1|>2\gamma k$, a contradiction to $\disc(T)\leq 2\gamma k$. Thus $|(W\cap O_1)\setminus Z|\leq \gamma k$.
Let $Z'\subset Z$ be the set of leaves in~$Z$ with no sibling in~$Z$.\Referee{(168)} Observe that Fact~\ref{fact:leavesB} gives $|Z\setminus Z'|\le 2\gamma k$.
We can now use Lemma~\ref{lem:vnor-obecPartition} with $A=L_*$, $B=\Sstrong[i]\cup(L^i\setminus L_*)$, $f=\gamma k$, and the
partition $(O_1,O_2)$ of $T$ to get $T\subset G$.
Indeed, the above bounds imply that the set $P$ (as defined in
Lemma~\ref{lem:vnor-obecPartition}) is large.

{\bf ($\diamondsuit$2)}~~Suppose that $G$ is in the deficient configuration. Consider the index $i\in [m]$ and the sets~$\Supper[j]$, and $K\subset L^i$ given by Lemma~\ref{lemma_StartDiana}. Let us discard from $O_1^*$ arbitrary vertices so that we have $|O_1^*|=1.5\gamma k$ (cf.~Claim~\ref{club3}). \Referee{(170)}We embed greedily the tree $T^-=T-(O_1^*\cup\parent(O_1^*))$ in $G[L^i\cup \Sstrong[i]]$ using~$L^i$ to host $O_1\setminus O_1^*$ and $\Sstrong[i]$ to host $O_2\setminus \parent(O_1^*)$, and so that the vertices of $\parent(\parent(O_1^*))$ are always mapped to $K$. Such an embedding exists by~\eqref{eq:Am} and~\eqref{eq:LSdiamond}, and because $|\parent(\parent(O_1^*))|\le|\parent(O_1^*)|=|O_1^*|\le 1.5\gamma k$, and $|K|\ge \frac{k}{10}$. It remains to extend the embedding of $T^-$ first to $\parent(O_1^*)$ and then to $O_1^*$. For any vertex in $\parent(\parent(O_1^*))$ mapped to a vertex in $K$, we embed its child from $\parent(O_1^*)$ greedily to $L^i\cup\bigcup_{j\neq i}(L^j\cup \Supper[j])$. 
This way, only vertices of $(O_1\setminus O_1^*)\cup\parent(O_1^*)$ could be embedded in $L^i$. As $|(O_1\setminus O_1^*)\cup\parent(O_1^*)|=|O_1|=\frac{k}2+1=\lceil\frac{k+1}2\rceil$, we can extend the embedding to $\parent(O_1^*)$ by~\eqref{eq:defK}. In the last step, we extend the embedding to $O_1^*$. 
Consider an arbitrary vertex $x\in\parent(O_1^*)$. The vertex $x$ was embedded to~$L^i$, or to $\bigcup_{j\neq i}(L^j\cup \Supper[j])$. If $x$ is mapped to $\bigcup_{j}L^j$, we use the high degree of those vertices to extend the embedding to the child of $x$. In the case $x$ was mapped to $v\in \Supper[j]$ for some $j\neq i$, observe that only vertices from 
$O_1^*\cup\parent(O_1^*)$ could have been mapped to $L^j$. As $|O_1^*\cup\parent(O_1^*)|= 2|O_1^*|=3\gamma k$, the definition of $\Supper[j]$ tells us that $\deg(v, L^j)\ge \frac{k}{5m}$ and we can map the child of $x$ to $L^j$.\Referee{(171)}
\end{proof}

\smallskip
\begin{proof}[Proof of Lemma~\ref{lemma:secondstep}]
\setcounter{theorem}{\value{lemmasecondstep}}
We assume that $T$ has an $8\gamma k$-ideal semi-independent partition $(U_1,U_2)$. Let $W_2$ be the leaves in $U_2$, and let $W_1^*$ be\Referee{(172)} those leaves in $U_1$ which have no leaf-sibling in $U_1$. By Fact~\ref{fact:leavesB}, we have $|W_1^*|\ge 6\gamma k$.

\medskip

First, we show how to resolve the situation in the abundant case. Let $i$ be such that $|L^i|\ge \frac{k+1}2$.  We first embed $T-(W_1^*\cup W_2)$ in $G[L^i\cup \Sstrong[i]]$, using $L^i$ to host $U_1\setminus W_1^*$, and $\Sstrong[i]$ to host $U_2\setminus W_2$. Properties~\eqref{eq:Am} and~\eqref{eq:LSdiamond} tell us that such an embedding exists.

Next, we map $W_1^*$ to\Referee{(173)} the set $L^*\subset L^i$ of unused vertices of $L^i$. To this end, consider an auxiliary bipartite graph $H$ whose two colour classes are $L^*$ and $\parent(W_1^*)$.\Referee{(174)} A pair $vx$, $v\in L^*$, $x\in \parent(W_1^*)$ forms an edge in~$H$ if~$x$ was mapped to a vertex that is adjacent to~$v$ in~$G$. By the definition of~$\Sstrong[i]$, and by~\eqref{eq:Am}, we have $\delta_H(\parent(W_1^*),L^*)\ge |L^*|-\gamma k/2$, and $\delta_H(L^*,\parent(W_1^*))\ge |\parent(W_1^*)|-\gamma k/2=|W_1^*|-\gamma k/2$. We conclude that~$H$ has no vertex cover of size less than $\min\{|W_1^*|,|L^*|\}=|W_1^*|$. By K\"onig's Theorem, there exists a matching covering $\parent(W_1^*)$ in $H$. This matching tells us how to embed $W_1^*$. In the last step, we embed $W_2$. This can be done greedily as $\parent(W_2)$ were mapped to~$L$.

\medskip
It remains to resolve the situation in the deficient case. Consider the index $i\in [m]$ and the sets $\Supper[j]$, and $K\subset L^i$ given by Lemma~\ref{lemma_StartDiana}. Set $W_1^{**}=\{x\in W_1^*\::\:\deg(\parent(x))\le \gamma (k+1)\}$.\Referee{(175)} The degree sum formula for trees gives $|W_1^{**}|\ge |W_1^*|-2/\gamma>5.9\gamma k$. Let $\tilde T\subset V(T)$ be a full-subtree rooted at a vertex $r\in V(T)$,\Referee{(176)} such that $v(\tilde T)\in [k/4,k/2]$. The existence of $\tilde T$ is guaranteed by Fact~\ref{fact_fullsubtrees}. Let $W_1^{***}\subseteq W_1^{**}\setminus \neighbor(r)$ be a set of size $5.8 \gamma k$. This is possible, as by the definition of~$W_1^*$ we have $|W_1^{**}\cap \neighbor(r)|\le 1$.\Referee{(177)} Observe that $|W_1^{***}\cap V(\tilde T)|\ge 2.9\gamma k$ or $|W_1^{***}\setminus V(\tilde T)|\ge 2.9\gamma k$.

First assume the former case. Let $X=\{x\in \parent(W_1^{***}\cap V(\tilde T))\::\: \parent(x)\in U_1\}$. Observe that $X\subseteq V(\tilde T)\setminus \{r\}$. Thus
\begin{equation}\label{eq:v(Tdownarrow)}
\sum_{x\in X} v(T(\downarrow x))\le v(\tilde T)\le \frac k2\;.
\end{equation}
\RefereeX{(178)}{The sentence was removed}
We begin embedding greedily the tree $T'=T-W_2-\bigcup_{x\in X}T(\downarrow x)$ so that~$U_1$ is mapped to $L^i$, $\parent(X)$ is mapped to $K$, and $U_2$ is mapped to $\Sstrong[i]$.  We can do so, as $ |\parent(X)|\le |W_1^{***}|= 5.8 \gamma k\le |K|$ (c.f. Lemma~\ref{lemma_StartDiana})\Referee{(179)}. Such an embedding exists by~\eqref{eq:Am} and~\eqref{eq:LSdiamond}. 

For every $x\in X$, we sequentially assign an index $j_x\in [m]$ to denote where $T(\downarrow x)$ will be embedded, according to the following rule. Let $X'\subset X$ be the set of those $y$'s for which the index $j_y$ has been assigned in previous rounds.
Let $v_x\in K$ be the image of $\parent(x)$.
If there exists any index  $j\neq i$ such that 
\begin{enumerate}[label={$(\roman{*})$}]
\item \label{it:enoughneighbours}$\deg(v_x,(L^j\cup \Sweak[j]))>|\{y\in X'\::\: j_y=j\}|$, and 
\item \label{it:enoughspace Vj}$\sum_{y\in X'\::\: j_y=j}v(T(\downarrow y))\le k/(5m) -2\gamma k\;,$
\end{enumerate}  
then choose such an index $j$ and fix $j_x=j$. If no such index $j\neq i$ exists, than set $j_x=i$.

\RefereeX{!}{The proof was expanded substantially from this point on.}
The assignment finished, for every $x\in X$ with $j_x\neq i$ we map $x$ to $\neighbor (v_x)\cap (L^{j_x}\cup \Sweak[j_x]) $. 
This is possible thanks to Condition~\ref{it:enoughneighbours}. Having mapped all $X_{\neq i}=\{y\in X\::\: j_y\neq i\}$, we embed $\{\children(x), x\in X_{\neq i}\}$ in $L^{j_x}\cup \Sstrong[j_x]$ (see Figure~\ref{fig:xInSweak}). 
Even if~$x$ is mapped to $\Sweak[j_x]$, the at most $\gamma (k+1)$ children of~$x$ (cf.\ definition of $W_1^{**}$) can be embedded thanks to Condition~\ref{it:enoughspace Vj} and the definition of $\Sweak[j]$. 
Having embedded all the vertices $\bigcup_{x\in X_{\neq i}}\children(x)$,  we continue as follows.
For each $x\in X_{\neq i}$ we embed the rest of $T(\downarrow x)$ greedily in $L^{j_x}\cup \Sstrong[j_x]$, which is possible by~\eqref{eq:v(Tdownarrow)}.
We are finished with embedding~$T$ in the case that $j_x\neq i$ for all $x\in X$. Thus, assume that 
\begin{equation}\label{eq:Cond1NOT}
j_x=i \text{ for some $x\in X$\;.}
\end{equation}

Suppose that $\sum_{y\in X\::\: j_y=j}v(T(\downarrow y))> k/(5m) -2\gamma k$ for some $j\neq i$. Set $\mathcal D=\{ T(\downarrow y)\::\: j_y=j\}$. 
\begin{AuxiliaryCl}
\label{cl:U1big}
We have $|U_1\cap V(\mathcal D)|\ge 500 \gamma k$.
\end{AuxiliaryCl}
\begin{proof}
First, consider the case that the total order of \emph{small components} of $\mathcal D$, defined as with at most $10$ vertices, is at least $| V(\mathcal D)|/2$. 
In each such component, there is at least one vertex of $W_1^{***}\subseteq U_1$. Hence $|U_1\cap V(\mathcal D)|\ge \frac 1{10}\cdot \frac{|V(\mathcal D)|}{2}\ge \frac{k}{200 m}>500 \gamma k$.

Next, consider the case that the total order of \emph{large components} of $\mathcal D$ (those having more than 10 vertices) is more than $|V(\mathcal D)|/2$. 
Let $\mathcal D_1$ be those large components $D\in \mathcal D$ with $|U_2\cap V(D)|<10|U_1\cap V(D)|$, and let $\mathcal D_2$ be those large components $D\in \mathcal D$ with $|U_2\cap V(D)|\ge 10|U_1\cap V(D)|$. 
Consider the tree $T''=T-\mathcal D_2$, and its colour classes $T''_\oplus$ and $T''_\ominus$. Let $R$ be the roots of the trees in $\mathcal D_2$. 
We have $|R|\le |V(\mathcal D_2)|/10$. 
Set the partition $V(T)=U_1'\dcup U_2'$, where $U_2'=T''_\oplus\cup (U_2\cap V(\mathcal D_2))\setminus R$ and $U_1'=V(T)\setminus U_2'=T''_\ominus\cup (U_1\cap V(\mathcal D_2))\cup R$. 
Observe that $U_2'$ is an independent set. 
As $\disc(T)<2\gamma k$, we have
$$2\gamma k>|U_2'|-|U_1'|\ge |U_2\cap V(\mathcal D_2)|-|U_1\cap V(\mathcal D_2)|-2|R|\ge (\tfrac{9}{11}-\tfrac{1}{5})|V(\mathcal D_2)|\;.$$
We conclude that $|V(\mathcal D_2)|<4\gamma k$. In particular, we have $|V(\mathcal D_1)|\ge |V(\mathcal D)|/4$. Then
$$|U_1\cap V(\mathcal D)|\ge|U_1\cap V(\mathcal D_1)|\ge \tfrac1{11}\cdot
\tfrac{|V(\mathcal D)|}4>500\gamma k\; 
,$$ as needed. 
\end{proof}
Recall that we have embedded the entire tree $T$ except the set $M=\bigcup_{x\in X\setminus X_{\neq i}} V(T(\downarrow x))$.
Let $Q\subset U_2\cap M$ be a set of size $\max\{400\gamma k, |U_2\cap M|\}$. To finish the embedding of $T$, we embed greedily the vertices of $Q\cup (U_1\cap M)$ in $L^i$ and the vertices of $(U_2\cap M)\setminus Q$ in $\Sstrong[i]$.
Prior to this embedding, by Claim~\ref{cl:U1big}, at least $500\gamma k$ vertices of $U_1$ had been embedded outside of $L^i$. Thus, the minimum-degree conditions~\eqref{eq:Am} guarantee that such a greedy embedding indeed exists.

Thus, it remains to consider that $\sum_{y\in X\::\: j_y=j}v(T(\downarrow y))\le k/(5m) -2\gamma k$ for all $j\neq i$. At the same time, we had not been able to satisfy Condition~\ref{it:enoughneighbours} for any $j\not= i$ for the vertex $x$ from~\eqref{eq:Cond1NOT}. Then $\deg(v_x,\bigcup _{j\neq i}(L^j\cup \Sweak[j]))\leq |\{y\in X\::\: j_y\neq i\}|$.

At least 
\begin{equation}\label{eq:outtLi}
(k+1)/2-|L^i|\ge |U_1|-|L^i|
\end{equation}
vertices of~$U_1$ were embedded outside~$L^i$. Indeed, at least $\deg(v_x,\bigcup _{j\neq i}(L^j\cup \Sweak[j]))$ vertices $x\in X$ were assigned some $j_x\neq i$. Each corresponding tree $T(\downarrow x)$ contains at least one child of~$x$, belonging to $W_1^{***}\subseteq U_1$, which was thus embedded outside $L^i$. Using~\eqref{eq:defK}, we get  \eqref{eq:outtLi}.

 It remains to embed the trees $\{T(\downarrow x)\::\: x\in X\setminus X_{\neq i}\}$. 
We first embed all the trees $T(\downarrow x)\setminus (W_{1}^{***}\cup W_2)$, $x\in X\setminus X_{\neq i}$. The extension  to $W_{1}^{***}\cup W_2$ will be done at the very end.

Set $W_X=W_1^{***}\cap \left (\bigcup_{x\in X\setminus X_{\neq i}}V(T(\downarrow x))\right)$ 
and $W_Y=W_1^{***}\cap \left (\bigcup_{x\in  X_{\neq i}}V(T(\downarrow x))\right)$. 
\begin{AuxiliaryCl}\label{AC:caj}
We have $|W_X\cup W_Y|\ge 1.9\gamma k$.
\end{AuxiliaryCl}
\begin{proof}Let $\tilde W$ be all vertices in $W_1^{***}\cap V(\tilde T)\subseteq U_1$ whose parent lies in $U_2$.  For $y\in \tilde W$, the independence of $U_2$ implies that $\parent(\parent (y))\in U_1$. Thus by the definition of $X$, we have that $\parent(y)\in X$ and thus $y\in W_1^{***}\cap \left(\bigcup_{x\in X}V(T(\downarrow x))\right)=W_X\cup W_Y$. Hence, $\tilde W\subset W_X\cup W_Y$.

The semi-independent partition 
$$\big(U_1\setminus ((W_1^{***}\cap V(\tilde T))\setminus \tilde W)
\:,\:
U_2\cup ((W_1^{***}\cap V(\tilde T))\setminus \tilde W)\big)$$ has gap 
\begin{align*}
&\left(|U_2|+|W_1^{***}\cap V(\tilde T)|-|\tilde W|\right)
-\left(|U_1|+|\tilde W|-|W_1^{***}\cap V(\tilde T)|\right)\\
&\ge 
\left(|W_1^{***}\cap V(\tilde T)|-|\tilde W|\right)
-\left(|\tilde W|-|W_1^{***}\cap V(\tilde T)|\right)\ge 2\cdot 2.9\gamma k-2|\tilde W|\;.
\end{align*}
Since $\disc (T)<2\gamma k$, we get $|\tilde W|\ge 1.9\gamma k$.
\end{proof}

Set $N=\bigcup_{x\in X_{\neq i}}V(T(\downarrow x))$.
By Claim~\ref{AC:caj}, we have that $|W_Y|\ge  0.75\gamma k$, or $|W_X|\ge  0.75\gamma k$. Hence, 
\begin{equation}\label{eq:U1WXN}
|U_1\setminus (W_X\cup N)|\le \frac {k+1}2- 0.75\gamma k\leBy{\eqref{eq:Am}} |L^i|- \frac {\gamma k}{2}\;.
\end{equation}

 For a fixed $x\in X\setminus X_{\neq i}$,  we proceed embedding $T(\downarrow x)\setminus W_X$ greedily in $G[L^i\cup \Sstrong[i]]$, using~$L^i$ to host $(U_1\cap V(T(\downarrow x)))\setminus W_X$, and $\Sstrong[i]$ to host $(U_2\cap V(T(\downarrow x)))\setminus W_2$. 
By~\eqref{eq:U1WXN},~\eqref{eq:Am}, and the definition of $\Sstrong[i]$, we have  $\deg(v_x, L^i)\ge |L^i|-\gamma k/2$ and 
$\delta(\Sstrong[i], L^i)\ge |L^i|-\gamma k/2$ is sufficient to accommodate the vertices from $U_1\setminus (W_X\cup N)$ in $L^i$. 
As for the vertices that need to be mapped to $\Sstrong[i]$, recall that the fact that $(U_1,U_2)$ is $8\gamma k$-ideal yields $|W_2|\ge 8\gamma k$. Together with the fact that we are considering the deficient case, we get that at most 
\[|U_2\setminus W_2|\le (\frac k2+\gamma k)-8\gamma k\le |\Sstrong[i]|-7\gamma k\] vertices are mapped to $\Sstrong[i]$. Hence, the minimum degree  of vertices of $L^i$ to $\Sstrong[i]$ is sufficient for a greedy embedding.

\begin{figure}[t]\label{fig:xInSweak}
\centering
\subfigure[{In case $x$ is mapped to $L^j$ we can embed the tree $T(\downarrow x)$ greedily in $G[L^j, \Sstrong[j]]$.}]
{
\hspace{0.8in}
\includegraphics[scale=1.2]{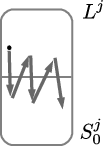}
\hspace{0.8in}
}
\hspace{0.2in}
\subfigure[{In case $x$ is mapped to $v\in\Sweak[j]$ (but not necessarily in $\Sstrong[j]$) we first embed all its children to $L^j$. To this end we make use of Condition~\ref{it:enoughspace Vj}. The rest of the embedding goes in $G[L^j, \Sstrong[j]]$.}]{
\hspace{0.6in}
\includegraphics[scale=1.2]{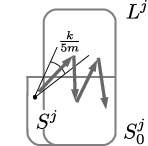}
\hspace{0.6in}
}
\caption{Embedding the tree $T(\downarrow x)$ in Lemma~\ref{lemma:secondstep}. The placement of $x$ is denoted by a black dot. The embedding the proceeds following the arrows.}
\end{figure}

The next stage is to embed the vertices of $W_X$. \Referee{(180)} Let $L^*\subseteq L^i$ be the set of unused vertices. We consider a bipartite graph~$H$ whose two colour classes are $L^*$ and $\parent(W_X)$. A pair $vx$, $v\in L^*$, $x\in \parent(W_X)$ forms an edge in~$H$ if $x$ was mapped to a vertex that is adjacent to~$v$ in~$G$. By the definition  of~$\Sstrong[i]$, and by~\eqref{eq:Am}, we have $\delta_H(\parent(W_X),L^*)\ge |L^*|-\gamma k/2$, 
and $\delta_H(L^*,\parent(W_X)\ge |\parent(W_X)|-\gamma k/2=| W_X|-\gamma k/2$. 
We conclude that $H$ has no vertex cover of size less than $\min\{|W_X|,|L^*|\}$.
As we did not embed any vertex from $W_X$ yet, and by~\eqref{eq:outtLi} we mapped to $L^i$ at most $|U_1|-(|U_1|-|L^i|)-|W_X|=|L^i|-|W_X|$ vertices, we get $|L^*|\ge  |W_X|$ and thus the minimum vertex cover has size at least $|W_X|$. By K\"onig's Theorem, there exists a matching covering $\parent(W_X)$ in $H$. This matching tells us how to embed $W_X$. In the last step, we embed~$W_2$. This can be done greedily as $\parent(W_2)$ were mapped to~$L$.

\medskip
The case $|W_1^{***}\setminus V(\tilde T)|\ge 2.9\gamma k$ is treated similarly, the difference being that this time we start with $X=\{x\in \parent(W_1^{***}\setminus V(\tilde T))\::\: \parent(x)\in U_1\}$.
\end{proof}
\section*{Acknowledgement}
We would like to thank Yi Zhao for thorough discussions over his
paper~\cite{Z07+}. This work is a part of the Masters thesis of JH written under the supervision of Daniel Kr\'al'. The second reader was Zden\v ek Dvo\v r\'ak. Dan and Zden\v ek made useful comments on previous versions of the manuscript.
Mikl\'os Simonovits and Endre Szemer\'edi encouraged us during the project. We further thank two referees for their very
detailed comments.

JH was supported in part by the grant GAUK 202-10/258009.
The work leading to these results was partially carried out while DP was affiliated  to the Institute for Theoretical Computer Science, Faculty of Mathematics and Physics, Charles University, Malostransk\'e n\'am\v{e}st\'{\i}~25, 118~00 Prague, Czech Republic, and to the Alfr\'ed R\'enyi Institute of Mathematics, Hungarian Academy of Sciences, Re\'altanoda utca 13-15, H-1053, Budapest, Hungary. The Institute for Theoretical Computer Science of Charles University is supported as project 1M0545 by Czech Ministry of Education. DP was partially supported by the FIST (Finite Structures) project, in the framework of the European Community's ``Transfer of Knowledge'' programme.

\bibliographystyle{plain}
\bibliography{bibl}

\appendix
\def\LPDJ{{\mathrm{L\ref{lemma_Embedding1}}}}
\section{Proofs of some auxiliary facts}
Proofs of several auxiliary statements
were omitted in the main body of the paper. Here we give these proofs.

\subsection{Proof of Lemma~\ref{prop_SE_embeddingwithfewleaves}}
For the proof we need the following two statements. The first one is a  simple corollary of Hall's Matching Theorem.
\begin{lemma}\label{prop_Konig_v1}
Let $K=(W_1,W_2; J)$ be a bipartite graph such that $\delta(K)\ge
\tfrac{|W_1|}2$ and $|W_1|\le |W_2|$. Then $K$ contains a matching
covering $W_1$.
\end{lemma}
Let $\ell$ be the number of leaves of $T$. Recall that $\ell<\alpha k$. Fact~\ref{fact_ManyLeavesInUnbalanced} gives that $\disc(F)<\alpha k$. In particular the lower bounds given in Properties~\ref{it:sA} and~\ref{it:BEM} of the lemma, combined with the upper bound in Property~\ref{it:sEM} yield $|A|,|B|\ge \frac{4k}{10}$.

We write $r=|B_\mathrm{d}|$, and $\mathcal{Q}=\{P_1,\ldots,P_{r}\}$.
Root $T$ at an arbitrary vertex $v\in T_\ominus$. An {\em $c$-induced path}
$a_1\ldots a_{c+1}\subset T$ is a path whose internal vertices have degree two
in $T$. Take a maximum family
$\mathcal{F}$ of vertex-disjoint 7-induced paths in $T$. We show that
$|V(\mathcal{F})|\ge k-19\ell$.

Let $D_3=\{u\in V(T)\: :\: \deg_T(u)\ge 3\}$ and $D_i=\{u\in V(T)\: :\: \deg_T(u)= i\}$ for $i=1,2$. By Fact~\ref{fact_FewLeavesManyDeg3}, we have $|D_3|<\ell$ (and $|D_2|\ge k-2\ell$). From
\begin{equation*}
2k=\sum_{u\in V(T)}
\deg(u)=|D_1|+2|D_2|+\sum_{u\in D_3}\deg(u)\ge 2k-3\ell+\sum_{u\in D_3}\deg(u) \; ,
\end{equation*}
we deduce that there are at most $3\ell +1$ maximal (w.~r.~t.\ inclusion) paths formed by vertices of degree $2$ or $1$ not containing the root $v$. On each such maximal path, at most $7$ vertices are not covered by $\mathcal {F}$. Thus the total number of vertices uncovered by $\mathcal {F}$ is at most $7(3\ell+1)+|D_3|+|\{v\}|\leq 26\ell$.
The order $\preceq_v$ naturally extends to an order on the paths of $\mathcal{F}$. For a family
$\mathcal{F}'\subset\mathcal{F}$ we write $T(\downarrow \mathcal{F}')$ to denote
all the vertices of $V(\mathcal{F}')$, and all vertices that are below some
vertex of $V(\mathcal{F}')$, i.e.,
$$T(\downarrow \mathcal{F}')=\bigcup_{u\in V(\mathcal{F}')}V(T(\downarrow u))\;
\mbox{.}$$
There is a family $\mathcal{R}\subset \mathcal{F}$ satisfying
the three properties below.
\begin{itemize}
 \item[{\bf(P1)}] $|\mathcal{R}|\le
|\mathcal{E}|+|\mathcal{M}|$.
 \item[{\bf(P2)}] $|T(\downarrow\mathcal{R})|<34\alpha k$, and
$4(|\mathcal{E}|+|\mathcal{M}|)\le
\min\{|T_\oplus\cap T(\downarrow\mathcal{R})|, |T_\ominus\cap
T(\downarrow\mathcal{R})|\}$.
 \item[{\bf(P3)}] $\mathcal{R}$ is a $\preceq_v$-antichain.
\end{itemize}
We describe a procedure how to obtain such a family $\mathcal{R}$. By an inductive construction, we first find an auxiliary family $\mathcal{R}'$, starting with
$\mathcal{R}'=\emptyset$. While
$|\mathcal{R}'|<|\mathcal{E}|+|\mathcal{M}|$ we
take a $\preceq_v$-minimal path in $\mathcal{F}$ which is not included in
$\mathcal{R}'$ and add it to $\mathcal{R}'$. From the
bound $|V(T)\setminus V(\mathcal{F})|\le 26\ell$, in each step we have that
$|T(\downarrow\mathcal{R}')|< 8|\mathcal{R}'|+26\alpha k$, and obviously
$4|\mathcal{R}'|\le\min\{|T_\oplus\cap
T(\downarrow\mathcal{R}')|,|T_\ominus\cap T(\downarrow\mathcal{R}')|\}$.
Let $\mathcal{R}$ be the $\preceq_v$-maximal elements of $\mathcal{R}'$. Hence $|T(\downarrow \mathcal R)|=|T(\downarrow \mathcal R')|$.
The properties {\bf(P1)}, {\bf(P2)}, and {\bf(P3)} are satisfied.

Set $d=5\alpha k$. Take a family $\mathcal{X}=\{X_1,\ldots,X_d\}$ of
$d$ 5-induced vertex-disjoint $T_\oplus\leftrightarrow T_\ominus\leftrightarrow
T_\oplus\leftrightarrow T_\ominus\leftrightarrow T_\oplus$ paths that avoid $\{v\}\cup T(\downarrow\mathcal{R})$. For each path $R\in \mathcal{R}$ we write $a_R$ to
denote its $\preceq_v$-maximum vertex in $T_\ominus$, and set $b_R=\children(a_R)$, $c_R=\children(b_R)$, and $d_R=\children(c_R)$. We set
$U=A\cap(V(\mathcal{E})\cup
V(\mathcal{M}))$ and $Q=A\cap V(\mathcal{Q})$.

We now describe the embedding $\psi$ of $T$. We do not have to
embed those leaves whose parents are embedded in $A$ until the very end. Indeed,  such a
partial embedding easily extends to an embedding of $T$ using Property~\ref{it:XUN} of the lemma.  We map the root $v$ to an arbitrary vertex in $A\setminus (U\cup Q)$. We
continue embedding $T$ greedily, mapping vertices from $T_\ominus$ to
$A\setminus (U\cup Q)$ and internal vertices of $T_\oplus$ to $B_\mathrm{a}$.
However, there are two exceptions in the greedy procedure:
\begin{itemize}
 \item[{\bf(S1)}] If we are about to map a vertex $b_R$ (for some $R\in\mathcal{R}$),  we skip its embedding, as well as the embedding of $T(\downarrow b_R)$.
 \item[{\bf(S2)}] If we are about to map a vertex $x_2$ which was part of some path $x_1x_2x_3x_4x_5\in\mathcal{X}$, we skip its embedding, as well as the embedding of the vertices $x_3$ and $x_4$. We continue with mapping $x_5$ to $B_\mathrm{a}$.
\end{itemize}
Observe that we are able to finish the greedy part of the embedding since the two ``skipping rules'' guarantee that both in $A$ and in $B$ at least $d>\alpha k$ vertices of $T$ remain unembedded.

In the next step, we build missing connections in the graph $H$ caused by the skipping rules.
We construct an auxiliary
bipartite graph $K_1=(O_\mathrm{a},O_\mathrm{b};E_1)$. We arbitrarily pair up
$2(d-r)$ vertices of $A\setminus (U\cup Q)$ unused by $\psi$ into pairs
$\mu_1=\{a_1^1,a_1^2\},\ldots,\mu_{d-r}=\{a_{d-r}^1,a_{
d-r}^2\}$. The remaining $r$ pairs are formed by
endvertices of the paths in $\mathcal{Q}$. We set $\mu_{i+d-r}=A\cap V(P_i)$. 
Vertices of the color class $O_\mathrm{b}$ are formed by the pairs $\mu_i$
($i\in[d]$). Vertices of the color class $O_\mathrm{a}$ are formed by the paths
in $\mathcal{X}$. A path $x_1x_2x_3x_4x_5\in \mathcal{X}$ is adjacent in $K_1$
to a pair $\mu_i$ if and only if there exists a perfect matching in the graph
$H[\{\psi(x_1),\psi(x_5)\},\mu_i]$. Since $|O_\mathrm{a}|=|O_\mathrm{b}|$ and
$\delta(K_1)\ge |O_\mathrm{a}|-2\alpha k\ge \tfrac{|O_\mathrm{a}|}2$, there
exists, by Lemma~\ref{prop_Konig_v1}, a perfect matching $M_1$ in $K_1$. The matching $M_1$ tells us where to map the vertices $x_2$ and $x_4$ of each path $x_1x_2x_3x_4x_5\in\mathcal{X}$. We extend $\psi$ accordingly on the vertices $\bigcup_{x_1x_2x_3x_4x_5\in\mathcal{X}}\{x_2,x_4\}$. If a path $x_1x_2x_3x_4x_5\in\mathcal{X}$ was matched with $\mu_{i+d-r}$ (for some $i\in[r]$) in $K_1$, we map $x_3$ to the middle vertex of the path~$P_i$. We write $\mathcal{X}'$ for those paths $x_1x_2x_3x_4x_5\in\mathcal{X}$ whose vertex $x_3$ was not yet mapped. It holds  $|\mathcal{X}'|\ge 4\alpha k$.

Let $\xi:\mathcal{R}\rightarrow U$ be an arbitrary
injective mapping. We construct another bipartite graph $K_2=(J_\mathrm{a},J_\mathrm{b}; E_2)$.
Vertices of the color class $J_\mathrm{a}$ are elements of
$\mathcal{R}\cup\mathcal{X}'$ and vertices of the color class $J_\mathrm{b}$ are
vertices of $B_\mathrm{a}$ unused by $\psi$. 
\begin{AuxiliaryCl}\label{cl:JaJb}
We have $|J_\mathrm{a}|\le |J_\mathrm{b}|$.
\end{AuxiliaryCl}
\begin{proof}
Let $W$ be the set of leaves of $T_\oplus\setminus V(\mathcal R)$. Remember that the set $W$ is mapped only at the very end of the embedding procedure. Further, for any path $x_1\ldots x_5\in \mathcal X\setminus \mathcal X'$, the vertex $x_3\in T_\oplus$ has been  mapped to $B_\mathrm{d}$, which is disjoint from $B_\mathrm{a}$. Next for each path $x_1\ldots x_5\in \mathcal X'$, the vertex $x_3\in T_\oplus$ has not been embedded, yet. Each path in $\mathcal R$ has at most one vertex in $T_\oplus$ that has already been embedded. Therefore we have
\begin{align*}
|J_\mathrm{b}|&\ge |B|-
\big(|T_\oplus|-|W|-|X\setminus X'|-|X'|-(|T_\oplus\cap V(\downarrow\mathcal{R})|-|\mathcal R|)\big)\\
&\ge |B|-
|T_\oplus|+|W|+|X|+3(|\mathcal E|+|\mathcal M|)\\
&\ge |W| +|X|   +2|\mathcal R|-1\ge |J_\mathrm{a}|+|W|+|R|-1\ge |J_\mathrm{a}|\:,
\end{align*}
where the last inequality follows from the fact that if $\mathcal R=\emptyset$ then $W\neq \emptyset$ by FactFact~\ref{fact_ManyLeavesInUnbalanced}.
\end{proof}
A path $R\in\mathcal{R}$ is
adjacent in~$K_1$ with a vertex $b\in J_\mathrm{b}$ if and only if
$b\psi(a_R)\in E(H)$ and $b\xi(R)\in E(H)$. A path $x_1x_2x_3x_4x_5\in\mathcal{X}'$ is adjacent to a vertex $b\in J_\mathrm{b}$ if and only if $b\psi(x_2)\in E(H)$ and
$b\psi(x_4)\in E(H)$. 
 Indeed, $\delta(K_1)\ge
|J_\mathrm{a}|-2\gamma k> \tfrac{|J_\mathrm{a}|}2$, and $|J_\mathrm{a}|\le
|J_\mathrm{b}|$. By Lemma~\ref{prop_Konig_v1}, there exists a matching $M_2$ in $K_2$
covering $J_\mathrm{a}$. Such a matching tells us where to map
unembedded vertices $x_3$ (in the case of a path $x_1x_2x_3x_4x_5\in\mathcal{X}'$) and
vertices $b_R$ (in the case of a path $R\in\mathcal{R}$). For a path
$R\in\mathcal{R}$ we finish embedding the part of the tree $T(\downarrow
c_R)$, extending the mapping~$\psi$. If $\psi(c_R)\in V(\mathcal{E})$ we just use the
corresponding connecting edge of $\mathcal{E}$ to map $d_R$
to $H_\kappa$ (for some $\kappa\in I$) and continue embedding $T(\downarrow
d_R)$ greedily in $H_\kappa$. If $\psi(c_R)\in V(\mathcal{M})$ we map $d_R$ to the middle vertex of the corresponding
connecting path $\mathcal{M}$ and embed the rest of
$T(\downarrow d_R)$ greedily in $H_\kappa$ (for some $\kappa\in I$). 

\subsection{Omitted proofs from Section~\ref{sec_RLproof}}
\begin{proof}[Proof of Fact~\ref{fact:zap}~\ref{it:za1}]
Let $\tilde X\subseteq X$ be the set of vertices that are not typical w.~r.~t.\ $\bigcup_{i=1}^{\ell}W_i$, i.e., for every $v\in \tilde X$ we have $\deg(v, \bigcup_{i=1}^\ell W_i)<\sum_{i=1}^\ell(d(X,Y_i)-\varepsilon)|W_i|$. Thus $e(\tilde X, \bigcup_{i=1}^\ell W_i)<|\tilde X|\cdot \sum_{i=1}^\ell (d(X,Y_i)-\varepsilon)|W_i|$. Hence, there is an index $i\in[\ell]$ such that $d(\tilde X,W_i)<d(X,Y_i)-\varepsilon$. As $W_i$ is significant and $(X,Y_i)$ is $\varepsilon$-regular, we get that $|\tilde X|\le \varepsilon |X|$.
\end{proof}

\begin{proof}[Proof of Lemma~\ref{lemma_Embedding1}]
Without loss of generality assume that $|P'|\ge \Delta$. Let us fix an arbitrary set $S_P\subset P$ with $|S_P|=\Delta$ and another set $S_Q\subset Q$ with
$|S_Q|=\Delta$. The sets $S_P$ and $S_Q$ are significant.
Choose a
vertex $v\in P'$ which is typical w.~r.~t.\ $S_Q$. There are at least
$|P'|-\varepsilon s\ge 1$ such vertices. Set $\phi(r)=v$.

We inductively extend the embedding $\phi$, so that every vertex of $t$ that is mapped to $S_P$ is typical w.~r.~t.\ $S_Q$, and so
that every vertex that is mapped to $S_Q$ is typical w.~r.~t.\ $S_P$.
We illustrate the inductive step by describing how to embed the
neighborhood of a vertex $u$ that was already mapped to $P$.
The case that $\phi(u)\in Q$ is analogous. Let
$N\subset\neighbor(u)$ be the yet unembedded neighbors of $u$. The vertex
$\phi(u)$ has at least $(d-2\varepsilon)\Delta\ge \varepsilon s+v(t)$
neighbors in~$S_Q$. At least $|N|$ of them are typical w.~r.~t.\ $S_P$ and are not yet used by $\phi$. We map $N$ to these vertices.

For the moreover part, we only need to observe that if $|P'|\ge \Delta$, there is at least one vertex in $P'$ which is typical
w.~r.~t.\ $S_Q$. We map the root~$r$ to this vertex. The second condition of the moreover part is analogous.
\end{proof}

For the proof of Lemma~\ref{lem:Embedding-3}, we need to embed the shrubs of a
given tree in an efficient way. To this end, we try to
fill the clusters of a regular pair in a balanced way. The following
definition of packedness formalizes this.
Let $X,Y,Z$ be three disjoint sets of vertices of a graph~$G$.
We say that $U\subset X\cup Y$ is $(\lambda,\tau)$-\emph{packed} with respect
to the {\em head set $Z$}  and the {\em embedding sets} $X$ and
$Y$,\footnote{the embedding sets will be typically clear, and then we only
specify the head set} if 
\begin{align}
\label{eq:packed1}
\min\{|X\cap U|,|Y\cap
U|\}&\ge\min\{\wdeg_{\mathbf{H}}(Z,X),\wdeg_{\mathbf{H}}(Z,Y)\}-\lambda, \mbox{or}\\ \label{eq:packed2}
\left||X\cap U|-|Y\cap U|\right|&\le \tau
\end{align}

\begin{proof}[Proof of Lemma~\ref{lem:Embedding-3}]
Assume that~$\mathbf {H}$ has~$N$ clusters.
Let $\tilde X\subseteq X'$ be the set of vertices that are typical
w.~r.~t.\ all but at most~$\sqrt{\varepsilon} N$ sets~$C\cap V^X$, $C\in
V(M)$, w.~r.~t.\ all but at most~$\sqrt{\varepsilon}N$ clusters $Z\in  \mathcal Z$, and 
w.~r.~t.\ the cluster~$Y$. Let $\tilde Y\subseteq Y'$ be the set of vertices that are typical w.~r.~t.\ all
but at most~$\sqrt{\varepsilon} N$ sets~$C\cap V^Y$, $C\in V(M)$ and w.~r.~t.\
the cluster~$X$.  Let $\tilde Z\subseteq \bigcup \mathcal Z$ be the set
of vertices (viewed as vertices of individual clusters of $\mathcal Z$) that are typical w.~r.~t.\ all but at most~$\sqrt{\varepsilon} N$ sets~$C\cap V^{\mathcal Z}$, $C\in
V(M)$ and 
w.~r.~t.\ the cluster~$X$. Observe that by Fact~\ref{fact:zap}, $|X'\setminus
\tilde X|\le3\sqrt{\varepsilon} s$, $|Y'\setminus
\tilde Y|\le2\sqrt{\varepsilon} s$, and for every $Z\in \mathcal Z$,
\begin{equation}
\label{eq:cafedu}
|Z\setminus \tilde Z|\le 2\sqrt{\varepsilon} s\;.
\end{equation}
Let $Q_X$ be the set of vertices (viewed as vertices of individual clusters of $V(\mathbf H)$) typical w.~r.~t.\
$\tilde X$. We define analogously $Q_Y$. For each $v\in \tilde X\cup \tilde
Y$, let
\begin{align*}
M_v&=\{CD\in M\::\:
v\textrm{ is typical w.~r.~t.\ both }C\cap V^X\textrm{ and }D\cap V^X\} \quad\mbox{if $v\in \tilde X$}\;,\\
M_v&=\{CD\in M\::\:
v\textrm{ is typical w.~r.~t.\ both }C\cap V^Y\textrm{ and }D\cap V^Y\} \quad\mbox{if $v\in \tilde Y$}\;.
\end{align*}
For each cluster $C\in V(M)$ we have  by  Fact~\ref{fact:zap},
\begin{align}\label{eq:QO}
|C\setminus Q_X|,|C\setminus Q_Y|&\le \varepsilon s\;,\quad\mbox{ and }\\
\label{eq:sizeMv}
|M_v|&\ge
|M|-2\sqrt{\varepsilon} N\;.
\end{align}

The embedding of~$F$ is divided into $w$ steps, where $w=|W_X\cup W_Y|$. We label the vertices of $W_X\cup
W_Y$ as $x_1,\ldots,x_w$, indexing from an arbitrary fixed root~$R\in W_X\cup W_Y$ downwards, i.e., in such
way that $j_1\le j_2$ whenever $x_{j_1}\succeq_R x_{j_2}$. 
We denote by~$\varphi$ the partial embedding of~$F$. For a set $U\subset V(F)$,
$\varphi(U)$ refers to the image of the already embedded part of $U$ at that
moment.\footnote{In particular, one may have $|\varphi(U)|<|U|$.} In step $i\geq
1$, we embed the tree $$F_i=F\Big[\{x_i\}\cup\bigcup_{\ell\in[c_i]}V(t_i^\ell)\Big]\;\mbox{,}$$ where
$t_i^1,\ldots,t_i^{c_i}$ are the components $t\in \mathcal D_X\cup \mathcal D_Y$ such that $\children(x_i)\cap V(t)\not=\emptyset$.  Set $V_i^\ell=\bigcup_{j<i}V(F_j)\cup\bigcup_{j<\ell}V(t_i^j)$, and
$U_i^\ell=\varphi(V_i^\ell)$. We call the embedding $\varphi$ \emph{equable} at step $i$ and substep~$\ell$, if for each $CD\in M$, we have 
$||U_i^\ell\cap V^{\mathcal Z}\cap C|- |U_i^\ell\cap V^{\mathcal Z}\cap D||\le \tau$. During the embedding procedure, we use an auxiliary set $\mathcal Z'\subseteq \mathcal Z$ of ``active'' clusters in $\mathcal Z$.

For~$i=1$, set~$N_i=\tilde X\cup \tilde Y$ and $\mathcal Z'=\mathcal Z$. For~$i>1$, let $p_i=\parent(x_i)$
and set $N_i=\neighbor_H(\varphi(p_i))\cap (\tilde X\cup \tilde Y)$. 
During the embedding process we will keep the following three properties in
every step~$i\in [w]$, and every substep~$j\in [c_i]$.
\begin{itemize}
\item[{\bf (I1)}] 
For each $CD\in M$, the set $U_i^j\cap V^X\cap(C\cup D)$ is $(\tfrac {8\varepsilon s}{d}, \tau)$-packed w.~r.~t.\ the head set~$X$ and
the set $U_i^j\cap V^Y\cap(C\cup D)$ is $(\tfrac {8\varepsilon s}{d}, \tau)$-packed w.~r.~t.\ the head set $Y$.
\item[{\bf (I2)}] $|N_i\cap X|\ge |W_X|$ and $|N_i\cap Y|\ge |W_Y|$.
\item[{\bf (I3)}] $\varphi(W_X)\subseteq \tilde X $, $\varphi(W_Y)\subseteq
\tilde Y $, $\varphi(\mathcal D_Y)\subseteq V^Y$, $\varphi(\mathcal
D_1)\subseteq V^X\setminus V(M_X)$, $\varphi(\mathcal D_2)\subseteq V^X\cap
V(M_X)$, $\varphi(\mathcal D_3\setminus \neighbor_F(W_X))\subseteq V^{\mathcal Z}$, and
$\varphi(\mathcal D_3\cap \neighbor_F(W_X))\subseteq \bigcup \mathcal Z$.
\item [{\bf (I4)}] Either the embedding $\varphi$ is equable and $\mathcal Z'=\mathcal Z$, or for every $CD\in M$ and every $Z\in \mathcal Z'$ we have $$\min\{\wdeg_{\mathbf {H}}(Z, C\cap V^{\mathcal Z}),\wdeg_{\mathbf {H}}(Z, D\cap V^{\mathcal Z})\}\le \min \{ |(C\cap \varphi (\mathcal D_3)|,|(D\cap \varphi (\mathcal D_3)|\}+\frac {8\varepsilon s}{d}\;,$$ and $\wdeg_{\mathbf {H}}(X, \bigcup \mathcal Z')\ge |(V(\mathcal D_3)\cap \neighbor_{F}(W_X))\setminus V_i^j|+|U_i^j\cap \bigcup \mathcal Z'|+\frac {\xi n}2\;.$
\end{itemize}

For $i=1$ and $j=1$, {\bf (I1)}, {\bf (I3)}, and {\bf (I4)} hold trivially. Further, $
\max\{|W_X|,|W_Y|\}\le \tfrac{12 k}\tau\ll \varepsilon s\le \min\{|\tilde
X|,|\tilde Y|\}$ by Definition~\ref{def:ellfine}~\ref{ellfine:max}, yielding {\bf
(I2)}.

We now proceed with a general step. We first give two claims which we then make use of for the embedding itself.
\begin{AuxiliaryCl}\label{rtyhk}$~$
\begin{enumerate}[label={$(\alph{*})$}]
\item\label{BHX1} Suppose that $\mathcal D_Y\neq \emptyset$. Then for every
$v\in \tilde Y$, there is an edge $CD\in M_v$ such that
\begin{equation*}
\deg(v,(C\cup D)\cap V^Y)\ge |\varphi(\mathcal D_Y)\cap (C\cup D)|+2\tau
+\frac {\xi s}2\;.
\end{equation*} 
\item\label{BHX2} Suppose that $\mathcal D_1\neq \emptyset$. Then for every $v\in \tilde X$, there is an edge $CD\in
M_v\setminus M_X$ such that
\begin{equation*}
\deg(v,(C\cup D)\cap V^X\cap Q_X)\ge |\varphi(\mathcal D_1)\cap (C\cup D)|+2\tau
+\frac {\xi s}2\;.
\end{equation*} 
\item \label{aux:XMX}
Suppose that $\mathcal D_2\neq \emptyset$. Then for every $v\in \tilde X$, there is an edge $CD\in M_X\cap M_v$ such that 
\begin{equation}\label{eq:tinD2}
\deg(v,C\cap V^X\cap Q_X)\ge |\varphi(\mathcal D_2)\cap C|
+\frac {\xi s}2\quad\mbox{and}  \quad |D\cap V^X|\ge |\varphi(\mathcal
D_2)\cap D|+\frac {\xi s}2\;.
\end{equation} 
\item \label{lodw}
Suppose that $\mathcal D_3\neq \emptyset$. Then for every $v\in \tilde X$, there is a cluster $Z\in \mathcal Z'$ such that 
\begin{equation}\label{eq:tinD3}
\deg(v, Z\cap\tilde Z)\ge |\varphi(\mathcal D_3)\cap Z|
+\frac {\xi s}4\;.
\end{equation} 
\end{enumerate} 
\end{AuxiliaryCl}
\begin{proof} 
\ref{BHX1}~Using the typicality of $v$, we get
\[\sum_{CD\in M_v}\deg(v, (C\cup D)\cap V^Y)\geBy{\eqref{eq:sizeMv}}
\wdeg_{\mathbf{H}}(Y,V^Y)-2\sqrt{\varepsilon}Ns-\varepsilon n\geBy{\ref{emb3-V^Y}}v(\mathcal
D_Y)+\frac {3\xi n}{4}\;,\]
which implies the statement.

The proof of~\ref{BHX2} is analogous,
using~\eqref{eq:QO} and~\ref{emb3-V^X}.

\ref{aux:XMX} By~\eqref{eq:sizeMv} and by the typicality
of~$v$, we have
\begin{align*}
\sum_{C\in V(M_X\cap M_v)\cap \neighbor_{\mathbf H}(X)}\deg(v, C\cap
V^X)
&\ge \sum_{CD\in M_X}\deg(v, (C\cup D)\cap
V^X)-2\sqrt{\varepsilon}Ns\\ 
&\ge \wdeg_{\mathbf{H}}(X, V^X\cap \bigcup V(M_X))-\varepsilon n-2\sqrt{\varepsilon}Ns\\
&\geBy{\ref{emb3-MX}}v(\mathcal D_2)-c^2k+\xi
n-3\sqrt{\varepsilon}Ns\;.
\end{align*}
As $\mathcal D_2$ is $c$-balanced, we get
$
v(\mathcal D_2)
\ge  c^2k+\sum_{CD\in M_X\cap M_v}\max\{|\varphi
(\mathcal D_2)\cap C|, |\varphi
(\mathcal D_2)\cap D|\}
$.
So, there is an edge $CD\in M_X\cap M_v$ such that 
\[|D\cap V^X|\geBy{\ref{emb3-balanced}} \deg(v,C\cap V^X)\ge \max\{|\varphi(\mathcal
D_2)\cap C|, |\varphi(\mathcal D_2)\cap
D|\}+\xi s-3\sqrt{\varepsilon}s\;.\]
Together with~\eqref{eq:QO}, we get~\eqref{eq:tinD2}.

\ref{lodw} 
The vertex $v$ is typical w.~r.~t.\ all but at most $\sqrt{\varepsilon} N$ clusters $Z\in  \mathcal Z'$.

First assume that~$\varphi$ is equable and $\mathcal Z'=\mathcal Z$. We have
\begin{align*}
\deg(v,\tilde Z\cap \bigcup \mathcal Z')
\;&\geBy{$\mathcal Z'=\mathcal Z$}\; \deg(v,\bigcup \mathcal
Z)-\left|\bigcup \mathcal Z\setminus \tilde Z\right|
\;\geBy{\eqref{eq:cafedu}}\; \wdeg_{\mathbf{H}}(X,\bigcup \mathcal
Z)-\varepsilon n
-(1+2)\sqrt{\varepsilon}N s\\
&\geBy{\ref{emb3-AZ}}|V(\mathcal D_3)\cap \neighbor_F(W_X)|+\xi n -
4\sqrt{\varepsilon}n\;.
\end{align*}
As by~{\bf (I3)} only $V(\mathcal
D_3)\cap \neighbor_F(W_X)$ is mapped to $\bigcup \mathcal Z=\bigcup \mathcal Z'$,  there exists a
cluster $Z\in \mathcal Z'$ satisfying~\eqref{eq:tinD3}.

If~$\varphi$ is not equable,  we get
\begin{align*}
\deg(v,\tilde Z\cap \bigcup \mathcal Z')&\ge \deg(v,\bigcup \mathcal
Z')-\left|\bigcup \mathcal Z\setminus \tilde Z\right|\ge \wdeg_{\mathbf{H}}(X,\bigcup \mathcal
Z')-\varepsilon n
-(1+2)\sqrt{\varepsilon}N s\\
&\geBy{{\bf (I4)}}|V(\mathcal D_3)\cap \neighbor_F(W_X)\setminus V_i^j|+|U_i^j\cap \bigcup \mathcal Z'|+\xi n/2 -
4\sqrt{\varepsilon}n\;.
\end{align*}
As by~{\bf (I3)} only $V(\mathcal
D_3)\cap \neighbor_F(W_X)$ is mapped to $\bigcup \mathcal Z'$,  there exists a
cluster $Z\in \mathcal Z'$ satisfying~\eqref{eq:tinD3}.
\end{proof}

\begin{AuxiliaryCl}
\label{BHX3} Suppose that $\mathcal D_3\neq \emptyset$. Then for every vertex $v\in \tilde Z$, there is an edge  $CD\in M_{v}$ such that 
\begin{equation}\label{eq:degZ-CD}
\deg(v,(C\cup D)\cap V^{\mathcal Z})\ge |\varphi(\mathcal D_3)\cap (C\cup D)|+2\tau +
2\varepsilon s +\frac {\xi s}2\;.
\end{equation}
\end{AuxiliaryCl}

\begin{proof}
Suppose that $v$ lies in a cluster $Z$. Using the typicality of $v$, we get
\[\sum_{CD\in M_{v}}\deg(v, (C\cup D)\cap V^{\mathcal Z})\geBy{\eqref{eq:sizeMv}}
\wdeg_{\mathbf{H}}(Z,V^{\mathcal Z})-4\sqrt{\varepsilon}Ns-\varepsilon n\geBy{\ref{emb3-Z}}v(\mathcal
D_3)+\frac {3\xi n}{4}\;.\]
\end{proof}

Assume that we are in step $i\ge 1$ and that we want to embed the forest $F_i$.
By~{\bf(I2)}, we can map~$x_i$ to an unused vertex in~$N_i$ (in $N_i\cap X$ if $x_i\in W_X$, and in $N_i\cap Y$ if $x_i\in W_Y$). Observe that
$\varphi(x_i)$ has at least $(d-\varepsilon)s-d s/2-3\sqrt{\varepsilon}s\ge
\varepsilon s\ge |W_X\cup W_Y|$ neighbors in $\tilde X$ or in $\tilde Y$
(depending whether $\varphi(x_i)\in \tilde Y$ or $\varphi(x_i)\in \tilde X$).
This ensures that {\bf(I2)} still holds. Assume that we are in substep $j\in [c_i]$, i.e., we have already embedded the components $t_i^1,\dots, t_i^{j-1}$ and that we want to embed the component~$t_i^j$.
\begin{itemize}
  \item [\bf{(1)}] If $t_i^j\in \mathcal D_Y$, pick an edge $CD\in
  M_{\varphi(x_i)}$ as in Claim~\ref{rtyhk}~\ref{BHX1}. 
  We use Lemma~\ref{lemma_Embedding1} to embed $t_i^j$ (where the root
  of~$t_i^j$ is the neighbor of~$x_i$) with the following setting. \begin{align*}P'&=\neighbor_H(\varphi(x_i))\cap C\cap V^Y\setminus
  U_i^j\qquad P= C\cap V^Y\setminus U_i^j\subseteq C \;,\\
Q'&=\neighbor_H(\varphi(x_i))\cap D\cap
  V^Y\setminus U_i^j\qquad Q=D\cap V^Y\setminus U_i^j\subseteq
  D\;,\end{align*} and $\Delta=\tfrac {4\varepsilon s}{d}$. We have that \[\max\{|P'|,|Q'|\}\ge \frac
  12 \deg(\varphi(x_i), (C\cup D)\cap V^Y\setminus U_i^j)\ge \frac 12
  (2\tau+\tfrac{\xi s}2)\ge \tfrac {4\varepsilon s}{d}\;,\]
  which verifies one of the assumption of Lemma~\ref{lemma_Embedding1}. We use Lemma~\ref{lemma_Embedding1} differently in cases
  \begin{align}\label{eq:c1E1}
  \min\{|\varphi(\mathcal D_Y)\cap C|,|\varphi(\mathcal D_Y)\cap
  D|\}&< \min\{\wdeg_{\mathbf{H}}(X, C\cap V^Y), \wdeg_{\mathbf{H}}(X,D\cap V^Y)\}-\tfrac{8\varepsilon s}{d}\quad{\mbox{and}}\\
\label{eq:c2E1}
  \min\{|\varphi(\mathcal D_Y)\cap C|,|\varphi(\mathcal D_Y)\cap
  D|\}&\ge \min\{\wdeg_{\mathbf{H}}(X, C\cap V^Y), \wdeg_{\mathbf{H}}(X,D\cap V^Y)\}-\tfrac{8\varepsilon s}{d}
    \end{align}
  Suppose first that we do not  have~\eqref{eq:c1E1}. Thus in particular, the packedness of $U_i^j\cap V^X\cap(C\cup D)$ in {\bf (I1)} has the form of~\eqref{eq:packed2}. Then 
  \begin{align*}
  \min\{|P'|,|Q'|\}&=\min\{\deg(\varphi(x_i), C\cap V^Y\setminus U_i^j),\deg(\varphi(x_i), D\cap
  V^Y\setminus U_i^j)\}\\   
  &\geBy{{\bf (I3)}} \min\{\deg(\varphi(x_i), C\cap
  V^Y),\deg(\varphi(x_i), D\cap V^Y)\}-\max\{|\varphi(\mathcal D_Y)\cap C|,
  |\varphi(\mathcal D_Y)\cap D|\}\\ 
  &\geBy{{\bf (I1)}} \min\{\wdeg_{\mathbf{H}}(X, C\cap V^Y),\wdeg_{\mathbf{H}}(X, D\cap V^Y)\}-\varepsilon s\\
  &~~~~~-\min\{|\varphi(\mathcal D_Y)\cap C|, |\varphi(\mathcal D_Y)\cap D|\}-\tau\\ &\geBy{\eqref{eq:c1E1}} \tfrac {8\varepsilon s}{d}-\varepsilon s -\tau\ge \tfrac {4\varepsilon s}{d}\;,
  \end{align*}  
  which allows us to use the ``moreover'' part of Lemma~\ref{lemma_Embedding1}. We can then choose in this case to
  which set $P'$ or $Q'$ we map the root of~$t_i^j$.  We thus can ensure that
  $\big||\varphi(\mathcal D_Y)\cap C|-|\varphi(\mathcal D_Y)\cap D|\big|\le \tau$ still holds after embedding~$t_i^j$, yielding~{\bf{(I1)}}. 
  
  Suppose now that~\eqref{eq:c1E1} holds.
  Then 
\begin{align*}
\min\{|P|,|Q|\}&=\min\{|C\cap V^Y\setminus U_i^j|,|D\cap V^Y\setminus U_i^j|\}\\
&\ge 
\max\{\deg(\varphi(x_i), C\cap V^Y), \deg(\varphi(x_i),D\cap
V^Y)\}-\max \{|\varphi(\mathcal
D_Y)\cap C|,|\varphi(\mathcal D_Y)\cap D|\}\\ 
\ge &\deg(\varphi(x_i), (C\cup D)\cap
V^Y)-|\varphi(\mathcal D_Y)\cap (C\cup D)|\\&-\min\{\deg(\varphi(x_i), C\cap
V^Y), \deg(\varphi(x_i),D\cap V^Y)\}+\min\{|\varphi(\mathcal D_Y)\cap
C|,|\varphi(\mathcal D_Y)\cap D|\}\\
\geBy{\eqref{eq:c1E1}} & 2\tau +\tfrac {\xi s}{2}-\tfrac {8\varepsilon s}{d}-\varepsilon s\ge
\tfrac {4\varepsilon s}{d}\;,
\end{align*}
which indeed allows us to embed $t^j_i$ using Lemma~\ref{lemma_Embedding1} in this case.
After the embedding of $t_i^j$ in this case, {\bf(I1)} holds trivially. 

In both
cases, {\bf (I2)} holds, as~$\mathcal D_Y$ contains only end-shrubs. The
tree~$t_i^j$ was embedded in $(C\cup D)\cap V^Y$, ensuring~{\bf (I3)}. {\bf (I4)} is immaterial in this step as nothing was done regarding $V^{\mathcal Z}$ or $\mathcal D_3$.
  
\item [\bf{(2)}] If $t_i^j\in \mathcal D_1$, pick an edge $CD\in
M_{\varphi(x_i)}\setminus M_X$ as in Claim~\ref{rtyhk}~\ref{BHX2}. The
embedding is done analogously to the case {\bf (1)}, setting
\begin{align*}
P'&=
\neighbor_H(\varphi(x_i))\cap C\cap V^X\cap Q_X\setminus
U_i^j
\qquad P= C\cap V^X\cap Q_X\setminus U_i^j\subseteq C \;,\\
Q'&=
\neighbor_H(\varphi(x_i))\cap D\cap V^X\cap Q_X\setminus
U_i^j
\qquad Q= D\cap V^X\cap Q_X\setminus U_i^j\subseteq D \;.
\end{align*}
As~$\varphi(V(t_i^j))\subseteq Q_X$, every vertex in $V(t_i^j)\cap
\neighbor_F(W_X)$ is mapped to a vertex that has at least
$(d-\varepsilon)|\tilde X|\ge |W_X|$ neighbours in $\tilde X$, ensuring {\bf
(I2)}. Conditions {\bf (I1)} and {\bf (I3)} are maintained as in case {\bf (1)}. Again, {\bf (I4)} is maintained automatically.

\item[\bf{(3)}]  If
$t_i^j\in \mathcal D_2$, we pick an edge $CD\in M_{\varphi(x_i)}\cap M_X$ as in
Claim~\ref{aux:XMX}. We use Lemma~\ref{lemma_Embedding1} with the following
setting.
\begin{align*}
P'&=\neighbor_H(\varphi(x_i))\cap C\cap V^X\cap Q_X\setminus U_i^j\subseteq 
C\cap V^X\cap Q_X\setminus U_i^j\subseteq C\;,\\
Q'&=\emptyset\subseteq D\cap V^X\setminus U_i^j\subseteq D\;,
\end{align*}
and $\Delta=\tfrac {4\varepsilon s}{d}$. The
requirements on $\max\{|P'|,|Q'|\}$, and $\min\{|P|,|Q|\}$ are fulfilled
by~\eqref{eq:tinD2}. We get an embedding of $t_i^j$ in $(C\cup D)\cap V^X$
(ensuring {\bf {(I3)}}) such that every vertex at even distance to the root of
$t_i^j$ is mapped to $Q_X$. Therefore its image sends at least $(d-\varepsilon)|\tilde X|\ge |W_X|$
edges to $\tilde X$ (ensuring {\bf(I2)}). The condition {\bf (I1)} trivially
holds by the property of $M_X$. 

\item [\bf{(4)}] Suppose that $t_i^j\in \mathcal D_3$.

First we consider the case, when there is a cluster $Z\in \mathcal Z$ such that 
\begin{itemize}
\item[(*)] $\deg(\varphi(x_i), Z\cap\tilde Z)\ge |U_i^j\cap Z|
+\frac {\xi s}4\;,$ and 
\item[(**)] there is an edge $CD\in M$ such that $\wdeg(Z,C\cap V^{\mathcal Z})\ge |U_i^j\cap C\cap V^{\mathcal Z}|+\tau +
\varepsilon s+\frac {3\varepsilon s+\tau}{d-2\varepsilon}$, and $\wdeg(Z,D\cap V^{\mathcal Z})\ge |U_i^j\cap D\cap V^{\mathcal Z}|+\tau +
\varepsilon s+\frac {3\varepsilon s+\tau}{d-2\varepsilon}$.
\end{itemize}

Then we embed $t_i^j$ in $Z\cup C\cup D$ as follows. 
We map the root $r$ of $t_i^j$ to an unused vertex  $v\in Z\cap \tilde Z$ that is typical w.r.t. $C\cap V^{\mathcal Z}$ and typical w.r.t. $D\cap V^{\mathcal Z}$. 
By Fact~\ref{fact:zap} there are at least $\frac{\xi s}{4}-2\varepsilon s>0$ such vertices. By (**), the vertex $v$ satisfies
\begin{equation}\label{eq:balance neighbourhoods}
\begin{split}
\deg(v,(C\cap V^{\mathcal Z})\setminus U_i^j)&\ge \tau+\frac {3\varepsilon s+\tau}{d-2\varepsilon}\;, \mbox{ and}\\
\deg(v,(D\cap V^{\mathcal Z})\setminus U_i^j)&\ge \tau+\frac {3\varepsilon s+\tau}{d-2\varepsilon}\;.
\end{split}
\end{equation}
Let $K\subset C\cup D$ be the set of vertices  that are  typical (where typicality refers to $C$ or $D$, respectively) w.r.t.~$(Z\cap \tilde Z)\setminus U_i^j$. Note that the set $(Z\cap \tilde Z)\setminus U_i^j$ is significant by~(*). By Fact~\ref{fact:zap},
\begin{equation}\label{eq:tilde CD-small}
|C\setminus K|,|D\setminus K|\le \varepsilon s\;.
\end{equation}

Let $t_{even}$ be the set of vertices in $V(t_i^j)\setminus \{r\}$ of even distance from $r$, and let $t_{odd}$ be the ones of odd distance. 
If $|t_{odd}|<|t_{even}|$ and $|(C\cap V^{\mathcal Z})\setminus U_i^j|\le |(D\cap V^{\mathcal Z})\setminus U_i^j|$, or $|t_{odd}|\ge |t_{even}|$ and $|(C\cap V^{\mathcal Z})\setminus U_i^j|> |(D\cap V^{\mathcal Z})\setminus U_i^j|$, set $X_{\LPDJ}=D$, and $Y_{\LPDJ}=C$. Otherwise set $X_{\LPDJ}=C$, and $Y_{\LPDJ}=D$. 

Consider the set $\mathcal T^{r}$ of components of $t_i^j-\neighbor_F(W_X)$ that are incident to $r$. By Definition~\ref{def:ellfine}(ix), $V(t_i^j)\cap \neighbor_F(W_X)$ has one or two elements.
If $r$ is the only element in $V(t_i^j)\cap \neighbor_F(W_X)$ then $\mathcal T^{r}$ contains all the components of $t_i^j\setminus\{r\}$. 
We embed the elements $t\in \mathcal T^{r}$ one after the other using Lemma~\ref{lemma_Embedding1}. At each application of Lemma~\ref{lemma_Embedding1} we use 
$P=(X_{\LPDJ}\cap K\cap V^{\mathcal{Z}})\setminus \varphi(\mathcal D_3)$, $P'=P\cap \neighbor(v)$, $Q=(Y_{\LPDJ}\cap K\cap V^{\mathcal{Z}})\setminus\varphi(\mathcal D_3)$, and $Q'=Q\cap \neighbor(v)$. 
By~\eqref{eq:balance neighbourhoods} and~\eqref{eq:tilde CD-small} we have  $\min\{|P'|,|Q'|\}\ge \frac {3\varepsilon s+\tau}{d-2\varepsilon}-\varepsilon s\ge \frac {\varepsilon s+\tau}{d-2\varepsilon s}$. 
By the ``moreover'' part of Lemma~\ref{lemma_Embedding1} we can ensure that the root of $t$ (i.e.\ the unique vertex in $V(t)\cap \neighbor_F(r)$) is mapped to the set $P'$.

If $r$ is the only element in $V(t_i^j)\cap \neighbor_F(W_X)$, then we are done with embedding $t_i^j$. Otherwise, let $r'$ be the second vertex in $V(t_i^j)\cap \neighbor_F(W_X)$. 
The predecessor of $r'$ is mapped on a vertex $u\in K$. Since $u$ is typical w.r.t.\ the set $(Z\cap \tilde Z)\setminus U_i^j$, we have
\begin{equation}\label{eq:degv'}
\deg(u, (Z\cap \tilde Z)\setminus U_i^j)\ge (d-\varepsilon)|(Z\cap \tilde Z)\setminus U_i^j|\geBy{(*)}  (d-\varepsilon)\frac{\xi s}4> 3\varepsilon s\;.
\end{equation}
We can thus map the vertex $r'$ to an unused vertex $v'\in (Z\cap \tilde Z)\cap \neighbor (u)$ that is typical w.r.t.~$C\cap V^{\mathcal Z}$ and typical w.r.t.~$D\cap V^{\mathcal Z}$.

Consider the set $\mathcal T^{r'}$ of components of $t_i^j-\{r'\}$ that are incident to $r'$ and does not contain~$r$. See Figure~\ref{fig:TrTr}. 
\begin{figure}[t]
\centering 
\includegraphics[scale=0.7]{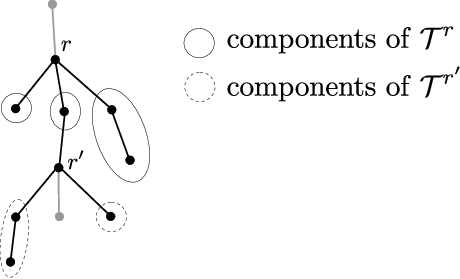}
\caption{The components $\mathcal T^r$ and $\mathcal T^{r'}$. Vertices of $W_X$ are shown in gray.}
\label{fig:TrTr}
\end{figure}
We embed the elements $t\in \mathcal T^{r'}$ one after the other using Lemma~\ref{lemma_Embedding1}.  At each application of Lemma~\ref{lemma_Embedding1} we use 
$P=(X_{\LPDJ}\cap K\cap V^{\mathcal{Z}})\setminus \varphi(\mathcal D_3)$, $P'=P\cap \neighbor(v)$, $Q=(Y_{\LPDJ}\cap K\cap V^{\mathcal{Z}})\setminus\varphi(\mathcal D_3)$, and $Q'=Q\cap \neighbor(v)$. 
By~(**), the vertex $v'$ satisfies
\begin{align*}
\deg(v',(C\cap V^{\mathcal Z})\setminus\varphi(\mathcal D_3))&\ge \frac {3\varepsilon s+\tau}{d-2\varepsilon}\;, \mbox{ and}\\
\deg(v',(D\cap V^{\mathcal Z})\setminus \mathcal D_3))&\ge \frac {3\varepsilon s+\tau}{d-2\varepsilon}\;.
\end{align*}
By~\eqref{eq:tilde CD-small}, we have that $\min\{|P'|,|Q'|\}\ge \frac {3\varepsilon s+\tau}{d-2\varepsilon}-\varepsilon s\ge \frac {\varepsilon s+\tau}{d-2\varepsilon s}$. We can thus use the ``moreover'' part of Lemma~\ref{lemma_Embedding1} to ensure that the root of $t$ (i.e.\ the unique vertex in $V(t)\cap \neighbor_F(r')$) is mapped to the set $P'$.

As we embedded $V(t_i^j)\cap \neighbor(W_X)$ in $Z$ and the rest of $t_i^j$ in $(C\cup D)\cap V^{\mathcal Z}$, properties {\bf (I1)}, {\bf (I2)}, and {\bf (I3)} trivially hold. As for {\bf (I4)}, we did not alter the set $\mathcal Z'$ and the set~$\varphi(\mathcal D_3)$ may have only increased. Also observe that by {\bf (I3)} we have that $|V(t_i^j)\cap \neighbor_F(W_X)|\ge |\varphi(t_i^j)\cap \bigcup \mathcal Z'| $. 
Therefore, it is enough to show that if $\varphi$ was equable at the substep~$j$, it is still equable at substep~$j+1$ (i.e.\ after we embedded $t_i^j$). This was guaranteed by the choice of $X_{\LPDJ}$ and $Y_{\LPDJ}$, so that to minimise the difference  between $|\varphi(\mathcal D_3)\cap (C\cap V^{\mathcal Z})|$ and $|\varphi(\mathcal D_3)\cap (D\cap V^{\mathcal Z})|$ together with the fact that $v(t_i^j)\le \tau$.

Now consider the case when there is no cluster $Z\in \mathcal Z$ that satisfies (*) and (**). If~$\varphi$ is  equable and $\mathcal Z'=\mathcal Z$ then we redefine $\mathcal Z'$ to be the set of clusters in $\mathcal Z$ with respect to which the vertex~$\varphi(x_i)$ is typical and for which (*) holds. 
We want to check {\bf (I4)} for this new set $\mathcal Z'$.

As $\varphi(x_i)\in \tilde X$ by {\bf (I2)}, we have that $\varphi(x_i)$ is typical to all but at most $\sqrt{\varepsilon}N$ clusters of $\mathcal Z$. Therefore,
\begin{align*}\wdeg_{\mathbf{H}}(X, \bigcup \mathcal Z')&\ge \wdeg_{\mathbf{H}}(X, \bigcup \mathcal Z)-\sqrt{\varepsilon} n- \left(|U_i^j\cap \bigcup (\mathcal Z\setminus \mathcal Z')|+\frac {\xi n}{4}\right)\\
&\ge |V(\mathcal D_3)\cap \neighbor_F(W_X)|- |U_i^j\cap \bigcup (\mathcal Z\setminus \mathcal Z')|+\xi n-\sqrt{\varepsilon}n-\frac{\xi n}{4}\\
&\geBy{{\bf(I3)}}  |(V(\mathcal D_3)\cap \neighbor_F(W_X))\setminus V_i^j|+|U_i^j\cap \bigcup \mathcal Z'|+\frac {\xi n}{2}\;.
\end{align*}
By the definition of $\mathcal Z'$, (**) does not hold for any cluster $Z\in \mathcal Z'$. Then, as $\varphi$ is equable, we have  
\begin{align*}
\min&\{\wdeg(Z, C\cap V^{\mathcal Z}), \wdeg(Z, D\cap V^{\mathcal Z})\}\\
&\le \max\{|U_i^j\cap C\cap V^{\mathcal Z}|,|U_i^j\cap D\cap V^{\mathcal Z}|\}+\tau+\varepsilon s+\frac {3\varepsilon s+\tau}{d-2\varepsilon}\\
&\le \min\{|U_i^j\cap C\cap V^{\mathcal Z}|,|U_i^j\cap D\cap V^{\mathcal Z}|\}+2\tau +\varepsilon s+\frac {3\varepsilon s+\tau}{d-2\varepsilon}\\
&\leBy{{\bf (I3)}} \min\{|\varphi(\mathcal D_3)\cap C|,|\varphi(\mathcal D_3)\cap D|\}+\frac {8\varepsilon s}{d}\;,
\end{align*} 
showing that the newly created set $\mathcal Z'$ satisfies  {\bf (I4)}.

So we may assume that the second condition of {\bf (I4)} is satisfied.
Let  $Z\in \mathcal Z'$ be a cluster as in Claim~\ref{rtyhk}~\ref{lodw} and map the root of~$t_i^j$ to a vertex $v\in Z\cap \tilde Z$. Then pick an edge $CD\in M_v$ as in Claim~\ref{BHX3}. Let $K\subset C\cup D$ be the set of vertices  that are  typical (where typicality refers to $C$ or $D$, respectively) w.r.t.~$(Z\cap \tilde Z)\setminus U_i^j$.

Without loss of generality, assume that $|\neighbor(v,K\cap D\cap V^{\mathcal Z}\setminus U_i^j)|\le |\neighbor(v,K\cap C\cap V^{\mathcal Z}\setminus U_i^j)|$. Let $X_{\LPDJ}=C$ and $Y_{\LPDJ}=D$.
Consider the set $\mathcal T^{r}$ of components of $t_i^j-\neighbor_F(W_X)$ that are incident to $r$. 
We embed the elements $t\in \mathcal T^{r}$ one after the other using Lemma~\ref{lemma_Embedding1} with the following
setting,
\begin{align*}
P&=(X_{\LPDJ}\cap K\cap V^{\mathcal{Z}})\setminus \varphi(\mathcal D_3)\;,
&P'&=P\cap \neighbor(v)\;,\\
Q&=(Y_{\LPDJ}\cap K\cap V^{\mathcal{Z}})\setminus \varphi(\mathcal D_3)\;,
&Q'&=Q\cap \neighbor(v)\;.
\end{align*}
By~\eqref{eq:degZ-CD} we have $|(P'\cup Q')|\ge\frac {\xi s}{4}-2\varepsilon s$. 
 As by assumption we have $|P'|\ge |Q'|$, we get $|P'|\ge \frac {\xi s}{8}-\varepsilon s\ge \frac {\varepsilon s+\tau}{d-2\varepsilon}. $

From~\eqref{eq:degZ-CD} we derive
\begin{align*}
|Q|&\ge |\tilde D\cap V^{\mathcal Z}|-|(D\cap V^{\mathcal Z})\cap U_i^j| = |\tilde D\cap V^{\mathcal Z}|-|((C\cup D)\cap V^{\mathcal Z})\cap U_i^j|+|(C\cap V^{\mathcal Z})\cap U_i^j|\\
&\ge |\tilde D\cap V^{\mathcal Z}|-\left( \deg(v,(C\cup D)\cap V^{\mathcal Z})-2\tau -
2\varepsilon s -\frac {\xi s}2\right)+|(C\cap V^{\mathcal Z})\cap U_i^j|\\
& = |\tilde D\cap V^{\mathcal Z}|-\max\left\{\deg(v, C\cap V^{\mathcal Z}),\deg(v, D\cap V^{\mathcal Z})\right\}+2\tau +
2\varepsilon s +\frac {\xi s}2\\
&-\min\left\{\deg(v, C\cap V^{\mathcal Z}),\deg(v, D\cap V^{\mathcal Z})\right\}+
\min\left\{|(C\cap V^{\mathcal Z})\cap U_i^j|,|(D\cap V^{\mathcal Z})\cap U_i^j|\right\}\\
&\ge 2\tau +
\varepsilon s +\frac {\xi s}2 -\left(\min\left\{\wdeg(Z, C\cap V^{\mathcal Z}),\wdeg(Z, D\cap V^{\mathcal Z})\right\}+\varepsilon s\right)\\&+
\min\left\{|(C\cap V^{\mathcal Z})\cap U_i^j|,|(D\cap V^{\mathcal Z})\cap U_i^j|\right\}\\
&\geBy{{\bf(I4)}}  2\tau +
\varepsilon s +\frac {\xi s}2-\varepsilon s\ge \frac {\varepsilon s+\tau}{d-2\varepsilon}\;.
\end{align*}

If $r$ is the only element in $V(t_i^j)\cap \neighbor_F(W_X)$, then we are done with embedding $t_i^j$. Otherwise, let $r'$ be the second vertex in $V(t_i^j)\cap \neighbor_F(W_X)$. The predecessor of $r'$ is mapped on a vertex $u\in K$ that is typical w.r.t.\ the set $(Z\cap \tilde Z)\setminus U_i^j$ and hence satisfies~\eqref{eq:degv'}. We can thus mapped the vertex $r'$ to an unused vertex $v'\in (Z\cap \tilde Z)\cap \neighbor (u)$ that is typical w.r.t.~$C\cap V^{\mathcal Z}$ and typical w.r.t.~$D\cap V^{\mathcal Z}$. 
 Let $X_{\LPDJ}=C$ and $Y_{\LPDJ}=D$.
Consider the set $\mathcal T^{r'}$ of components of $t_i^j-\{r'\}$ that are incident to $r'$ and does not contain~$r$. 
We embed the elements $t\in \mathcal T^{r'}$ one after the other using Lemma~\ref{lemma_Embedding1} similarly as we did for the elements of $\mathcal T^{r}$.  At each application of Lemma~\ref{lemma_Embedding1} we use 
$P=(X_{\LPDJ}\cap K\cap V^{\mathcal{Z}})\setminus \varphi(\mathcal D_3)$, $P'=P\cap \neighbor(v')$, $Q=(Y_{\LPDJ}\cap K\cap V^{\mathcal{Z}})\setminus \varphi(\mathcal D_3)$, and $Q'=Q\cap \neighbor(v')$. 

Observe that the embedding $\varphi$ satisfies {\bf(I1)}--{\bf(I4)}.
\end{itemize}

\end{proof}

\end{document}